\documentclass[reqno]{amsart}

\usepackage{amsmath}
\usepackage{amsthm}
\usepackage{amsfonts}
\usepackage{amssymb}
\usepackage{enumerate}
\usepackage{graphicx}
\usepackage{todonotes}
\usepackage[nocompress]{cite}
\usepackage{hyperref}
\usepackage{subfigure}
\graphicspath{ {./images/} }

\newcommand{\indicator}[1]{\ensuremath{\mathbf{1}_{\{#1\}}}}
\newcommand{\oindicator}[1]{\ensuremath{\mathbf{1}_{{#1}}}}

\numberwithin{equation}{section}

\DeclareMathOperator{\cov}{Cov}
\DeclareMathOperator{\var}{Var}
\DeclareMathOperator{\corr}{Corr}
\DeclareMathOperator{\tr}{tr}

\DeclareMathOperator{\dist}{dist}
\DeclareMathOperator{\rank}{rank}
\DeclareMathOperator{\Mat}{Mat}
\DeclareMathOperator{\Sparse}{Sparse}
\DeclareMathOperator{\supp}{supp}
\DeclareMathOperator{\Comp}{Comp}
\DeclareMathOperator{\Incomp}{Incomp}
\DeclareMathOperator{\Span}{Span}

\newcommand{\Prob}{\mathbb{P}}
\newcommand{\E}{\mathbb{E}}
\newcommand{\C}{\mathbb{C}}

\renewcommand\Re{\operatorname{Re}}
\renewcommand\Im{\operatorname{Im}}
\newcommand{\eps}{\varepsilon}
\newcommand{\Sp}{\mathbb{S}}

\theoremstyle{plain}
\newtheorem{theorem}{Theorem}[section]

\newtheorem{proposition}[theorem]{Proposition}

\newtheorem{lemma}[theorem]{Lemma}
\newtheorem{corollary}[theorem]{Corollary}

\newtheorem{question}[theorem]{Question}

\theoremstyle{definition}
\newtheorem{definition}[theorem]{Definition}

\newtheorem{remark}[theorem]{Remark}

\newcommand{\NN}{\mathbb{N}}
\newcommand{\ZZ}{\mathbb{Z}}

\newcommand{\RR}{\mathbb{R}}
\newcommand{\CC}{\mathbb{C}}

\newcommand{\vv}{\nu}
\newcommand{\ee}{\varepsilon}

\newcommand{\pp}{\varphi}
\newcommand{\EE}{\mathbb{E}}
\newcommand{\PP}{\mathbb{P}}

\newcommand{\HH}{\mathbb{H}}
\newcommand{\Corr}{\text{Corr}}

\begin{document}
	\title{Spectrum of Heavy-Tailed Elliptic Random Matrices}
	
	\author[A. Campbell]{Andrew Campbell}
	\thanks{A. Campbell has been supported in part by NSF grant ECCS-1610003.}
	\address{Department of Mathematics, University of Colorado at Boulder, Boulder, CO 80309}
	\email{andrew.j.campbell@colorado.edu}
	
	\author[S. O'Rourke]{Sean O'Rourke}
	\thanks{S. O'Rourke has been supported in part by NSF grants DMS-1810500 and ECCS-1610003.}
	\address{Department of Mathematics, University of Colorado at Boulder, Boulder, CO 80309}
	\email{sean.d.orourke@colorado.edu}

	\begin{abstract}
		An elliptic random matrix $X$ is a square matrix whose $(i,j)$-entry $X_{ij}$ is a random variable independent of every other entry except possibly $X_{ji}$.  
		Elliptic random matrices generalize Wigner matrices and non-Hermitian random matrices with independent entries. 
		When the entries of an elliptic random matrix have mean zero and unit variance, the empirical spectral distribution is known to converge to the uniform distribution on the interior of an ellipse determined by the covariance of the mirrored entries. 
		
		We consider elliptic random matrices whose entries fail to have two finite moments. Our main result shows that when the entries of an elliptic random matrix are in the domain of attraction of an $\alpha$-stable random variable, for $0<\alpha<2$, the empirical spectral measure converges, in probability, to a deterministic limit.  
		This generalizes a result of Bordenave, Caputo, and Chafa\"i for heavy-tailed matrices with independent and identically distributed entries. The key elements of the proof are (i) a general bound on the least singular value of elliptic random matrices under no moment assumptions; and (ii) the convergence, in an appropriate sense, of the matrices to a random operator on the Poisson Weighted Infinite Tree.
	\end{abstract}
	
	\maketitle
	
	\section{Introduction}
	
	Let  $\Mat_n(\mathbb{F})$ be the set of $n \times n$ matrices over the field $\mathbb{F}$. For a matrix $A \in \Mat_n(\mathbb{C})$, the singular values of $A$ are the square roots of the eigenvalues of $A A^\ast$, where $A^\ast$ is the conjugate transpose of $A$.  We let $s_n(A) \leq \cdots \leq s_1(A)$ denote the ordered singular values of $A$, and $\lambda_1(A),\dots,\lambda_n(A)\in\CC$ be the eigenvalues of $A$ in no particular order. For a matrix $A\in\Mat_n(\CC)$ we define the empirical spectral measure
	$$\mu_A:=\frac{1}{n}\sum_{i =1}^n\delta_{\lambda_i(A)}$$
	and the empirical singular value measure
	$$\vv_{A}:=\frac{1}{n}\sum_{i =1}^n\delta_{s_i(A)}.$$
	These measures are central objects in random matrix theory, and the goal of this paper is to study the asymptotic behavior of the empirical spectral measure for a class of heavy-tailed elliptic random matrices. %We study the convergence of these measures as the dimension of the matrix goes to infinity.  
	
	Elliptic random matrices can be thought of as interpolating between random matrices whose entries are independent and identically distributed (i.i.d.)\ and Wigner matrices. We now give a precise definition. 
	
	\begin{definition}[Elliptic Random Matrix]
		Let $(\xi_1,\xi_2)$ be a random vector with complex-valued random variable entries, $\zeta$ a complex random variable, and $X_n=(X_{ij})_{i,j=1}^n$ be an $n\times n$ random matrix. $X_n$ is an \textit{elliptic random matrix} if \begin{itemize}
			\item[(i)] $\{X_{ii}: 1\leq i\leq n \}\cup\{(X_{ij},X_{ji}):1\leq i<j\leq n \}$ is a collection of independent random elements.
			\item[(ii)] the pairs $\{(X_{ij},X_{ji})\}_{1\leq i< j\leq n}$ are independent copies of $(\xi_1,\xi_2)$.
			\item[(iii)] the diagonal elements $\{X_{ii}: 1\leq i\leq n \}$ are independent copies of $\zeta$.
		\end{itemize} We refer to $(\xi_1,\xi_2),\zeta$ as the \textit{atom variables} of the matrix $X_n$. 
	\end{definition}
	
	Elliptic random matrices were originally introduced by Girko \cite{girko1986elliptic,girko1995elliptic} in the 1980s, with the name coming from the limit of the empirical spectral measure. When the entries of the matrix have four finite moments, the limiting empirical spectral measures was investigated by Naumov \cite{naumov2012elliptic}.  The general case, when the entries are only assumed to have finite variance, was studied in \cite{TheEllipticLaw}.  
	\begin{theorem}[Elliptic law for real random matrices, Theorem 1.5 in \cite{TheEllipticLaw}]\label{EllipticLaw}
		Let $X_n$ be an $n\times n$ elliptic random matrix with real atom variables $(\xi_1,\xi_2),\zeta$. Assume $\xi_1,\xi_2$ have mean zero and unit variance, and $\EE[\xi_1\xi_2]=:\rho$ for $-1<\rho<1$. Additionally assume $\zeta$ has mean zero and finite variance. Then almost surely the empirical spectral measure $\mu_{\frac{1}{\sqrt{n}}X_n}$ of $\frac{1}{\sqrt{n}}X_n$ converges weakly to the uniform probability measure on 
		$$\mathcal{E}_\rho:=\left\{z\in\CC: \frac{\Re(z)^2}{(1+\rho)^2}+\frac{\Im(z)^2}{(1-\rho)^2}\leq 1 \right\}$$
		as $n\rightarrow\infty$.
	\end{theorem}   
	Elliptic random matrices have also been studied in \cite{G_tze_2015,GNT,CLTelliptic,O_Rourke_2015}. 
	
	Our main result, Theorem \ref{EigenConv}, gives an analogous result for heavy-tailed elliptic random matrices, i.e.\ when $\EE|\xi_1|^2$ and $\EE|\xi_2|^2$ are both infinite. In 1994, Cizeau and Bouchaud \cite{PhysRevE.50.1810} introduced L\'evy matrices as a heavy-tailed analogue of the Gaussian Orthogonal Ensemble (GOE). Instead of Gaussian entries, these matrices have entries in the domain of attraction of an $\alpha$-stable random variable, for $0<\alpha<2$. They predicted a deterministic limit $\mu_\alpha$,  which depends only on $\alpha$, for the empirical spectral measures of these matrices when properly scaled. Convergence to a deterministic limit was first proved by Ben Arous and Guionnet \cite{Arous_2007} and later by Bordenave, Caputo, and Chafa\"i \cite{SymHeavy} with an alternative characterization in their study of random Markov matrices. In \cite{HeavyIId} Bordenave, Caputo, and Chafa\"i proved the empirical spectral measure of random matrices with i.i.d.\ entries in the domain of attraction of a complex $\alpha$-stable random variable converges almost surely to an isotropic measure on $\CC$ with unbounded support. Notably for L\'evy matrices the limiting spectral measure inherits the tail behavior of the entries, while the limiting spectral measure of heavy-tailed random matrices with i.i.d.\ entries has a lighter tail and finite moments of every order. 
	
	Much of the work on heavy-tailed random matrices has been on the spectrum of symmetric matrices, either L\'evy or sample covariance matrices \cite{Arous_2007,Auffinger_2009,Auffinger_2016,Belinschi_2009,Biroli_2007,GOEforLevy,Soshnikov_2004,SymHeavy,heiny2020limiting}. Motivated by questions of delocalization from physics there has also been considerable work done in studying the eigenvectors of symmetric heavy-tailed matrices \cite{Bordenave_2013,Bordenave_2017,EigenvectorforLevy,Monthus_2016,PhysRevE.50.1810,Tarquini_2016,benaych-georges2014}. 
	
	As is often the case in random matrix theory most of the work on heavy-tailed random matrices has focused on ensembles where the entries are independent up to symmetry conditions on the matrix. Work on matrices with dependent entries is still limited. Heavy-tailed matrices with normalized rows have been considered for random Markov chains in \cite{SymHeavy,HeavyMarkovOriented} and sample correlation matrices in \cite{heiny2020limiting}. In \cite{basrak2019extreme} extreme eigenvalue statistics of symmetric heavy-tailed random matrices with $m$-dependent entries were studied and shown to converge to a Poisson process. This $m$-dependence can be thought of as a short range dependence meant to model stock returns that depend on stocks from the same sector of size determined by $m$. To the best of our knowledge there are not any results on non-Hermitian heavy-tailed matrices with long range dependence between entries from different rows outside the random reversible Markov chains studied in \cite{SymHeavy}.

	The key question when approaching heavy-tailed elliptic random matrices is how to measure the dependence between $\xi_1$ and $\xi_2$. Without two finite moments the covariance between $\xi_1$ and $\xi_2$, which was the key parameter in Theorem \ref{EllipticLaw}, cannot be defined. Similar notions, such as covariation or codifference, exist for $\alpha$-stable random variables but they do not seem sufficient for our purposes. The difference is that the covariation does not provide as much information for $\alpha$-stable random vectors as the covariance does for Gaussian random vectors. If $X=(X_1,X_2)$ is a bivariate Gaussian random vector where $X_1$ and $X_2$ are standard real Gaussian random variables, then the correlation $\rho=\EE[X_1X_2]$ uniquely determines the distribution of $X$. Thus one approach to measuring dependence is to find a parameter which uniquely determines the distribution of a properly normalized $\alpha$-stable random vector. The distribution of an $\alpha$-stable random vector $Y$ in $\mathbb{R}^n$ is determined uniquely through its characteristic function of the form 
	$$\EE\exp\left(iu^TY \right)=\begin{cases}
		\exp\left(-\int_{\mathbb{S}^{n-1}}|u^Ts|^\alpha(1-i\text{sign}(u^Ts)\tan(\frac{\pi\alpha}{2}))d\theta(s)+iu^Ty \right), \alpha\neq 1\\
		
		\exp\left(-\int_{\mathbb{S}^{n-1}}|u^Ts|(1+i\text{sign}(u^Ts)\log|u^Ts|)d\theta(s)+iu^Ty \right), \alpha=1
	\end{cases}$$ for a finite measure $\theta$ on the unit sphere $\mathbb{S}^{n-1}$ and a deterministic vector $y$. $Y$ can be translated and scaled so that $y=0$ and $\theta$ is a probability measure uniquely determining the distribution of $Y$. $\theta$ is called the spectral measure of $Y$, and it turns out to be the appropriate explicit description of the dependence between the entries of $Y$. The definition of $\theta$ can be extended to random variables which are not stable but rather in the domain of attraction of an $\alpha$-stable random variable, see Definition \ref{def:depmeas}. 
	
	If the components of $Y$ are independent, then $\theta$ is supported entirely on the intersection of the axes and the unit sphere. Intuitively, when considering the mirrored entries of a random matrix, if $\theta$ is close to a measure supported on the intersection of the axes and the unit sphere, the entries are close to independent, after scaling. If $\theta$ is close to a measure supported on the set $\{(z_1,z_2) : |z_1|^2 + |z_2|^2 = 1, z_1=\bar{z}_2\}$ then the matrix is close to Hermitian. Numerical simulations seem to reflect this intuition in the spectrum of elliptic random matrices, see Figures \ref{fig:a}, \ref{fig:b}, and \ref{fig:c}.

	%We begin by defining what is meant by a heavy-tailed elliptic random matrix. Essentially this will be a random matrix with entries in the domain of attraction of an $\alpha$-stable random variable and some dependence between the $ij$ and $ji$ entry. In the finite variance case the Empirical Spectral Distribution (ESD) converges to the uniform distribution on an ellipse whose shape is defined by the parameter $\Corr(X_{ij},X_{ji})$. In the heavy-tailed case the lack of a second moment prevents the correlation from being defined so it is not immediately obvious what parameters should define the dependence. However, if we view correlation as defining the covariance matrix for the Gaussian random vector for which $(X_{ij},X_{ji})$ is in the domain of attraction, the analogues parameter for heavy-tailed random variables is a finite measure on the unit sphere.

	\subsection{Matrix distribution} We will consider elliptic random matrices whose atom variables satisfy the following conditions. 
	
	\begin{definition}[Condition \textbf{C1}]\label{def:depmeas}
		We say the atom variables $(\xi_1,\xi_2),\zeta$ satisfy Condition \textbf{C1} if\begin{itemize}
			\item[(i)]  there exists a positive number $0<\alpha<2$, a sequence $a_n=\ell(n)n^{1/\alpha}$ for a slowly varying function $\ell$ (i.e. $\lim_{t\rightarrow\infty}\ell(tx)/t=1$ for all $x>0$), and a finite measure $\theta_d$ on the unit sphere in $\CC^2$ such that for all Borel subsets $D$ of the unit sphere with $\theta_d(\partial D)=0$ and all $r>0$,
			$$\lim\limits_{n\rightarrow\infty}n\PP\left(\frac{(\xi_1,\xi_2)}{\|(\xi_1,\xi_2) \|}\in D,\|(\xi_1,\xi_2) \|\geq ra_n \right)=\theta_d(D)m_\alpha([r,\infty)),$$
			where $m_\alpha$ is a measure on $(0,\infty)$ with density $f(r)=\alpha r^{-(1+\alpha)}$.  
			\item[(ii)] there exists a constant $C > 0$ such that $\PP(|\zeta|\geq t)\leq Ct^{-\alpha}$ for all $t>0$.
		\end{itemize}
	\end{definition}
	
	We will reserve $\theta_d$ to denote the measure on a sphere associated to the atom variables of a heavy-tailed elliptic random matrix; it may be worth noting that $d$ stands for ``dependence'' and is not a parameter. As it turns out, Condition $\textbf{C1}$ is enough to prove convergence of the empirical singular value distribution, see Theorem \ref{SingValueConv}. We will need more assumptions to prove convergence of the empirical spectral measure. 
	
	\begin{definition}[Condition \textbf{C2}]
		We say $(\xi_1,\xi_2),\zeta$ satisfy Condition \textbf{C2} if the atom variables satisfy Condition \textbf{C1} and if
		%Assume the atom variables $(\xi_1,\xi_2)\in\CC^2,\zeta\in\CC$ satisfy Condition \textbf{C1}, then we say $(\xi_1,\xi_2),\zeta$ satisfy Condition \textbf{C2} if 
		\begin{itemize}
			\item[(i)] there exists no $(a,b)\in\CC^2\setminus\{(0,0)\}$ such that 
			$$\supp(\theta_d)\subseteq\{(z,w)\in\CC^2:az+bw=0, |z|^2+|w|^2=1\}.$$
			
			\item[(ii)] $a_n/n^{1/\alpha}\rightarrow c$, for some constant $c>0$ as $n\rightarrow\infty$.
		\end{itemize}
	\end{definition}

	\subsection{Main results} For simplicity, let $A_n:=\frac{1}{a_n}X_n$. Our first result gives the convergence of the singular values of $A_n-zI_n$, which we will denote as $A_n-z$, for $z\in\CC$.  Here and throughout, $I_n$ denotes the $n \times n$ identity matrix.  While interesting on its own, this is the first step in the method of Hermitization to establish convergence of the empirical spectral measure. Throughout we will use $\Rightarrow$ to denote weak convergence of probability measures and convergence in distribution of random variables.
	
	\begin{theorem}[Singular values of heavy-tailed elliptic random matrices]\label{SingValueConv}
		Let $X_n$ be an $n\times n$ elliptic random matrix with atom variables $(\xi_1,\xi_2),\zeta$ which satisfy Condition \textbf{C1}. Then for each $z\in\CC$ there exists a deterministic probability measure $\nu_{z,\alpha,\theta_d}$, depending only on $z,\alpha$ and $\theta_d$, such that almost surely
		$$\nu_{A_n-zI_n}\Rightarrow\nu_{z,\alpha,\theta_d}$$
		as $n\rightarrow\infty$.
	\end{theorem}
	
	Under Condition \textbf{C2} we prove the convergence of the empirical spectral measure of $A_n$.
	
	\begin{theorem}[Eigenvalues of heavy-tailed elliptic random matrices]\label{EigenConv}
		Let $X_n$ be an $n\times n$ elliptic random matrix with atom variables $(\xi_1,\xi_2),\zeta$ which satisfy Condition \textbf{C2}. Then there exists a deterministic probability measure $\mu_{\alpha,\theta_d}$, depending only on $\alpha$ and $\theta_d$, such that 
		$$\mu_{A_n}\Rightarrow\mu_{\alpha,\theta_d}$$
		in probability as $n\rightarrow\infty$. Moreover for any smooth $\pp:\CC\rightarrow\CC$ with compact support \begin{equation}\label{eq:HermEquation}
			\int_\CC\pp \ d\mu_{\alpha,\theta_d}=\frac{1}{2\pi}\int_\CC\Delta\pp(z)\left[\int_{0}^\infty\log(t)\ d\nu_{\alpha,z,\theta_d}\right]dz,
		\end{equation} where $\nu_{\alpha,z,\theta_d}$ is as in Theorem \ref{SingValueConv} and $dz$ is the Lebesgue measure on $\CC$.
	\end{theorem}

	A distributional equation describing the Stieltjes transform of $\nu_{\alpha,z,\theta_d}$ is given in Proposition \ref{prop:RDE}, which, when combined with \eqref{eq:HermEquation},  gives a description of $\mu_{\alpha,\theta_d}$.
	
	\begin{remark}
		If $\theta_d=1/2(\theta_1+\theta_2)$ where $\theta_1$ and $\theta_2$ are probability measures with $\supp(\theta_i)\subseteq \{(z_1,z_2)\in\mathbb{S}: z_i=0 \}$ and $\theta_1(A)=\theta_2\left(\{(z_1,z_2):(z_2,z_1)\in A\}\right)$, then $\nu_{\alpha,z,\theta_d}$ in Theorem \ref{SingValueConv} and $\mu_{\alpha,\theta_d}$ in Theorem \ref{EigenConv} are the same measures as in the main results of Bordenave, Caputo, and Chafa\"i \cite{HeavyIId}. This can be seen by computing $\theta_d$ when $\xi_1$ and $\xi_2$ are independent and identically distributed. It is also worth noting that if the matrix $X_n$ is complex Hermitian but not real symmetric, then $(\xi_1,\xi_2)$ will satisfy Condition \textbf{C2} (i), and thus Theorem \ref{EigenConv} holds.
	\end{remark}
	
	%Something about why $\mu_{\alpha,\theta_d}$ is so hard to analyze or maybe some characterization of the resolvent
	Numerical simulations seem to give weight to the reasoning that the support of $\theta_d$ determines how close $\mu_{\alpha,\theta_d}$ is to being isotropic or supported on the real line. In Figure \ref{fig:c} we see as $\supp(\theta_d)$ moves further from $\{z_1=z_2\}$ the mass of the spectrum moves further from the real line. A similar phenomenon appears in Figure \ref{fig:b}: as $\supp(\theta_d)$ moves further from the intersection of the axes and the sphere the spectrum becomes further from isotropic. We also see in Figure \ref{fig:a} that the tail behavior of the spectrum appears to depend on $\theta_d$ and may vary in different directions. 
	
	Spectrum similar to Figure \ref{fig:a} can be found in \cite{HeavyLevyOperators} where the authors study the spectral measure of $C_1+iC_2$ where $C_1$ and $C_2$ are free random L\'evy elements. The authors use tools in free probability to show the spectrum is supported inside the ``hyperbolic cross", which appears to be the case for heavy-tailed elliptic random matrices for certain $\theta_d$. The limiting spectral measure in \cite{HeavyIId} is not contained in a ``hyperbolic cross", which suggests $C_1+iC_2$ is not the heavy-tailed analogue of a circular element, in contrast with the case when $C_1$ and $C_2$ are free semicircular elements and $C_1+iC_2$ is a circular element.

	\begin{figure}
		\centering
		\subfigure{{\includegraphics[width=6cm]{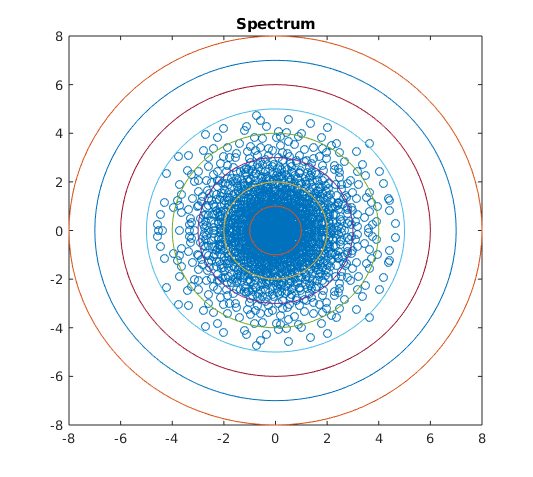} }}%
		\subfigure{{\includegraphics[width=6cm]{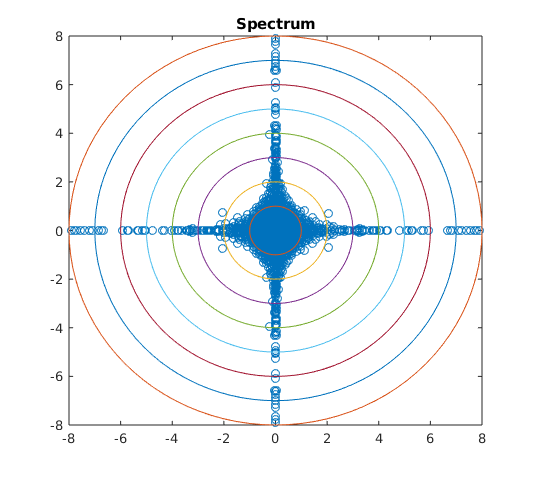} }}%
		\caption{The plot on the left is the spectrum of an $n\times n$ matrix $n^{-1/\alpha}X$ where $\alpha=1.25$, $n=2000$, and the entries of $X$ are i.i.d.\ random variables distributed as $\ee U^{-1/\alpha}$, where $\ee$ is uniformly distributed on $\{-1,1\}$ and $U$ is uniformly distributed on $[0,1]$. The plot on the right is the spectrum of an $n\times n$ elliptic random matrix $n^{-1/\alpha}X$ where $X$ has atom variables $(U^{-1/\alpha}\cos(w),U^{-1/\alpha}\sin(w)),1$ with $\alpha=1.25$, $n=2000$, $U$ uniformly distributed on $[0,1]$, and $w$ uniformly distributed on $[0,2\pi]$. The plot window on the right is cropped to avoid extreme values. }%
		\label{fig:a}%
	\end{figure}     
	
	\subsection{Further questions}
	
	Since our results capture both heavy-tailed Hermitian matrices and heavy-tailed matrices with i.i.d.\ entries, $\mu_{\alpha,\theta_d}$ does actually depend on $\theta_d$. However, as can be seen from \cite{HeavyIId}, $\mu_{\alpha,\theta_d}$ does not depend on every aspect of $\theta_d$. 
	
	\begin{question}\label{ESDuniQ}
		What properties of $\theta_d$ determine $\mu_{\alpha,\theta_d}$?
	\end{question}
	
	In the case when $X_n$ has i.i.d.\ entries the limiting empirical spectral measure has an exponential tail \cite{HeavyIId} while in the case when $X_n$ is Hermitian it has the same tail behavior as the entries \cite{Arous_2007,SymHeavy}. This leads us to ask how the tail of $\mu_{\alpha,\theta_d}$ depends on $\theta_d$.
	
	\begin{question}
		How does the tail behavior of $\mu_{\alpha,\theta_d}$ vary with respect to $\theta_d$? 
	\end{question}

	\begin{figure}
		\centering
		\subfigure{{\includegraphics[width=6cm]{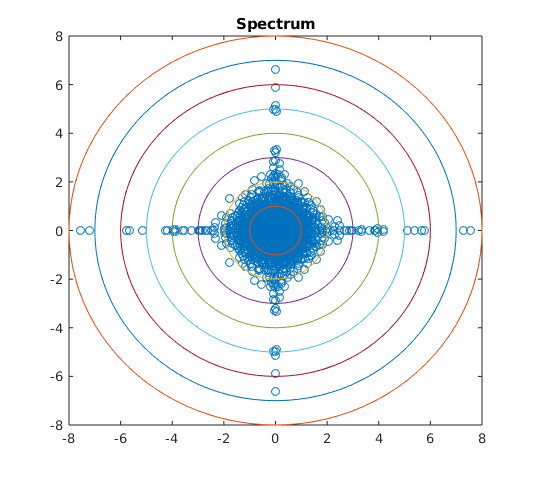} }}%
		\subfigure{{\includegraphics[width=6cm]{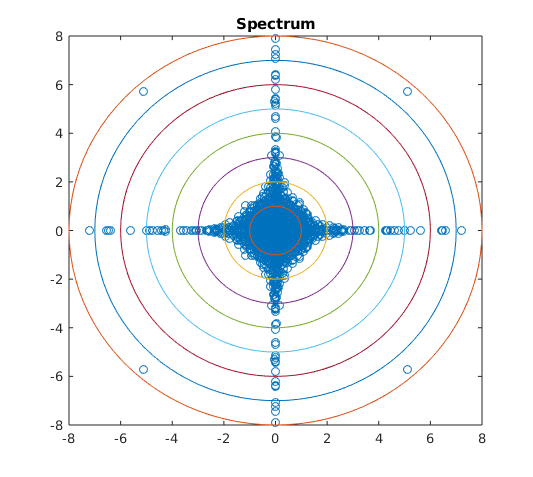} }}%
		\caption{Both plots show the spectrum of an $n\times n$ elliptic random matrix $n^{-1/\alpha}X$ where $X$ has atom variables $(U^{-1/\alpha}\cos(w),U^{-1/\alpha}\sin(w)),1$ with $\alpha=1.25$, $n=2000$, $U$ uniformly distributed on $[0,1]$, and $w$ uniformly distributed on $\{0,\pi/2,\pi,3\pi/2 \}+[-b\pi/4,b\pi/4]$. For the plot on the left $b=0.1$, while for the plot on the right $b=0.5$. Both plot windows are trimmed to avoid extreme values. }%
		\label{fig:b}%
	\end{figure}
	
	\begin{figure}
		\centering
		\subfigure{{\includegraphics[width=6cm]{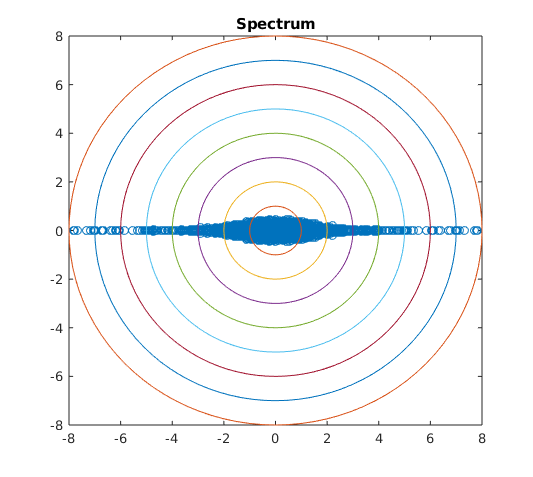} }}%
		\subfigure{{\includegraphics[width=6cm]{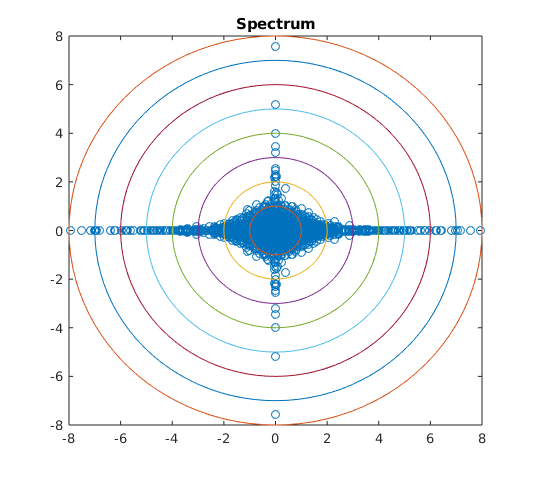} }}%
		\caption{Both plots show the spectrum of an $n\times n$ elliptic random matrix $n^{-1/\alpha}X$ where $X$ has atom variables $(U^{-1/\alpha}\cos(w),U^{-1/\alpha}\sin(w)),1$ with $\alpha=1.25$, $n=2000$, $U$ uniformly distributed on $[0,1]$, and $w$ uniformly distributed on $\{\pi/4,5\pi/4 \}+[-b\pi/4,b\pi/4]$. For the plot on the left $b=1$, while for the plot on the right $b=4/3$. Both plot windows are trimmed to avoid extreme values.}%
		\label{fig:c}%
	\end{figure}

	\subsection{Outline}As expected with the empirical spectral measure of non-Hermitian random matrices, we make use of Girko's Hermitization method. However, instead of considering the logarithmic potential directly, we follow the approach of Bordenave, Caputo, and Chafa\"i \cite{SymHeavy,HeavyIId} by using the objective method of Aldous and Steele \cite{ObjMethod} to get convergence of the matrices to an operator on Aldous' Poisson Weighted Infinite Tree (PWIT). This objective method approach was expanded by Jung \cite{Jung16} to symmetric light-tailed random matrices, adjacency matrices of sparse random graphs, and symmetric random matrices whose column entries are some combination of heavy-tailed and light-tailed random variables.
	
	In Sections \ref{sect:PPP} and \ref{sect:bipartres} we give a collection of results and background for approaching the proof of Theorem \ref{SingValueConv}. In Section \ref{sect:OpConv} we define the PWIT, establish the local weak convergence of the matrices $A_n$ to an operator associated with the PWIT, and give a proof of Theorem \ref{SingValueConv}. In Section \ref{sect:LeastSingValue} we give a very general bound on the least singular value of elliptic random matrices. In Section \ref{sect:SingVal} we establish the uniform integrability of $\log(\cdot)$ against $\nu_{A_n-zI_n}$ and complete the proof of Theorem \ref{EigenConv}. The appendix contains some auxiliary results.  
	
	We conclude this section by establishing notation, giving a brief description of Hermitization, and stating some properties of $\xi_1$ and $\xi_2$ implied by Condition \textbf{C2}.

	\subsection{Notation}
	We now establish notation we will use throughout.
	
	Let $[n] := \{1, \ldots, n\}$ denote the discrete interval.  For a vector $v = (v_i)_{i=1}^n \in \C^n$ and a subset $I \subset [n]$, we let $v_I := (v_i)_{i \in I} \in \C^{I}$.  Similarly, for an $m \times n$ matrix $A = (A_{ij})_{i \in [m], j \in [n]}$ and $I \subset [m], J \subset [n]$, we define $A_{I \times J} := (A_{ij})_{i \in I, j \in J}$. For a countable set $S$ we will let $I_S$ denote the identity on $\ell^2(S)$. In the case when $S=[n], \ell^2(S)=\CC^n$ we will simply write $I_n$ or $I$ if the dimension is clear. 
	
	For a vector $x \in \C^n$, $\|x \|$ is the Euclidean norm of $x$.  For a matrix $A$, $A^{\mathrm{T}}$ is the transpose of $A$ and $A^\ast$ is the conjugate transpose of $A$.  In addition, $\|A\|=s_1(A)$ is the spectral norm of $A$ and $\|A\|_2$ is the Hilbert--Schmidt  norm of $A$ defined by the formula
	\[ \|A\|_2 = \sqrt{ \tr (A A^\ast)}. \]
	For a linear, but not necessarily bounded, operator $A$ on a Hilbert space we let $\mathcal{D}(A)$ denote the domain of $A$.  We will often use the shorthand $A+z$ to denote $A+zI$ where $I$ is the identity operator. 
	
	For two complex-valued square integrable random variables $\xi$ and $\psi$, we define the correlation between $\xi$ and $\psi$ as
	\[ \corr(\xi, \psi) := \frac{\cov(\xi, \psi)}{\sqrt{ \var(\xi) \var(\psi) } }, \]
	where $\cov(\xi, \psi):=\EE[(\xi-\EE\xi)\overline{(\psi-\EE\psi)}]$ is the covariance between $\xi$ and $\psi$, and $\var(\xi)=\EE|\xi-\EE\xi|^2$ is the variance of $\xi$. For two random elements $X$ and $Y$ we say $X\overset{d}{=}Y$ if $X$ and $Y$ have the same distribution.  We will also say a positive random variable $Y$ stochastically dominates a positive random variable $Z$ if for all $x>0$
	$$\PP(Y\geq x)\geq\PP(Z\geq x).$$

	For a topological space $E$, $\mathcal{B}(E)$ will always denote the Borel $\sigma$-algebra of $E$. $\RR_+$ will denote the positive real numbers.
	
	Throughout this paper we will use asymptotic notation ($O, o, \Theta$, etc.) under the assumption that $n\rightarrow\infty$. $X=O(Y)$ if $X\leq CY$ for an absolute constant $C>0$ and all $n \geq C$, $X=o(Y)$ if $X\leq C_nY$ for $C_n\rightarrow 0$, $X=\Theta(Y)$ if $cY\leq X\leq CY$ for absolute constants $C,c>0$ and all $n \geq C$, and $X\sim Y$ if $X/Y\rightarrow1$.
	
	\subsection{Hermitization}
	Let $\mathcal{P}(\CC)$ be the set of probability measures on $\CC$ which integrate $\log|\cdot|$ in a neighborhood of infinity. For every $\mu\in\mathcal{P}(\CC)$ the logarithmic potential $U_\mu$ of $\mu$ on $\CC$ is a function $U_\mu:\CC\rightarrow[-\infty,\infty)$ defined for every $z\in\CC$ by 
	$$U_\mu(z)=\int_\CC\log|z-w|d\mu(w).$$
	
	In $\mathcal{D'}(\CC)$ one has 
	$\Delta U_\mu=2\pi\mu$, where $\mathcal{D'}(\CC)$ is the set of Schwartz-Sobolev distributions on $\CC$ endowed with its usual convergence with respect to all infinitely differentiable functions with compact support.
	\begin{lemma}[Lemma A.1 in \cite{HeavyIId}]\label{Unicity}
		For every $\mu,\nu\in\mathcal{P}(\CC)$, if $U_\mu=U_\nu$ a.e. then $\mu=\nu$.
	\end{lemma}
	
	To see the connection between logarithmic potentials and random matrices consider an $n\times n$ random matrix $A$. If $P(z)=\det(A-zI_n)$ is the characteristic polynomial of $A$, then \begin{equation}
		U_{\mu_A}(z)=\int_\CC\log|z-w|d\mu_A(w)=\frac{1}{n}\log|P_A(z)|=\int_{0}^\infty \log(t)d\nu_{A-z}(t).
	\end{equation} Thus, through the logarithmic potential we can move from a question about eigenvalues of $A$ to singular values of $A-z$. We refer the reader to \cite{BC} for more on the logarithmic potential in random matrix theory. One immediate issue is that $\log(\cdot)$ is not a bounded function on $\RR_+$, and thus we need more control on the integral of $\log(\cdot)$ with respect to $\{\nu_{A_n-z}\}_{n\geq 1}$.

	\begin{definition}[Uniform integrability almost surely and in probability]
		Let $\{\mu_n\}_{n=1}^\infty$ be a sequence of random probability measures on a measurable space $(T,\mathcal{T})$. We say a measurable function $f:T\rightarrow\RR$ is uniformly integrable almost surely with respect to $\{\mu_n\}_{n=1}^\infty$ if 
		$$\lim\limits_{t\rightarrow\infty}\limsup\limits_{n\rightarrow \infty}\int_{|f|>t}|f|d\mu_n=0,$$
		with probability one. We say a measurable function $f:T\rightarrow\RR$ is uniformly integrable in probability with respect to $\{\mu_n\}_{n=1}^\infty$ if for every $\ee>0$
		$$\lim\limits_{t\rightarrow\infty}\limsup\limits_{n\rightarrow \infty}\PP\left(\int_{|f|>t}|f|d\mu_n>\ee\right)=0.$$   
	\end{definition}
	
	\begin{lemma}[Lemma 4.3 in \cite{BC}]\label{Girko}
		Let $(A_n)_{n\geq 1}$ be a sequence of complex random matrices where $A_n$ is $n\times n$ for every $n\geq 1$. Suppose for Lebesgue almost all $z\in\CC$, there exists a probability measure $\nu_z$ on $[0,\infty)$ such that \begin{itemize}
			\item a.s. $(\nu_{A_n-z})_{n\geq 1}$ tends weakly to $\nu_z$
			\item a.s. (resp. in probability) $\log(\cdot)$ is uniformly integrable for $(\nu_{A_n-z})_{n\geq 1}$.
		\end{itemize}
		Then there exists a probability measure $\mu\in\mathcal{P}(\CC)$ such that\begin{itemize}
			\item a.s. (resp. in probability) $(\mu_{A_n})_{n\geq 1}$ converges weakly to $\mu$
			\item for almost every $z\in\CC$,
			$$U_\mu(z)=\int_{0}^\infty\log(t)d\nu_z(t).$$
		\end{itemize}
	\end{lemma}

	\subsection{Individual entries and stable random vectors. }  We now state some useful properties of $\xi_1$ and $\xi_2$ implied by Condition \textbf{C2}. First, if $(\xi_1^{(1)},\xi_2^{(1)}),(\xi_1^{(2)},\xi_2^{(2)})\dots$ are i.i.d\ copies of $(\xi_1,\xi_2)$, then Condition \textbf{C1} (i) guarantees, see \cite{sakovich1956single}, there exists a sequence $b_n$ such that
	$$\frac{1}{a_n}\sum_{i=1}^{n}(\xi_1^{(i)},\xi_2^{(i)})-b_n\Rightarrow Z=(Z_1,Z_2),$$  
	for some $\alpha$-stable random vector $Z$ with spectral measure $\theta_d$. Condition \textbf{C2} (i) guaranties neither $Z_1$ nor $Z_2$ is identically $0$. We need the following theorem to get results on stable random vectors.
	
	\begin{theorem}[Theorem 2.3.9 in \cite{StableProcesses}]
		Let $(X_1,\dots,X_d)$ be an  $\alpha$-stable vector in $\RR^d$. Then $(X_1,\dots,X_k)$ is an $\alpha$-stable random vector for any $k\leq d$. 
	\end{theorem}
	
	From this, with $k=1$ for real $\xi_1$ and $k=2$ for complex $\xi_1$, we get that $\xi_1$ is in the domain of attraction of an $\alpha$-stable random variable and satisfies \begin{equation}\label{eq:indivassump}
		\PP(|\xi_1|\geq t)=L(t)t^{-\alpha}
	\end{equation} for some slowly varying function $L$.  In addition, 
	\begin{equation}
		\lim\limits_{t\rightarrow\infty}\PP\left(\frac{\xi_1}{|\xi_1|}\in \cdot\Bigg| |\xi_1|\geq t \right)=\theta_1(\cdot)
	\end{equation} for some probability measure $\theta_1$, again see \cite{sakovich1956single}. The same holds for $\xi_2$ with a possibly different probability measure $\theta_2$. Also note\begin{equation}\label{eq:momentbound}
		\EE[|\xi_i|^p]<\infty
	\end{equation} for all $0\leq p<\alpha$ and $i=1,2$.
	
	\section*{Acknowledgment}
	The second author thanks Djalil Chafa\"{i} for answering his questions concerning \cite[Appendix A]{BC}. 
	
	%\subsection{A Useful Model}
	
	%Before moving on to the proofs of Theorem \ref{SingValueConv} and Theorem \ref{EigenConv} we describe a useful model to keep in mind. Let $\{X_{ij}\}_{1\leq i\leq j\leq n}$ be i.i.d. random variables satisfying (H1) and (H2). Let $\{Y_{ij}\}_{1\leq i\leq j\leq n}$ be i.i.d. copies of the $X_{ij}$ and let $\{\ee_{ij} \}_{1\leq i\leq j\leq n}$ be i.i.d. Bernoulli random variables with $\PP(\ee_{12}=1)=p$. For $j>i$ let $X_{ji}=\ee_{ij}X_{ij}+(1-\ee_{ij})Y_{ij}$.
	
	%If $X_n=(X_{ij})_{i,j=1}^n$ then $\{X_n\}_{n\geq 1}$ satisfies conditions \textbf{(C0)} and \textbf{(C1)}. In fact condition \textbf{(C1)} is meant to describe matrices which are asymptotically similar to this model. While it may not technically meet the conditions the model where $X_{ji}=\ee_{ij}e^{iw_T}X_{ij}+(1-\ee_{ij})Y_{ij}$ can be shown to converge with pretty simple adaptations to the work below. This model could give a way of describing properties of the limiting distributions after the issue of universality is dealt with.  

	\section{Poisson point processes and stable distributions}\label{sect:PPP} In this section we give a brief review of Poisson Point Processes (p.p.p.)\ and their relation to the order statistics of random variables in the domain of attraction of an $\alpha$-stable distribution. See \cite{EmpPP, EV/RV/PP, HeavyPhenom},  and the references therein for proofs. 
	
	\subsection{Simple point processes}
	
	Throughout this section we will assume $E=\bar{\RR}^n\setminus\{0\}$ with the relative topology, where $\bar{\RR}^n$ is the one point compactification of $\RR^n$, but many of the results can be extended to other topological spaces. 
	
	Denote by $\mathcal{M}(E)$ the set of simple point Radon measures 
	$$\mu=\sum_{x\in D}\delta_x$$
	where $D$ is such that $D\cap(B_r(0))^c$ is a finite set for any $r>0$ and $B_r(0)$ is the ball of radius $r$ around the point $0$. Denote by $\mathcal{H}(E)$ the set of supports corresponding to measures in $\mathcal{M}(E)$. The elements of $\mathcal{H}(E)$ are called configurations. 
	
	Let $C_K(E)$ denote the set of real-valued continuous functions on $E$ with compact support. The vague topology on $\mathcal{M}(E)$ is the topology where a sequence $\mu_n$ converges to $\mu$ if for any $f\in C_K(E)$
	$$\int_E fd\mu_n\rightarrow\int_E fd\mu.$$
	$\mathcal{M}(E)$ with the vague topology is a Polish space, and thus complete and metrizable. 
	
	If one considers the one-to-one function $I:\mathcal{M}(E)\rightarrow\mathcal{H}(E)$ given by $I(\mu)=\supp(\mu)$, then the topology of $\mathcal{M}(E)$ can be pushed forward to $\mathcal{H}(E)$. The vague convergence in $\mathcal{M}(E)$ can be stated in terms of a convergence of the supports. Let $\mu_n\xrightarrow{v}\mu$ and give some labeling $\supp(\mu)=\{x^{(1)},x^{(2)},\dots \}$. Then this vague convergence implies there exists labeling $\supp(\mu_n)=\{x^{(1)}_n,x^{(2)}_n,\dots \}$ such that for all $k$, $x_n^{(k)}\rightarrow x^{(k)}$. This description will be particularly useful for our case.
	
	A simple point process $N$ is a measurable mapping from a probability space $(\Omega,\mathcal{F},\PP)$ to $(\mathcal{M}(E),\mathcal{B}(\mathcal{M}(E)))$, where $\mathcal{B}(\mathcal{M}(E))$ is the Borel $\sigma$-algebra defined by the vague topology. Weak convergence of simple point processes is defined by weak convergence of the measures in the vague topology.

	\subsection{Poisson point process}
	
	\begin{definition}
		Let $m$ be a Borel measure on $E$. A point process $N$ is called a Poisson Point Process (p.p.p.) with intensity measure $m$ if for any pairwise disjoint Borel sets $A_1,\dots, A_n$, the random variables $N( A_1),\dots,N(A_n)$ are independent Poisson random variables with expected values $m(A_1),\dots, m(A_n)$.
	\end{definition}
	
	\begin{remark}
		$N$ is a.s. simple if and only if $m$ is non-atomic.
	\end{remark}
	
	The next proposition makes clear the connection between point processes and stable random variables. See Proposition 3.21 in \cite{EV/RV/PP} for a proof. First, we describe an important point process. Let $\theta$ be a finite measure on the unit sphere in $\RR^n$, and $m_\alpha$ be a measure with density $\alpha r^{-(\alpha+1)}dr$ on $\RR_+$. We let $N_\alpha$ denote the p.p.p.\ on $\RR^n$ with intensity measure $\theta\times m_\alpha$. 
	
	\begin{proposition}\label{OrderStatConv}
		Let $\{\xi_n\}_{n\geq 1}$ be i.i.d.\ $\RR^d$-valued random variables. Assume there exists a finite measure $\theta$ on the unit sphere in $\RR^d$ such that
		$$\lim\limits_{n\rightarrow\infty}n\PP\left(\frac{\xi_n}{\|\xi_n \|}\in D,\|\xi_n \|\geq rb_n \right)=\theta(D)m_\alpha([r,\infty)),$$
		for every $r>0$ and all Borel subsets $D$ of the unit sphere with $\theta(\partial D)=0$, where $b_n=n^{1/\alpha}L(n)$ for $0<\alpha<2$ and a slowly varying function $L$. Then 
		$$\beta_n:=\sum_{i =1}^n\delta_{\xi_i/b_n}\Rightarrow N_\alpha$$
		as $n \to \infty$. 
	\end{proposition}
	
	\begin{remark}
		In \cite{EmpPP}, Davydov and Egorov prove this convergence in $\ell^p$-type topologies, under a smoothness assumption on the $\xi_n$.
	\end{remark}
	
	\subsection{Useful properties of $\mathbf{N_\alpha}$} The following are useful and well known properties of the p.p.p.\ $N_\alpha$. Again, see Davydov and Egorov, \cite{EmpPP}, and the references therein for more information and proofs.
	
	\begin{proposition}\label{PPPconstruction}
		Let $\{\lambda_i \}$ and $\{w_i \}$ be independent i.i.d.\ sequences where $\lambda_1$ has exponential distribution with mean $1$, and $w_1$ is $\frac{1}{\theta(\Sp^{d-1})}\theta$ distributed, with $\theta$ a finite nonzero measure on $\Sp^{d-1}$. Define $\Gamma_i=\lambda_1+\cdots+\lambda_i$. Then, for any $\alpha>0$,
		$$N_\alpha\overset{d}{=}\sum_{i =1}^\infty\delta_{\Gamma_i^{-1/\alpha}(\theta(\Sp^{d-1}))^{1/\alpha}w_i }.$$
	\end{proposition} 
	
	\begin{lemma}\label{Pialphaprop}
		The p.p.p.\ $N_\alpha$ has the following properties. \begin{itemize}
			\item[(1)] Almost surely there are only a finite number of points of $\supp(N_\alpha)$ outside a ball of positive radius centered at the origin.
			\item[(2)] $N_\alpha$ is simple.
			\item[(3)] Almost surely, we can label the points from $\supp(N_\alpha)$ according to the decreasing order of their norms $\supp(N_\alpha)=\{(y_1,y_2,\dots):\|y_1\|>\|y_2\|>\dots \}$.
			\item[(4)] With probability one, for any $p>\alpha$,
			$$\sum_{i=1}^\infty|y_i|^p<\infty.$$  
		\end{itemize}
	\end{lemma}

	\section{Bipartized resolvent matrix}\label{sect:bipartres} In this section we follow the notation of Bordenave, Caputo, and Chafa\"i in \cite{HeavyIId} to define the bipartizations of matrices and operators. 
	
	\subsection{Bipartization of a matrix} For an $n\times n$ complex matrix $A$ we consider a symmetrized version of $\nu_{A-zI}$,
	$$\check{\nu}_{A-z}:=\frac{1}{2n}\sum_{k=1}^{n}(\delta_{s_k(A-z)}+\delta_{-s_k(A-z)}).$$
	
	Let $\CC_+:=\{z\in\CC:\Im(z)>0 \}$ and 
	$$\HH_+:=\left\{U=\begin{pmatrix}
		\eta&z\\
		\bar{z}&\eta
	\end{pmatrix}: \eta\in\CC_+, z\in\CC \right\}\subset\text{Mat}_2(\CC).$$
	For any $z\in\CC,\eta\in\CC_+$ and $1\leq i,j\leq n$ define the following $2\times 2$ matrices
	$$U=U(z,\eta):=\begin{pmatrix}
		\eta&z\\
		\bar{z}&\eta
	\end{pmatrix}\quad\text{and}\quad B_{ij}:=\begin{pmatrix}
		0&A_{ij}\\
		\bar{A}_{ji}&0
	\end{pmatrix}.$$
	
	Define the matrix $B\in\text{Mat}_n(\text{Mat}_2(\CC))\simeq\text{Mat}_{2n}(\CC)$ by $B=(B_{ij})_{1\leq i,j\leq n}$. As an element of $\text{Mat}_{2n}(\CC)$, $B$ is Hermitian. We call $B$ the bipartization of the matrix $A$. We define the resolvent matrix in $\text{Mat}_n(\text{Mat}_2(\CC))$ by 
	$$R(U)=(B-U\otimes I_n)^{-1},$$
	so that for all $i,j$, $R(U)_{ij}\in\text{Mat}_2(\CC)$. For $1\leq k\leq n$ we write,
	$$R(U)_{kk}=\begin{pmatrix}
		a_k(z,\eta)&b_k(z,\eta)\\
		b'_k(z\eta)&c_k(z,\eta)
	\end{pmatrix}. $$
	
	Letting $B(z)=B-U(z,0)\otimes I_{n}$ we have 
	$$R(U)=(B(z)-\eta I_{2n})^{-1}.$$
	
	Let $m_\mu$ denote the Stieltjes transform of a probability measure $\mu$ on $\RR$ defined by
	$$m_\mu(\eta)=\int_\RR\frac{1}{x-\eta}d\mu(x),\qquad\eta\in\CC_+.$$ 
	
	\begin{theorem}[Theorem 2.1 in \cite{HeavyIId}]\label{StRe}
		Let $A\in\Mat_n(\CC)$. Then $\mu_{B(z)}=\check{\nu}_{A-z}$,
		$$m_{\check{\nu}_{A-z}}(\eta)=\frac{1}{2n}\sum_{k=1}^n(a_k(z,\eta)+c_{k}(z,\eta)),$$
		and in $\mathcal{D'}(\CC)$, the set of Schwartz-Sobolev distributions on $\CC$ endowed with its usual convergence with respect to all infinitely differentiable functions with compact support,
		$$\mu_{A}(\cdot)=-\frac{1}{\pi n}\sum_{k=1}^n\partial b_k(\cdot,0).$$
	\end{theorem}
	
	\subsection{Bipartization of an operator} Let $V$ be a countable set and let $\ell^2(V)$ denote the Hilbert space defined by the inner product 
	$$\langle\phi,\psi\rangle:=\sum_{u\in V}\bar{\phi}_u\psi_u,\quad \phi_u=\langle\delta_u,\phi\rangle,$$
	where $\delta_u$ is the unit vector supported on $u\in V$. Let $\mathcal{D}(V)$ denote the dense subset of $\ell^2(V)$ of vectors with finite support. Let $(w_{uv})_{u,v\in V}$ be a collection of complex numbers such that for all $u\in V$,
	$$\sum_{v\in V}|w_{uv}|^2+|w_{vu}|^2<\infty.$$
	We then define a linear operator $A$ with domain $\mathcal{D}(V)$ by \begin{equation}\label{eq:OP}
		\langle \delta_u,A\delta_v\rangle=w_{uv}.
	\end{equation}

	Let $\hat{V}$ be a set in bijection with $V$, and let $\hat{v}\in\hat{V}$ be the image of $v\in V$ under the bijection. Let $V^b=V\cup\hat{V}$, and define the bipartization of $A$ as the symmetric operator $B$ on $\mathcal{D}(V^b)$ by 
	$$\langle\delta_u,B\delta_{\hat{v}}\rangle=\overline{\langle\delta_{\hat{v}},B\delta_u\rangle}=w_{uv},$$
	$$\langle\delta_u,B\delta_v\rangle=\langle\delta_{\hat{u}},B\delta_{\hat{v}}\rangle=0.$$
	
	Let $\Pi_{u}:\ell^2(V^b)\rightarrow\Span\{\delta_u,\delta_{\hat{u}} \}$ denote the orthogonal projection onto the span of $\delta_u,\delta_{\hat{u}}$. $\Span\{\delta_u,\delta_{\hat{u}} \}$ is isomorphic to $\CC^2$ under the map $\delta_u\mapsto e_1,\delta_{\hat{u}}\mapsto e_2$. Under this isomorphism we may think of $\Pi_uB\Pi_v^*$ as a linear map from $\CC^2$ to $\CC^2$ with matrix representation   
	$$\Pi_uB\Pi_v^*=\begin{pmatrix}
		0&w_{uv}\\
		\bar{w}_{vu}&0\end{pmatrix}.$$
	Let $B(z)=B-U(z,0)\otimes I_V$. For simplicity we will denote by $B(z)$ the closure of $B(z)$. Recall the sum of an essentially self-adjoint operator and a bounded self-adjoint operator is essentially self-adjoint, thus if $B$ is (essentially) self-adjoint then $B(z)$ is (essentially) self-adjoint. For $\eta\in\CC_+$, $U(z,\eta)\in\HH_+$ we define the resolvent operator 
	$$R(U):=(B(z)-\eta I_{V^b})^{-1},$$
	and \begin{equation}\label{eq:ResFuncts}
		R(U)_{vv}=\Pi_vR(U)\Pi_v^*=\begin{pmatrix}
			a_v(z,\eta)&b_v(z,\eta)\\
			b_v'(z,\eta)&c_v(z,\eta)
		\end{pmatrix}.
	\end{equation}
	\begin{lemma}\label{BoundedRes}
		If $a_v,b_v,c_v,b_v'$ are defined by \eqref{eq:ResFuncts}, then \begin{itemize}
			\item for each $z\in\CC$, $a_v(z,\cdot),c_v(z,\cdot):\CC_+\rightarrow\CC_+$,
			
			\item the functions $a_v(z,\cdot),b_v(z,\cdot),b_v'(z,\cdot),c_v(z,\cdot)$ are analytic on $\CC_+$,
			
			\item and	$$|a_v|\leq(\Im(\eta))^{-1},\quad|c_v|\leq(\Im(\eta))^{-1},\quad|b_v|\leq(2\Im(\eta))^{-1},\quad\text{and}\quad|b_v'|\leq(2\Im(\eta))^{-1}.$$
		\end{itemize} Moreover, if $\eta\in i\RR_+$, then $a_v,c_v$ are pure imaginary and $b_v'=\bar{b}_v$.
	\end{lemma}
	
	See Reed and Simon \cite{MathPhys1} for a proof of the first two statements and \cite{HeavyIId} for the last.
	
	\section{Convergence to the Poisson Weighted Infinite Tree}\label{sect:OpConv}
	
	\subsection{Operators on a tree} Consider a tree $T=(V,E)$ on a vertex set $V$ with edge set $E$. We say $u\sim v$ if $\{u,v\}\in E$. Assume if $\{u,v\}\notin E$ then $w_{uv}=w_{vu}=0$, in particular $w_{vv}=0$ for all $v\in V$. We consider the operator $A$ defined by (\ref{eq:OP}). We begin with useful sufficient conditions for a symmetric linear operator to be essentially self-adjoint, which will be very important for our use. 
	
	\begin{lemma}[Lemma A.3 in \cite{SymHeavy}]\label{SelfAdjointCrit}
		Let $\kappa>0$ and $T=(V,E)$ be a tree. Assume that for all $u,v\in V$, $w_{uv}=\bar{w}_{vu}$ and if $\{u,v\}\notin E$ then $w_{uv}=w_{vu}=0$. Assume that there exists a sequence of connected finite subsets $(S_n)_{n\geq 1}$ of $V$, such that $S_n\subset S_{n+1}$, $\bigcup_n S_n=V$, and for every $n$ and $v\in S_n$,
		$$\sum_{u\notin S_n:u\sim v} |w_{uv}|^2\leq\kappa.$$
		Then $A$ is essentially self-adjoint.
	\end{lemma}
	
	\begin{corollary}[Corollary 2.4 in \cite{HeavyIId}]\label{SelfAdjointCrit2}
		Let $\kappa>0$ and $T=(V,E)$ be a tree. Assume that if $\{u,v\}\notin E$ then $w_{uv}=w_{vu}=0$. Assume there exists a sequence of connected finite subsets $(S_n)_{n\geq 1}$ of $V$, such that $S_n\subset S_{n+1}$, $\bigcup_n S_n=V$, and for every $n$ and $v\in S_n$,
		$$\sum_{u\notin S_n:u\sim v} (|w_{uv}|^2+|w_{vu}|^2)\leq\kappa.$$
		Then for all $z\in\CC$, $B(z)$ is self-adjoint.
	\end{corollary}
	
	The advantages of defining an operator on a tree also include the following recursive formula for the resolvents, see Lemma 2.5 in \cite{HeavyIId} for a proof.
	
	\begin{proposition}\label{prop:RecursiveEquation}
		Assume $B$ is sellf-adjoint and let $U=U(z,\eta)\in\HH_+$. Then \begin{equation}\label{eq:RecursiveEquation}
			R(U)_{\varnothing\varnothing}=-\left(U+\sum_{v\sim\varnothing}\begin{pmatrix}
				0&w_{\varnothing v}\\
				\bar{w}_{v\varnothing}&0
			\end{pmatrix}\tilde{R}(U)_{vv}\begin{pmatrix}
			0&w_{v\varnothing }\\
			\bar{w}_{\varnothing v}&0
		\end{pmatrix} \right)^{-1},
		\end{equation} where $\tilde{R}(U)_{kk}:=\Pi_v R_{B_v}\Pi_v^*$ where $B_v$ is the bipartization of the operator $A$ restricted to $\ell^2(V_v)$, and $V_v$ is the subtree of $V$ of vertices whose path to $\varnothing$ contains $v$.
	\end{proposition}
	
	\subsection{Local operator convergence.} We now define a useful type of convergence. 
	
	\begin{definition}[Local Convergence]
		Suppose $(A_n)$ is a sequence of bounded operators on $\ell^2(V)$ and $A$ is a linear operator on $\ell^2(V)$ with domain $\mathcal{D}(A)\supset\mathcal{D}(V)$. For any $u,v\in V$ we say that $(A_n,u)$ converges locally to $(A,v)$, and write 
		$$(A_n,u)\rightarrow(A,v),$$
		if there exists a sequence of bijections $\sigma_n:V\rightarrow V$ such that $\sigma_n(v)=u$ and, for all $\phi\in\mathcal{D}(V)$,
		$$\sigma_n^{-1}A_n\sigma_n\phi\rightarrow A\phi,$$
		in $\ell^2(V)$, as $n\rightarrow\infty$. 
	\end{definition} 
	
	Here we use $\sigma_n$ for the bijection on $V$ and the corresponding linear isometry defined in the obvious way. This notion of convergence is useful to random matrices for two reasons. First, we will make a choice on how to define the action of an $n\times n$ matrix on $\ell^2(V)$, and the bijections $\sigma_n$ help ensure the choice of location for the support of the matrix does not matter. Second, local convergence also gives convergence of the resolvent operator at the distinguished points $u,v\in V$. This comes down to the fact that local convergence is strong operator convergence, up to the isometries. See \cite{HeavyIId} for details.
	
	\begin{theorem}[Theorem 2.7 in \cite{HeavyIId}]\label{LCtoRes}
		Assume $(A_n,u)\rightarrow(A,v)$ for some $u,v\in V$. Let $B_n$ be the self-adjoint bipartized operator of $A_n$. If the bipartized operator $B$ of $A$ is self-adjoint and $\mathcal{D}(V^b)$ is a core for $B$ (i.e., the closure $B$ restricted to $\mathcal{D}(V^b)$ is $B$), then for all $U\in\HH_+,$
		$$R_{B_n}(U)_{uu}\rightarrow R_B(U)_{vv}.$$
	\end{theorem}
	
	To apply this to random operators we say that $(A_n,u)\rightarrow(A,v)$ in distribution if there exists a sequence of random bijections $\sigma_n$ such that $\sigma_n^{-1}A_n\sigma_n\phi\rightarrow A\phi$ in distribution for every $\phi\in\mathcal{D}(V^b)$. 
	
	\subsection{Poisson Weighted Infinite Tree (PWIT)} Let $\rho$ be a positive Radon measure on $\CC^n\setminus\{0\}$ such that $\rho(\CC^n\setminus\{0\})=\infty$. $\mathbf{PWIT}(\rho)$ is the random weighted rooted tree defined as follows. The vertex set of the tree is identified with $\NN^f:=\bigcup_{k\in\NN\cup\{0\}}\NN^k$ by indexing the root as $\NN^0=\varnothing$, the offspring of the root as $\NN$ and, more generally, the offspring of some $v\in\NN^k$ as $(v1),(v2),\cdots\in\NN^{k+1}$. Define $T$ as the tree on $\NN^f$ with edges between parents and offspring. Let $\{\Xi_v\}_{v\in\NN^f}$ be independent realizations of a Poisson point process with intensity measure $\rho$. Let $\Xi_\varnothing=\{y_1,y_2,\dots\}$ be ordered such that $\|y_1\|\geq\|y_2\|\geq\cdots$, and assign the weight $y_i$ to the edge between $\varnothing$ and $i$, assuming such an ordering is possible. More generally assign the weight $y_{vi}$ to the edge between $v$ and $vi$ where $\Xi_v=\{y_{v1},y_{v2},\dots \}$ where $\|y_{v1}\|\geq\|y_{v2}\|\geq\cdots.$

	Consider a realization of $\mathbf{PWIT}(\theta_d\times m_\alpha)$, with $y_{vk}=(y_{vk}^{(1)},y_{vk}^{(2)})$. Even though the measure $\theta_d\times m_\alpha$ has a more natural representation in polar coordinates, we let $(y_{vk}^{(1)},y_{vk}^{(2)})$ be the Cartesian coordinates of $y_{vk}$.
	Define an operator $A$ on $\mathcal{D}(\NN^f)$ by the formulas \begin{equation}\label{eq:LimitOp}
		\langle\delta_v,A\delta_{vk}\rangle=y_{vk}^{(1)},\quad\text{ and }\quad\langle\delta_{vk},A\delta_v\rangle=y_{vk}^{(2)}
	\end{equation}
	and $\langle\delta_v,A\delta_u\rangle=0$ otherwise. For $0<\alpha<2$, we know by Lemma \ref{Pialphaprop} that the points in $\Xi_v$ are almost surely square summable for every $v\in\NN^f$, and thus $A$ is actually a well defined linear operator on $\mathcal{D}(\NN^f)$, though is possibly unbounded on $\ell^2(\NN^f)$. Before showing the local convergence of the random matrices $A_n$ to $A$ we will show the bipartization of $A$ is self-adjoint.
	
	\begin{proposition}\label{BSelfAdjoint}
		With probability one, for all $z\in\CC$, $B(z)$ is self-adjoint,  where $B(z)$ is the bipartization of the operator $A$ defined by (\ref{eq:LimitOp}).
	\end{proposition}
	We begin with a lemma on point processes, proved in Lemma A.4 from \cite{SymHeavy}, before checking our criterion for self-adjointness. 
	
	\begin{lemma}\label{PoisStopTime}
		Let $\kappa>0, 0<\alpha<2$ and let $0<x_1<x_2<\cdots$ be a Poisson process of intensity $1$ on $\RR_+$. Define $\tau=\inf\{t\in\NN:\sum_{k=t+1}^{\infty}x_{k}^{-2/\alpha}\leq \kappa \}.$ Then $\EE\tau$ is finite and goes to $0$ as $\kappa$ goes to infinity.
	\end{lemma}
	
	\begin{proof}[Proof of Proposition \ref{BSelfAdjoint}] 
		For $\kappa>0$ and $v\in\NN^f$ define 
		$$\tau_v=\inf\left\{t\geq 0: \sum_{k=t+1}^\infty \|y_{vk}\|^2\leq\kappa \right\}.$$
		Note if $N$ is a homogeneous Poisson process on $\RR_+$ with intensity 1 and $f(x)=x^{-1/\alpha}$, then $f(N)$ is Poisson process with intensity measure $\alpha r^{-1-\alpha}dr$. Thus by Lemma \ref{PoisStopTime}, $\kappa>0$ can be chosen such that $\EE\tau_v<1$ for any fixed $v$. Since the random variables $\{\tau_v \}_{v\in\NN^f}$ are i.i.d.\ this $\kappa$ works for all $v$. Fix such a $\kappa$. We now color the vertices red and green in an effort to build the sets $S_n$ in Corollary \ref{SelfAdjointCrit2}. Put a green color on all vertices $v$ such that $\tau_v\geq 1$ and a red color otherwise. Define the sub-forest $T^g$ of $T$ where an edge between $v$ and $vk$ is included if $v$ is green and $1\leq k\leq \tau_v$. If the root $\varnothing$ is red let $S_1=\{\varnothing \}$. Otherwise let $T^g_\varnothing=(V^g_\varnothing,E^g_\varnothing)$ be the subtree of $T^g$ containing $\varnothing$. Let $Z_n$ denote the number of vertices in $T^g_\varnothing$ at a depth $n$ from the root, then $Z_n$ is a Galton-Watson process with offspring distribution $\tau_{\varnothing}$. It is well known that if $\EE\tau_{\varnothing}<1$, then the tree is almost surely finite, see Theorem 5.3.7 in \cite{durrett_2010}. Let $L^g_\varnothing$ be the leaves of the tree $T^g_\varnothing$. Set $S_1:=\bigcup_{v\in L^g_\varnothing}\{vk:1\leq k\leq\tau_v \}\bigcup V^g_\varnothing$. It is clear for any $v\in S_1$,
		$$\sum_{u\notin S_1:u\sim v} (|y_{uv}|^2+|y_{vu}|^2)\leq\kappa.$$
		Now define the outer boundary of $\{\varnothing \}$ as $\partial_\tau\{\varnothing \}=\{1,\dots,\max(\tau_{\varnothing},1) \} $ and for $v=(i_1\cdots i_k)$ set $\partial_\tau\{v\}=\{(i_1\cdots i_{k-1}(i_{k}+1)) \}\cup\{(i_1\cdots i_k1),\dots,(i_1\cdots i_k\max(\tau_v,1))\}$. For a connected set $S$ define its outer boundary as
		$$\partial_\tau S=\left(\bigcup_{v\in S}\partial_\tau\{v\} \right)\setminus S.$$
		Now for each $u_1,\dots,u_k\in \partial_\tau S_1$ apply the above process to get subtrees $\{T^g_{u_i}=(V^g_{u_i},E^g_{u_i})\}_{i=1}^k$ with roots $u_i$ and the leaves of tree $T^g_{u_i}$ denoted by $L^g_{u_i}$. Set
		$$S_2:=S_1\cup\left(\bigcup_{i=1}^k(V^g_{u_i}\cup_{v\in L^g_{u_i}}\{vj:1\leq j\leq \tau_v  \}) \right).$$
		Apply this procedure iteratively to get the sequence of subsets $(S_n)_{n\geq 1}$. Apply Corollary \ref{SelfAdjointCrit2} to complete the proof.    
	\end{proof}
	
	\subsection{Local convergence} For an $n\times n$ matrix $M$ we aim to define $M$ as a bounded operator on $\ell^2(\NN^f)$. For $1\leq i,j,\leq n,$ let $\langle \delta_i,M\delta_j\rangle=M_{ij}$. and $\langle\delta_u,M\delta_v\rangle=0$ otherwise.
	
	\begin{theorem}\label{LocalConv}
		Let $(P_n)_{n\geq 1}$ be a sequence of uniformly distributed random $n\times n$ permutation matrices, independent of $A_n$. Then, in distribution, $(P_nA_nP_n^T,1)\rightarrow(A,\varnothing)$ where $A$ is the operator defined by (\ref{eq:LimitOp}).  
	\end{theorem}
	
	\begin{remark}
		This theorem holds for any sequence of permutation matrices $P_n$ regardless of independence or distribution, but for our use later it is important they are independent of $A_n$ and uniformly distributed. This is to get around the fact the entries of $A_n$ are not exchangeable.
	\end{remark} 
	
	The rest of this subsection is devoted to the proof of Theorem \ref{LocalConv}. The procedure follows along the same lines as Bordenave, Caputo, and Chafa\"i in \cite{SymHeavy} and \cite{HeavyIId}. We will define a network as a graph with edge weights taking values in some normed space. To begin let $G_n$ be the complete network on $\{1,\dots,n\}$ whose weight on edge $\{i,j\}$ equals $\xi^n_{ij}$ for some collection $(\xi^n_{ij})_{1\leq i\leq j\leq n}$ of i.i.d.\ random variables taking values in some normed space. Now consider the rooted network $(G_n,1)$ with the distinguished vertex $1$.
	For any realization $(\xi_{ij}^n)$, and for any $B,H\in\NN$ such that $(B^{H+1}-1)/(B-1)\leq n$, we will define a finite rooted subnetwork $(G_n,1)^{B,H}$ of $(G_n,1)$   whose vertex set coincides with a $B$-ary tree of depth $H$. To this end we partially index the vertices of $(G_n,1)$ as elements in 
	$$J_{B,H}:=\bigcup_{l=0}^H\{1,\dots,B\}^l\subset\NN^f,$$
	the indexing being given by an injective map $\sigma_n$ from $J_{B,H}$ to $V_n:=\{1,\dots,n\}$. We set $I_\varnothing:=\{1\}$ and the index of the root $\sigma_n^{-1}(1)=\varnothing$. The vertex $v\in V_n\setminus I_\varnothing$ is given the index $(k)=\sigma_n^{-1}(v), 1\leq k\leq B$, if $\xi^n_{1,v}$ has the $k$-{th} largest norm value among $\{\xi_{1j}^n, j\neq 1 \}$, ties being broken by lexicographic order\footnote{To help keep track of notation in this section, note that $v=(w)\in V_n$ if $w\in J_{B,H}$ and $\sigma_n(w)=v$.}. This defines the first generation, and let $I_1$ be the union of $I_\varnothing$ and this generation. If $H\geq 2$ repeat this process for the vertex labeled $(1)$ on $V_n\setminus I_1$ to order $\{\xi_{(1)j}^n\}_{j\in V_n\setminus I_1}$ to get $\{11,12,\dots,1B\}$. Define $I_2$ to be the union of $I_1$ and this new collection. Repeat again for $(2),(3),\dots,(B)$ to get the second generation and so on. Call this vertex set $V_n^{B,H}=\sigma_n J_{B,H}$. 
	
	For a realization $T$ of $\mathbf{PWIT}(\rho)$, recall we assign the weight $y_{vk}$ to the edge $\{v,vk\}$. Then $(T,\varnothing)$ is a rooted network. Call $(T,\varnothing)^{B,H}$ the finite rooted subnetwork obtained by restricting $(T,\varnothing)$ to the vertex set $J_{B,H}$. If an edge is not present in $(T,\varnothing)^{B,H}$ assign the weight $0$. We say a sequence $(G_n,1)^{B,H}$, for fixed $B$ and $H$, converges in distribution, as $n\rightarrow\infty$, to $(T,\varnothing)^{B,H}$ if the joint distribution of the weights converges weakly. 
	
	Let $\pi_n$ be the permutation on $\{1,\dots,n \}$ associated to the permutation matrix $P_n$. We let $$\xi_{ij}^n=\left(\xi_{ij}^{n,(1)},\xi_{ij}^{n,(2)}\right) := \left(\frac{X_{\pi_n(i),\pi_n(j)}}{a_n},\frac{X_{\pi_n(j),\pi_n(i)}}{a_n}\right).$$
	
	We now consider $(G_n,1)^{B,H}$ with weights $\xi_{ij}^n$ and a realization, $T$, of $\mathbf{PWIT}(\theta_d\times m_\alpha)$. We aim to show $(G_n,1)^{B,H}$ converges in distribution to $(T,\varnothing)^{B,H}$, for fixed $B,H$ as $n\rightarrow\infty$. 
	
	Order the elements of $J_{B,H}$ lexicographically, i.e. $\varnothing\prec 1\prec 2\prec\cdots\prec B\prec 11\prec 12\prec\cdots\prec B\cdots B$. For $v\in J_{B,H}$ let $\mathcal{O}_v$ denote the offspring of $v$ in $(G,1)^{B,H}$. By construction $I_\varnothing=\{1\}$ and $I_v=\sigma_n\left(\bigcup_{w\prec v}\mathcal{O}_w \right)$, where $w\prec v$ must be strict in this union. Thus at every step of the indexing procedure we order the weights of neighboring edges not already considered at a previous step. Thus for all $v$,
	$$(\xi_{\sigma_n(v),j }^n)_{j\notin I_v}\overset{d}{=}(\xi_{1j}^n)_{1< j\leq n-|I_v|}.$$  
	
	Note that by independence, Proposition \ref{OrderStatConv} still holds if you take the empirical sum over $\{1,\dots,n \}\setminus I$ for any fixed finite set $I$. Thus by Proposition \ref{OrderStatConv} the weights from a fixed parent to its offspring in $(G_n,1)^{B,H}$ converge weakly to those of $(T,\varnothing)^{B,H}$. By independence we can extend this to joint convergence. Because for any fixed $n$, $(G_n,1)^{B,H}$ is still a complete network on $V_n^{B,H}$, we must now check the weights connected to vertices not indexed above converge to zero, which is the weight given to edges not in the tree. For $v,w\in J_{B,H}$ define 
	$$x_{v,w}^n:=\xi_{\sigma_n(v),\sigma_n(w)}^n.$$
	
	Also let $\{z_{v,w}^n, v,w\in J_{B,H} \}$ denote independent variables distributed as $\|\xi_{12}^n\|$. Let $E^{B,H}$ denote the set of edges $\{v,w\}\in J_{B,H}\times J_{B,H}$ that do not belong to the finite subtree $(T,\varnothing)^{B,H}$. Because we have sorted out the largest elements, the vector $\{\|x_{v,w}^n\|, \{v,w\}\in E^{B,H} \}$ is stochastically dominated by the vector $Z_n:=\{z_{v,w}^n, \{v,w\}\in E^{B,H} \}$ (see \cite{SymHeavy} Lemma 2.7). Since $J_{B,H}$ is finite, the vector $Z_n$ converges to $0$ as $n\rightarrow \infty$. Thus $(G,1)^{B,H}\Rightarrow (T,\varnothing)^{B,H}$.
	
	Let $A$ be the operator associated to $\mathbf{PWIT}(\theta_d\times m_\alpha)$ defined by (\ref{eq:LimitOp}). For fixed $B,H$ let $\sigma_n^{B,H}$ be the map $\sigma_n$ above associated to $(G,1)^{B,H}$, and arbitrarily extend $\sigma_n^{B,H}$ to a bijection on $\NN^f$, where $V_n$ is considered in the natural way as a subset of the offspring of $\varnothing$.  From the Skorokhod Representation Theorem we may assume $(G_n,1)^{B,H}$ converges almost surely to $(T,\varnothing)^{B,H}$. Thus there is a sequences $B_n,H_n$ tending to infinity and $\hat{\sigma}_n:=\sigma_n^{B_n,H_n}$ such that for any pair $v,w\in\NN^f$, $\xi_{\hat{\sigma}_n(v),\hat{\sigma}_n(w)}^n$ converges almost surely to 
	$$\begin{cases}
		y_{vk},\quad \text{if $w=vk$ for some $k$}\\
		y_{wk},\quad \text{if $v=wk$ for some $k$}\\
		0,\quad\text{otherwise.}
	\end{cases}$$
	
	Thus almost surely 
	$$\langle\delta_v,\hat{\sigma}_n^{-1}P_nA_nP_n^T\hat{\sigma}_n\delta_w\rangle=\xi_{\hat{\sigma}_n(v),\hat{\sigma}_n(w)}^{n,(i)}\rightarrow\langle\delta_v,A\delta_w\rangle,$$
	where $(i)=1,2$ depending on whether $w$ is an offspring of $v$, vice versa, or we take the convention $i=1$ if neither in which case $\xi_{\hat{\sigma}_n(v),\hat{\sigma}_n(w)}^{n,(1)}\rightarrow 0$. To prove the local convergence of operators it is sufficient, by linearity, to prove point wise convergence for any $\delta_w$. For convenience let $\phi_n^w:=\hat{\sigma}_n^{-1}P_nA_nP_n^T\hat{\sigma}_n\delta_w$. Thus all that remains to be shown to complete the proof of Theorem \ref{LocalConv} is that almost surely as $n\rightarrow\infty$
	$$\sum_{u\in\NN^f}|\langle\delta_u,\phi_n^w\rangle-\langle\delta_u,A\delta_w\rangle |^2\rightarrow 0.$$
	
	Since for every $u$, $\langle\delta_u,\phi_n^w\rangle\rightarrow\langle\delta_u,A\delta_w\rangle$ and $\langle\delta_u, A\delta_w\rangle$ is square summable in $u$ it is enough to show that if $u\in\NN^f$ are given some indexing by $\NN$ $u_1,u_2,\dots$, then 
	$$\sup_{n\geq 1} \sum_{i=k}^{\infty}|\langle\delta_{u_i},\phi_n^w\rangle|^2\rightarrow 0$$ as $k\rightarrow\infty$. This follows from the uniform square-integrability of order statistics, see Lemma 2.4 of \cite{SymHeavy}. This completes the proof of Theorem \ref{LocalConv}.

	\subsection{Resolvent matrix}. Let $R$ be the resolvent of the bipartized random operator of $A$. For $U(z,\eta)\in\HH_+$, set \begin{equation}\label{eq:limitresolvent}
		R(U)_{\varnothing\varnothing}=\begin{pmatrix}
			a(z,\eta)&b(z,\eta)\\
			b'(z,\eta)&c(z,\eta)
		\end{pmatrix}.
	\end{equation}
	
	We have the following result.
	\begin{theorem}\label{ResConv}
		Let $P_n$, $A_n$, and $A$ be as in the Theorem \ref{LocalConv}. Since $B(z)$ is almost surely self-adjoint we may almost surely define $R$, the resolvent of $B(z)$. Let $\hat{R}_n$ be the resolvent of the bipartized matrix of $P_nA_nP_n^T$. For all $U\in\HH_+$, 
		$$\hat{R}_n(U)_{11}\Rightarrow R(U)_{\varnothing\varnothing}$$
		as $n \to \infty$. 
		
	\end{theorem} 
	
	\begin{proof}
		This follows immediately from Theorem \ref{LocalConv}, Proposition \ref{BSelfAdjoint}, and Theorem \ref{LCtoRes}.
	\end{proof}
	As functions, the entries of the resolvent matrix are continuous, and by Lemma \ref{BoundedRes} are bounded. Thus \begin{equation}\label{eq:Resolventexpectation}
		\lim\limits_{n\rightarrow\infty}\EE\hat{R}_n(U)_{11}=\EE R(U)_{\varnothing\varnothing}.
	\end{equation}
	Note that by independence of $\pi_n$ and $A_n$ \begin{equation}\label{eq:PermutedResolvent}
		\EE\hat{R}_n(U)_{11}=\EE R_n(U)_{\pi_n^{-1}(1)\pi_n^{-1}(1)}=\EE\frac{1}{n}\sum_{i=1}^nR_n(U)_{ii}.
	\end{equation}
	This is the reason for the choice of uniformly distributed $P_n$ independent of $A_n$. 
	
	\begin{theorem}\label{PfofSingConv}
		For all $z\in\CC$, almost surely the measures $\check{\nu}_{A_n-z}$ converge weakly to a deterministic probability measure $\check{\nu}_{\alpha,z,\theta_d}$ whose Stieltjes transform is given by 
		$$m_{\check{\nu}_{\alpha,z,\theta_d}}(\eta)=\EE a(z,\eta),$$
		for $\eta\in\CC_+$ and $a(z,\eta)$ in (\ref{eq:limitresolvent}).
	\end{theorem}
	
	\begin{proof}
		By Proposition \ref{BSelfAdjoint} for every $z\in\CC$, $B(z)$ is almost surely essentially self-adjoint. Thus, using the Borel functional calculus, there exists almost surely a random probability measure $\nu_{\varnothing,z}$ on $\RR$ such that 
		$$a(z,\eta)=\langle \delta_\varnothing,R(U)\delta_\varnothing \rangle=\int_\RR\frac{d\nu_{\varnothing,z}(x)}{x-\eta}=m_{\nu_{\varnothing,z}}(\eta).$$
		See Theorem VIII.5 in \cite{MathPhys1} for more on this measure and the Borel functional calculus for unbounded self-adjoint operators. Define $R_n$ as the resolvent matrix of $B_n(z)$, the bipartized matrix of $A_n$. For $U(z,\eta)\in\HH_+$, we write 
		$$R_n(U)_{kk}=\begin{pmatrix}
			a_k(z,\eta)&b_k(z,\eta)\\
			b'_k(z\eta)&c_k(z,\eta)
		\end{pmatrix},\qquad \hat{R}_n(U)_{kk}=\begin{pmatrix}
			\hat{a}_k(z,\eta)&\hat{b}_k(z,\eta)\\
			\hat{b}'_k(z\eta)&\hat{c}_k(z,\eta)
		\end{pmatrix}.$$
		By Theorem \ref{StRe}
		$$m_{\EE\check{\nu}_{A-z}}(\eta)=\EE\frac{1}{2n}\sum_{k=1}^n(a_k(z,\eta)+c_{k}(z,\eta))=\frac{1}{2}\left(\EE\hat{a}_1(z,\eta)+\EE\hat{c}_1(z,\eta) \right)=\EE\hat{a}_1(z,\eta).$$
		Thus, by (\ref{eq:Resolventexpectation}) and (\ref{eq:PermutedResolvent}),
		$$\lim\limits_{n\rightarrow\infty}m_{\EE\check{\nu}_{A-z}}(\eta)=\EE a(z,\eta).$$
		It follows that $\EE\check{\vv}_{A_n-z}$ converges to some deterministic probability measure $\check{\vv}_{z,\alpha,\theta_d}=\EE\nu_{\varnothing,z}$. By Lemma \ref{TVwIndWed} $\check{\nu}_{A_n-z}$ concentrates around its expected value, and thus by the Borel-Cantelli Lemma $\check{\vv}_{A_n-z}$ converges almost surely to $\check{\vv}_{z,\alpha,\theta_d}$. 
	\end{proof}
	
	Theorem \ref{SingValueConv} follows immediately from Theorem \ref{PfofSingConv}. We conclude this section with a recursive distributional equation describing $R(U)_{\varnothing\varnothing}$. 
	
	\begin{proposition}\label{prop:RDE}
		Let $\theta_d$ be a probability measure. Let $(w^{(1)},w^{(2)})$ be a $\theta_d$ distributed random vector independent of the matrix given in \eqref{eq:limitresolvent}. Additionally let $\rho_{z,\eta}$ be the measure on $\CC^4\cong\RR^8$ which is the distribution of the random vector \begin{equation}
			\begin{pmatrix}
				c(z,\eta)|w^{(1)}|^2\\
				b'(z,\eta)w^{(1)}w^{(2)}\\
				b(z,\eta)\bar{w}^{(1)}\bar{w}^{(2)}\\
				a(z,\eta)|w^{(2)}|^2
			\end{pmatrix},
		\end{equation} where $a(z,\eta),b(z,\eta),b'(z,\eta),$ and $c(z,\eta)$ are as in \eqref{eq:limitresolvent}. Then the matrix given in \eqref{eq:limitresolvent} satisfies the following recursive distributional equation \begin{equation}\label{eq:RDE}
		\begin{pmatrix}
			a(z,\eta)&b(z,\eta)\\
			b'(z,\eta)&c(z,\eta)
		\end{pmatrix}\overset{d}{=}-\left(\begin{pmatrix}
		\eta&z\\
		\bar{z}&\eta
	\end{pmatrix}+\begin{pmatrix}
	S_1&S_2\\
	S_3&S_4
\end{pmatrix} \right)^{-1}
	\end{equation} where $S=(S_1,S_2,S_3,S_4)$ is an $\alpha/2$-stable random vector with spectral measure is given by the image of the measure $\frac{\Gamma(2-\alpha/2)\cos(\pi\alpha/4)}{1-\alpha/2}\|v\|^{\alpha/2}d\rho_{z,\eta}(v)$ under the map $v\mapsto \frac{v}{\|v\|}$.
	\end{proposition} 
The proof of Proposition \ref{prop:RDE} is given in Appendix \ref{app:PropRDEProof}.
	
	\section{Least singular values of elliptic random matrices}\label{sect:LeastSingValue} 
	
	Now that we have proven Theorem \ref{SingValueConv}, we move on to show that $\log(\cdot)$ is uniformly integrable, in probability, with respect to $\{\nu_{A_n-zI_n}\}_{n\geq 1}$. We begin with a bound on the least singular value of an elliptic random matrix under very general assumptions. This section is entirely self contained.
	
	\begin{theorem}[Least singular value bound] \label{thm:lsv}
		Let $X = (X_{ij})$ be an $n \times n$ complex-valued random matrix such that
		\begin{enumerate}[(i)]
			\item (off-diagonal entries) $\{ (X_{ij}, X_{ji}) : 1 \leq i < j \leq n\}$ is a collection of independent random tuples,
			\item (diagonal entries) the diagonal entries $\{ X_{ii} : 1 \leq i \leq n\}$ are independent of the off-diagonal entries (but can be dependent on each other),
			\item there exists $a > 0$ such that
			the events 
			\begin{equation} \label{eq:defEij}
				\mathcal{E}_{ij} := \{ |X_{ij}| \leq a, |X_{ji}| \leq a \} 
			\end{equation}
			defined for $i \neq j$ satisfy 
			\[ b := \min_{i < j} \Prob(\mathcal{E}_{ij}) > 0, \quad \sigma^2 := \min_{i \neq j} \var(X_{ij} \mid \mathcal{E}_{ij}) > 0, \]
			and 
			\[ \rho := \max_{i < j} \left| { \corr( X_{ij} \mid \mathcal{E}_{ij}, X_{ji} \mid \mathcal{E}_{ji}) } \right| < 1. \]
		\end{enumerate}
		Then there exists $C = C(a, b, \sigma) > 0$ such that for $n \geq C$, any $M \in \Mat_n(\mathbb{C})$, $s \geq 1$, $0 < t \leq 1$, 
		\[ \Prob \left( s_n(X + M) \leq \frac{t}{\sqrt{n}}, s_1(X+M) \leq s \right) \leq C \left( \frac{\log (Cns)}{\sqrt{1-\rho}} \left( \sqrt{s^5 t} + \frac{1}{\sqrt{n}} \right) \right)^{1/4}. \]
	\end{theorem}
	
	\begin{remark}
		The constant $C$ from Theorem \ref{thm:lsv} only depends on $a,b,\sigma$ and does not depend on $\rho$.  This allows one to apply Theorem \ref{thm:lsv} to cases where $\rho$ depends on $n$.  
	\end{remark}

	\subsection{Proof of Theorem \ref{thm:lsv}}
	
	In this section we prove Theorem \ref{thm:lsv}.  Suppose $X$ and $M$ satisfy the assumptions of the theorem and denote $A := X + M$. We will use $A$ throughout this section and it may be worth noting it is not the $A_n$ of Theorems \ref{SingValueConv} and \ref{EigenConv}.  Throughout the section, we allow all constants to depend on $a,b, \sigma$ without mentioning or denoting this dependence.  Constants, however, will not depend on $\rho$; instead we will state all dependence on $\rho$ explicitly.  
	
	For the proof of Theorem \ref{thm:lsv}, it suffices to assume that $A$ and every principle submatrix of $A$ is invertible with probability $1$.  To see this, define $X' := X + \frac{t}{\sqrt{n}} \xi I$, where $I$ is the identity matrix and $\xi$ is a real-valued random variable uniformly distributed on the interval $[-1,1]$, independent of $X$.  It follows that $X'$ satisfies the assumptions of Theorem \ref{thm:lsv}.  However, since $\xi$ is continuously distributed, it also follows that $A' := X' + M$ and every principle submatrix of $A'$ is invertible with probability $1$.  By Weyl's inequality for the singular values (see, for instance, \cite[Problem III.6.13]{Bhatia}), we find
	\[ \max_{1 \leq k \leq n} |s_k(A) - s_k(A')| \leq \frac{t}{\sqrt{n}} \leq s. \]
	Hence, we conclude that
	\[ \Prob \left( s_n(A) \leq \frac{t}{\sqrt{n}}, s_1(A) \leq s \right) \leq \Prob \left( s_n(A') \leq \frac{2t}{\sqrt{n}}, s_1(A') \leq 2s \right).  \]
	In other words, it suffices to prove Theorem \ref{thm:lsv} under the additional assumption that $A$ and every principle submatrix of $A$ is invertible.  We work under this additional assumption for the remainder of the proof.  
	
	\subsection{Nets and a decomposition of the unit sphere}
	
	Consider a compact set $K \subset \C^n$ and $\eps > 0$.  A subset $\mathcal{N} \subset K$ is called an \emph{$\eps$-net} of $K$ if for every point $v \in K$ one has $\dist(v, \mathcal{N}) \leq \eps$.  
	
	For some real positive parameters $\delta, \tau > 0$ that will be determined later, we define the set of \emph{$\delta$-sparse vectors} as
	\[ \Sparse(\delta) := \{ x \in \mathbb{C}^n : |\supp(x)| \leq \delta n\}. \]
	We decompose the unit sphere $\mathbb{S}^{n-1}$ into the set of \emph{compressible vectors} and the complementary set of \emph{incompressible} vectors by
	\[ \Comp(\delta, \tau) := \{x \in \Sp^{n-1} : \dist(x, \Sparse(\delta)) \leq \tau \} \]
	and
	\[ \Incomp(\delta, \tau) := \Sp^{n-1} \setminus \Comp(\delta, \tau). \]
	
	We have the following result for incompressible vectors.  
	
	\begin{lemma}[Incompressible vectors are spread; Lemma A.3 from \cite{BC}] \label{lemma:lsv:spread}
		Let $x \in \Incomp(\delta, \tau)$.  There exists a subset $\pi \subset [n]$ such that $|\pi| \geq \delta n / 2$ and for all $i \in \pi$, 
		\[ \frac{\tau}{\sqrt{n}} \leq |x_i| \leq \sqrt{ \frac{2}{\delta n} }. \]
	\end{lemma}

	\subsection{Control of compressible vectors}
	
	The case of compressible vectors roughly follows the arguments from \cite{BC}.  For $i,j \in [n]$, we let $C_j$ denote the $j$-th column of $A$ and $C_j^{(i)}$ denote the $j$-th column of $A$ with the $i$-th entry removed.  
	
	\begin{lemma}[Distance of a random vector to a deterministic subspace] \label{lemma:lsv:distance}
		There exist constants $\eps, C, c, \delta_0 > 0$ such that for all $1 \leq i \leq n$ and any deterministic subspace $H$ of $\mathbb{C}^{n-1}$ with $1 \leq \dim(H) \leq \delta_0 n$, we have 
		\[ \Prob \left(\dist(C_i^{(i)}, H) \leq \eps \sqrt{n} \right) \leq C \exp(-c n). \]
	\end{lemma}
	\begin{proof}
		The proof follows the same arguments as those given in the proof of \cite[Theorem A.2]{BC}.  Fix $1 \leq i \leq n$; the arguments and bounds below are all uniform in $i$.  Recall the definitions of the events $\mathcal{E}_{ij}$ given in \eqref{eq:defEij}.  By the assumptions on $\mathcal{E}_{ij}$, the Chernoff bound gives
		\[ \Prob \left( \sum_{j \neq i } \oindicator{\mathcal{E}_{ji}} \leq \frac{(n-1)b}{2} \right) \leq \exp \left( - \frac{ (n-1)b}{8} \right). \]
		In other words, with high probability, at least $m := \left \lceil \frac{ (n-1)b}{2} \right \rceil$ of the events $\mathcal{E}_{ji}$, $j \neq i$ occur.  
		Thus, it suffices to prove the result by conditioning on the event 
		\begin{equation} \label{eq:Emcap}
			E_m := \bigcap_{j \in [m], j \neq i} \mathcal{E}_{ji}. 
		\end{equation}
		
		There are two cases to consider.  Either $i \in [m]$ or $i > m$.  The arguments for these two cases are almost identical except some notations must be changed slightly to remove the $i$-th index.  For the remainder of the proof, let us only consider the case when $i > m$; the changes required for the other case are left to the reader.  
		
		Recall that it suffices to prove the result by conditioning on the event $E_m$ defined in \eqref{eq:Emcap}.  In fact, as $i > m$, the definition of the event $E_m$ given in \eqref{eq:Emcap} can be stated as
		\[ E_m := \bigcap_{j \in [m]} \mathcal{E}_{ji}. \]
		Let $\E_m[ \cdot ] := \E[ \cdot \mid E_m, \mathcal{F}_m]$ denote the conditional expectation given the event $E_m$ and the filtration $\mathcal{F}_m$ generated by $X_{ji}$, $j > m$, $j \neq i$.  Let $W$ be the subspace spanned by $H$ and the vectors 
		\[ u := (0, \ldots, 0, X_{m+1, i}, \ldots, X_{i-1, i}, X_{i+1, i}, \ldots, X_{n, i} ) \]
		and
		\[ w := ( \E_m[X_{1, i}], \ldots, \E_m[X_{m, i}], 0, \ldots, 0). \]
		By construction, $\dim(W) \leq \dim(H) + 2$ and $W$ is $\mathcal{F}_m$ measurable.  In addition, 
		\[ \dist(C_i^{(i)}, H) \geq \dist(C_i^{(i)}, W) = \dist(Y, W), \]
		where
		\[ Y := ( X_{1,i} - \E_m[X_{1,i}], \ldots, X_{m,i} - \E_m[X_{m,i}], 0, \ldots, 0) = C_{i}^{(i)} - u - w. \]
		By assumption, the coordinates $Y_1, \ldots, Y_m$ are independent random variables which satisfy 
		\[ |Y_k| \leq 2a, \qquad E_m[Y_k] = 0, \qquad \E_m|Y_k|^2 \geq \sigma^2 \]
		for $1 \leq k \leq m$.   
		Thus, since the function $x \mapsto \dist(x,W)$ is convex and $1$-Lipschitz, Talagrand's concentration inequality (see for instance \cite[Theorem 2.1.13]{Tbook}) yields
		\begin{equation} \label{eq:probmtalagrand}
			\Prob_m( |\dist(Y,W) - \E_m[\dist(Y,W)]| \geq t) \leq C \exp \left( - c\frac{t^2}{a^2} \right) 
		\end{equation}
		for every $t \geq 0$, where $C, c > 0$ are absolute constants.  In particular, this implies that
		\begin{equation} \label{eq:Emtalagrand}
			( \E_m \dist(Y,W) )^2 \geq \E_m \dist^2(Y, W) - c' a^2 
		\end{equation}
		for an absolute constant $c' >0$.  Thus, if $P$ denotes the orthogonal projection onto the orthogonal complement of $W$, we get
		\begin{align*}
			\E_m \dist^2(Y,W) &= \sum_{k=1}^m \E_m|Y_k|^2 P_{kk} \\
			&\geq \sigma^2 \left(\tr P - \sum_{k=m+1}^{n-1} P_{kk} \right) \\
			&\geq \sigma^2 ((n-1) - \dim(H) - 2 - (n-1-m) ) \\
			&\geq \sigma^2 \left( \frac{ (n-1)b}{2}  - \delta_0 n - 2 \right).
		\end{align*}
		The last term is lower bounded by $c'' \sigma^2 n$ for all $n$ sufficiently large by taking $\delta_0 := b/4$.  Combining this bound with \eqref{eq:probmtalagrand} and \eqref{eq:Emtalagrand} completes the proof.  
	\end{proof}
	
	The next bound, which follows as a corollary of Lemma \ref{lemma:lsv:distance}, will be useful when dealing with compressible vectors.  
	
	\begin{corollary} \label{cor:lsv:distance}
		There exist $\eps, C, c, \delta_0 > 0$ such that for any deterministic subset $\pi \subset [n]$ with $|\pi| \leq \delta_0 n $ and any deterministic $u \in \mathbb{C}^n$, we have
		\[ \Prob \left( \min_{i \in \pi} \dist(C_i, H_i) \leq \eps \sqrt{n} \right) \leq Cn \exp(-c n), \]
		where $H_i := \Span\left(\{ C_j : j \in \pi, j \neq i\}\cup\{u\} \right)$.  
	\end{corollary}
	\begin{proof}
		We will apply Lemma \ref{lemma:lsv:distance} to control $\dist(C_i, H_i)$.  To this end, define $u^{(i)}$ to be the vector $u$ with the $i$-th entry removed, and set $H_i^{(i)} := \Span\left(\{C_j^{(i)} : j \in \pi, j \neq i \} \cup\{u^{(i)}\}\right)$.  Then
		\[ \dist(C_i, H_i) \geq \dist(C_i^{(i)}, H_i^{(i)}) \]
		for all $1 \leq i \leq n$.  Note that $H_i^{(i)}$ is independent of $C_i^{(i)}$.  Hence, conditioning on $H_i^{(i)}$ and applying Lemma \ref{lemma:lsv:distance}, we find the existence of $\eps, C, c, \delta_0$ such that for any $\pi \subset [n]$ with $|\pi| \leq \delta_0 n$ and all $i \in \pi$ 
		\begin{align*}
			\Prob \left( \dist(C_i, H_i) \leq {\eps} \sqrt{n} \right) \leq \Prob \left( \dist(C_i^{(i)}, H_i^{(i)}) \leq {\eps} \sqrt{n} \right) \leq C \exp(-c  n). 
		\end{align*}
		Therefore, by the union bound, 
		\[ \Prob \left( \min_{i \in \pi} \dist(C_i, H_i) \leq {\eps} \sqrt{n} \right) \leq \sum_{i \in \pi} \Prob \left( \dist(C_i, H_i) \leq {\eps} \sqrt{n} \right) \leq Cn \exp(-c n), \]
		and the proof is complete.  
	\end{proof}
	
	Let $\eps, \delta_0$ be as in Corollary \ref{cor:lsv:distance}.  From now on we set\begin{equation}\label{eq:taudef}
		\tau := \frac{1}{4} \min\left\{ 1, \frac{\eps}{s \sqrt{\delta}} \right\}.
	\end{equation}
	(Recall that $s \geq 1$ is the upper bound for $s_1(X+M)$ specified in the statement of Theorem \ref{thm:lsv}.)
	In particular, this definition implies that $\tau \leq 1/4$.  The parameter $\delta$ is still to be specified.  Right now we only assume that $\delta < \delta_0$.  
	
	\begin{lemma}[Control of compressible vectors] \label{lemma:lsv:comp}
		There exist constants $C, c, \delta > 0$ such that for any deterministic vector $u \in \mathbb{C}^n$ and any $s \geq 1$
		\[ \Prob \left( \inf_{x \in \Comp(\delta, \tau)} \|Ax - u \| \leq \frac{\eps}{2 \sqrt{\delta}}, s_1(A) \leq s \right) \leq C \exp(- c  n). \]
	\end{lemma}
	\begin{proof}
		Let $0 < \delta < \delta_0$ be a constant to be chosen later.  We decompose $\Comp(\delta, \tau)$ as
		\[ \Comp(\delta, \tau) = \bigcup_{\pi \subset [n] : |\pi| = \lfloor \delta n \rfloor} S_{\pi}, \]
		where
		\[ S_{\pi} := \{x \in \Comp(\delta, \tau) : \dist(x, \Sparse_\pi(\delta)) \leq \tau \} \]
		and
		\[ \Sparse_\pi(\delta) := \{ y \in \Sparse(\delta) : \supp(y) \subset \pi \}. \]
		So by the union bound, 
		\begin{align}
			\Prob &\left( \inf_{x \in \Comp(\delta, \tau)} \|Ax - u \| \leq \frac{\eps}{2 \sqrt{\delta}}, s_1(A) \leq s \right) \label{eq:ubdist}\\
			&\qquad\qquad\leq \sum_{\pi \in [n] : |\pi| = \lfloor \delta n \rfloor} \Prob \left( \inf_{x \in S_{\pi}} \|Ax - u \| \leq \frac{\eps}{2 \sqrt{\delta}}, s_1(A) \leq s \right). \nonumber
		\end{align}
		Fix $\pi \subset [n]$ with $|\pi| = \lfloor \delta n \rfloor$, and suppose there exists $x \in S_{\pi}$ such that $\|Ax - u \| \leq \frac{\eps}{2 \sqrt{\delta}}$ and $s_1(A) \leq s$.  Then there exists $y \in \Sparse_{\pi}(\delta)$ with $\|x - y \| \leq \tau$.  In particular, this implies that $\|y\| \geq 3/4$.  In addition, we have
		\begin{align*}
			\frac{\eps}{2 \sqrt{\delta}} \geq \|Ax - u \| \geq \|A y - u \| - \tau \|A \| \geq \|A y - u \| - \frac{\eps}{4 \sqrt{\delta}} 
		\end{align*}
		by the assumption that $\|A \| \leq s$ and the definition of $\tau$ (\ref{eq:taudef}).  Hence, we obtain
		\begin{equation} \label{eq:34Ayu}
			\| Ay - u \| \leq \frac{3 \eps}{4 \sqrt{\delta}}. 
		\end{equation}
		We now bound $\| Ay - u \|$ from below.  Indeed, we have
		\begin{align} \label{eq:Ayubnd}
			\|A y - u \|^2 &= \left\| \sum_{i \in \pi} C_i y_i - u \right\|^2  \geq \max_{i \in \pi} |y_i|^2 \dist^2(C_i, H_i), 
		\end{align}
		where $H_i := \Span \left(\{C_j : j \in \pi, j \neq i \}\cup\{u \}\right)$.  In addition, we bound 
		\begin{align*}
			\max_{i \in \pi} |y_i|^2 \dist^2(C_i, H_i) \geq \min_{i \in \pi} \dist^2(C_i, H_i) \frac{1}{|\pi|} \sum_{i \in \pi} |y_i|^2 \geq \min_{i \in \pi} \dist^2(C_i, H_i) \left( \frac{3}{4} \right)^2 \frac{1}{\delta n}.   
		\end{align*}
		Thus, combining \eqref{eq:34Ayu} and \eqref{eq:Ayubnd} with the bound above, we find 
		\[ \min_{i \in \pi} \dist(C_i, H_i) \leq \eps \sqrt{n}. \]
		To conclude, we have shown that
		\[ \Prob \left( \inf_{x \in S_{\pi}} \|Ax - u \| \leq \frac{\eps}{2 \sqrt{\delta}}, s_1(A) \leq s \right) \leq \Prob \left( \min_{i \in \pi} \dist(C_i, H_i) \leq \eps \sqrt{n} \right). \]
		In view of Corollary \ref{cor:lsv:distance}, there exist $C, c > 0$ such that for any $\pi \in [n]$ with $|\pi| = \lfloor \delta n \rfloor$, we have
		\[ \Prob \left( \inf_{x \in S_{\pi}} \|Ax - u \| \leq \frac{\eps}{2 \sqrt{\delta}}, s_1(A) \leq s \right) \leq Cn \exp(-c n). \]
		Returning to \eqref{eq:ubdist}, we conclude that
		\begin{align*}
			\Prob \left( \inf_{x \in \Comp(\delta, \tau)} \|Ax - u \| \leq \frac{\eps }{2 \sqrt{\delta}}, s_1(A) \leq s \right) &\leq \binom{n}{\lfloor \delta n \rfloor} Cn \exp(-c n) \\
			&\leq \left( \frac{ n e }{\lfloor \delta n \rfloor} \right)^{\lfloor \delta n \rfloor} C n \exp(-c  n) \\
			&\leq C' \exp(-c' n)
		\end{align*}
		for some constants $C', c' > 0$ by taking $\delta$ sufficiently small (in terms of $c$).  
	\end{proof}
	
	We now fix $\delta$ to be the constant from Lemma \ref{lemma:lsv:comp}.  Thus, $\delta$ and $\tau$ are now completely determined.  We will also need the following corollary of Lemma \ref{lemma:lsv:comp}.
	\begin{corollary}
		There exist constants $C, c > 0$ such that for any deterministic vector $u \in \mathbb{C}^n$ and any $s \geq 1$
		\begin{equation} \label{eq:xxcomp}
			\Prob \left( \inf_{x/\|x\| \in \Comp(\delta, \tau)} \frac{\|Ax - u \|}{\|x\|} \leq \frac{\eps}{4 \sqrt{\delta}}, s_1(A) \leq s \right) \leq Cs \exp(- c  n). 
		\end{equation}
	\end{corollary}
	\begin{proof}
		The proof is based on the arguments given in \cite{Rsym}.  We first note that if $u = 0$, then the claim follows immediately from Lemma \ref{lemma:lsv:comp}.  Assume $u \neq 0$.  Let $\mathcal{E}$ denote the event on the left-hand side of \eqref{eq:xxcomp} whose probability we would like to bound.  Suppose that $\mathcal{E}$ holds.  Then there exists $x_0 := x / \|x\| \in \Comp(\delta, \tau)$ and $u_0 := u / \|x\| \in \Span(u)$ such that $\|A x_0 - u_0 \| \leq \frac{\eps}{4 \sqrt{\delta}}$ and $s_1(A) \leq s$.  In particular, this implies that
		\[ \| u_0 \| \leq \|Ax_0 - u_0 \| + \|A x_0\| \leq \frac{\eps}{4 \sqrt{\delta}} + s. \]
		Let $\mathcal{N}$ be a $\frac{\eps}{4 \sqrt{\delta}}$-net of the real interval $[-\frac{\eps}{4 \sqrt{\delta}} - s, \frac{\eps}{4 \sqrt{\delta}} + s]$.  In particular, we can choose $\mathcal{N}$ so that
		\begin{equation} \label{eq:Nrealbnd}
			|\mathcal{N}| \leq C' s 
		\end{equation}
		for some constant $C' > 0$.  Here, we have used the assumption that $s \geq 1$.  In particular, there exists $c_0 \in \mathcal{N}$ such that
		\[ \left\| \frac{u}{\|x\|} - c_0 \frac{u}{\|u\|} \right\| \leq \frac{\eps}{4 \sqrt{\delta}}. \]
		By the triangle inequality, this implies that
		\[ \left\| Ax_0 - c_0 \frac{u}{\|u\|} \right\| \leq \frac{\eps}{2 \sqrt{\delta}}. \]
		To conclude, we have shown that 
		\[ \Prob(\mathcal{E}) \leq \Prob \left( \inf_{c_0 \in \mathcal{N}} \inf_{x_0 \in \Comp(\delta, \tau)} {\left\|Ax_0 - c_0 \frac{u}{\|u\|} \right\|} \leq \frac{\eps}{2 \sqrt{\delta}}, s_1(A) \leq s \right), \]
		and thus, by the union bound, we have
		\[ \Prob(\mathcal{E}) \leq \sum_{c_0 \in \mathcal{N}} \Prob \left( \inf_{x_0 \in \Comp(\delta, \tau)} {\left\|Ax_0 - c_0 \frac{u}{\|u\|} \right\|} \leq \frac{\eps}{2 \sqrt{\delta}}, s_1(A) \leq s \right). \]
		The claim now follows by the cardinality bound for $\mathcal{N}$ in \eqref{eq:Nrealbnd} and Lemma \ref{lemma:lsv:comp}.  
	\end{proof}
	
	When dealing with incompressible vectors we will need the following corollary.  
	
	\begin{corollary} \label{cor:lsv:inversestruct}
		There exists constants $C, c > 0$, such that, for any $u \in \C^{n}$ with $u \neq 0$, and any $s \geq 1$, we have
		\[ \Prob \left( \frac{  A^{-1} u }{ \| A^{-1} u \| } \in \Comp(\delta, \tau), s_1(A) \leq s \right) \leq C s \exp(-c n). \]
	\end{corollary}
	\begin{proof}
		If $x := A^{-1} u$, then $(Ax - u)/\|x\| = 0$.  Thus, we have
		\begin{align*}
			\Prob &\left( \frac{  A^{-1} u }{ \| A^{-1} u \| } \in \Comp(\delta, \tau), s_1(A) \leq s \right) \\
			&\qquad\qquad\leq \Prob \left( \inf_{x / \|x\| \in \Comp(\delta, \tau)} \frac{ \|A x - u \|}{\|x\|} = 0, s_1(A) \leq s \right). 
		\end{align*}
		The conclusion now follows from \eqref{eq:xxcomp}.  
	\end{proof}

	\subsection{Anti-concentration bounds} \label{sec:anticonc}
	
	In order to handle incompressible vectors, we will need several anti-concentration bounds.  The main idea is to use the rate of convergence from the Berry--Esseen Theorem to obtain the estimates.  This idea appears to have originated in \cite{LPRT} and has been used previously in many works including \cite{BC, LPRT, RVLO}.  
	
	\begin{lemma}[Small ball probability via Berry--Esseen; Lemma A.6 from \cite{BC}] \label{lemma:lsv:BE}
		There exists $C > 0$ such that if $Z_1, \ldots, Z_n$ are independent centered complex-valued random variables, then for all $t \geq 0$, 
		\[ \sup_{z \in \C} \Prob \left( \left| \sum_{i=1}^n Z_i - z \right| \leq t \right) \leq C \left( \frac{t}{\sqrt{ \sum_{i=1}^n \E|Z_i|^2} } + \frac{ \sum_{i=1}^n \E|Z_i|^3}{ \left( \sum_{i=1}^n \E|Z_i|^2 \right)^{3/2}} \right). \]
	\end{lemma}

	We begin the following anti-concentration bound for sums involving dependent random variables.  
	
	\begin{lemma} \label{lemma:lsv:anticor}
		Let $\{(\xi_i, \psi_i) : 1 \leq i \leq n\}$ be a collection of independent complex-valued random tuples, and assume there exist $a, b, \sigma > 0$ such that the events
		\[ \mathcal{E}_i := \{ |\xi_i| \leq a, |\psi_i| \leq a \} \]
		for $1 \leq i \leq n$ satisfy 
		\[ b \leq \min_{1 \leq i \leq n} \Prob(\mathcal{E}_i), \qquad \sigma^2 \leq \min_{1 \leq i \leq n} \var( \xi_i \mid \mathcal{E}_i), \qquad \sigma^2 \leq \min_{1 \leq i \leq n} \var( \psi_i \mid \mathcal{E}_i). \]
		In addition, assume there exists $\rho < 1$ such that
		\begin{equation} \label{eq:assumprhocorr}
			\max_{1 \leq i \leq n} | \corr( \xi_i \mid \mathcal{E}_i, \psi_i \mid \mathcal{E}_i ) | \leq \rho. 
		\end{equation}
		Then for any $\delta, \tau \in (0,1)$, any $w=(w_i)_{i=1}^n \in \Incomp(\delta, \tau)$, any $w' = (w'_i)_{i=1}^n \in \C^n$ with $\| w' \| \leq 1$, any $J \subset [n]$ with $|J| \geq n(1-\delta/4)$, and any $t \geq 0$, we have
		\begin{align*}
			\sup_{z \in \C} \Prob &\left( \left| \sum_{i \in J} (\xi_i w_i + \psi_i w_i') - z \right| \leq t \tau \right) \\
			&\leq \frac{C}{\sigma \sqrt{\delta b (1 - \rho)}} \sqrt{ \left \lceil \frac{1}{2} \log_2 \left( \frac{2}{\tau^2 \delta} \right) \right \rceil \left \lceil \log_2 \left( \frac{\sqrt{n}}{\tau} \right) \right \rceil } \left( t + \frac{a}{\sqrt{n}} \right) + \exp(-\delta n b / 32), 
		\end{align*}
		where $C > 0$ is an absolute constant. 
	\end{lemma}
	\begin{proof}
		The proof is based on the arguments from \cite{BC}.  Since $w \in \Incomp(\delta, \tau)$ and $|J| \geq n(1-\delta/4)$, Lemma \ref{lemma:lsv:spread} implies the existence of $\pi \subset J$ such that $|\pi| \geq \delta n /4$ and
		\[ \frac{ \tau}{\sqrt{n}} \leq |w_i| \leq \sqrt{\frac{2}{\delta n}} \]
		for all $i \in \pi$.  By conditioning on the random variables $\xi_i, \psi_i$ for $i \not\in \pi$ and absorbing their contribution into the constant $z$, it suffices to bound
		\[ \sup_{z \in \C} \Prob \left( \left| \sum_{i \in \pi} (\xi_i w_i + \psi_i w_i') - z \right| \leq t \tau \right). \]
		
		We now proceed to truncate the random variables $\xi_i, \psi_i$ for $i \in \pi$.  Indeed, by the Chernoff bound, it follows that
		\[ \sum_{i \in \pi} \oindicator{\mathcal{E}_i} \geq \frac{ |\pi| b}{2} \geq \frac{\delta b n}{8} \]
		with probability at least $1 - \exp(-\delta b n / 32)$.  Therefore, taking $m := \lceil \delta b n / 8 \rceil$, it suffices to show
		\begin{align*}
			\sup_{z \in \C} \Prob_m &\left( \left| \sum_{i =1}^m (\xi_i w_i + \psi_i w_i') - z \right| \leq t \tau \right) \\
			&\leq \frac{C}{\sigma \sqrt{\delta b (1 - \rho)}} \sqrt{ \left \lceil \frac{1}{2} \log_2 \left( \frac{2}{\tau^2 \delta} \right) \right \rceil \left \lceil \log_2 \left( \frac{\sqrt{n}}{\tau} \right) \right \rceil } \left( t + \frac{a}{\sqrt{n}} \right), 
		\end{align*}
		where $\Prob_m(\cdot) := \Prob(\cdot \vert E_m, \mathcal{F}_m)$ is the conditional probability given $\mathcal{F}_m$, the $\sigma$-algebra generated by all random variables except $\xi_1, \ldots, \xi_m, \psi_1, \ldots, \psi_m$, and the event
		\[ E_m := \left\{ |\xi_i| \leq a, |\psi_i| \leq a : 1 \leq i \leq m \right\} \bigcap \left\{ \frac{ \tau}{\sqrt{n}} \leq |w_i| \leq \sqrt{\frac{2}{\delta n}} : 1 \leq i \leq m\right\}.\footnote{When defining this event, we consider the $w_i$ to be constant random variables.} \]
		By centering the random variables  (and absorbing the expectations into the constant $z$), it suffices to bound
		\[ \sup_{z \in \C} \Prob_m \left( \left| \sum_{i=1}^m \left[ (\xi_i - \E_m[\xi_i]) w_i  + (\psi_i - \E_m[\psi_i]) w'_i \right] - z \right| \leq t \tau \right). \]
		
		We again reduce to the case where we only need to consider a subset of the coordinates of $w$ and $w'$ which are roughly comparable.  Indeed, as the random tuples  $(\xi_1, \psi_1), \ldots, (\xi_m, \psi_m)$ are jointly independent under the probability measure $\Prob_m$, we can condition on any subset of them (and again absorb their contribution into the constant $z$); hence, for any subset $I \subset [m]$, we have
		\begin{align} \label{eq:partitionsum2}
			\sup_{z \in \C} \Prob_m &\left( \left| \sum_{i=1}^m \left[ (\xi_i - \E_m[\xi_i]) w_i  + (\psi_i - \E_m[\psi_i]) w'_i \right] - z \right| \leq t \tau \right) \\
			&\leq \sup_{z \in \C} \Prob_m \left( \left| \sum_{i \in I} \left[ (\xi_i - \E_m[\xi_i]) w_i  + (\psi_i - \E_m[\psi_i]) w'_i \right] - z \right| \leq t \tau \right). \nonumber
		\end{align}
		
		We now choose the subset $I$ in a sequence of two steps.  First, define
		\[ L := 2 \left \lceil \frac{1}{2} \log_2 \left(\frac{2}{\delta \tau^2}\right) \right \rceil, \]
		and for $1 \leq j \leq L$ set
		\[ I_j := \left\{1 \leq i \leq m : \frac{2^{j-1} \tau}{\sqrt{n}} \leq |w_i| < \frac{2^j \tau}{\sqrt{n}} \right\}. \]
		By construction, $I_1, \ldots, I_L$ partition the index set $[m]$.  Hence, by the pigeonhole principle, there exists $j$ such that $|I_j| \geq m / L$.  Second, we partition the set $I_j$ as follows.  Define 
		\[ K := 2 \left \lceil \log_2 \left( \frac{\sqrt{n}}{\tau} \right) \right \rceil, \]
		and for $1 \leq k \leq K$ set
		\[ I_{j,k} := \left\{ i \in I_j : \frac{2^{k-1} \tau}{\sqrt{n}} \leq |w_i'| < \frac{2^k \tau}{\sqrt{n}} \right\} \]
		and define
		\[ I_{j,0} := \left\{ i \in I_j : |w'_i| < \frac{\tau}{\sqrt{n}} \right\}.  \]
		As $\|w ' \| \leq 1$ by assumption, the sets $I_{j,0}, I_{j,1}, \ldots, I_{j,K}$ form a partition of $I_j$.  By the pigeonhole principle, there exists $k$ such that 
		\begin{equation} \label{eq:Ijklowbnd}
			|I_{j,k}| \geq \frac{m}{L (K+1)} \geq \frac{m}{2 LK }. 
		\end{equation}
		Applying \eqref{eq:partitionsum2} to the set $I_{j,k}$, it now suffices to show
		\begin{align} \label{eq:sufficesupzm}
			\sup_{z \in \C} \Prob_m &\left( \left| \sum_{i \in I_{j,k}} \left[ (\xi_i - \E_m[\xi_i]) w_i  + (\psi_i - \E_m[\psi_i]) w'_i \right] - z \right| \leq t \tau \right) \\
			&\qquad\qquad\qquad \leq \frac{ C \sqrt{LK }}{\sigma \sqrt{ \delta b (1 - \rho) } } \left( t + \frac{a} {\sqrt{n}} \right)  \nonumber
		\end{align}
		for some absolute constant $C > 0$.  
		
		We will apply Lemma \ref{lemma:lsv:BE} to obtain \eqref{eq:sufficesupzm}.  For $i \in I_{j,k}$, define
		\[ Z_i := (\xi_i - \E_m[\xi_i]) w_i  + (\psi_i - \E_m[\psi_i]) w'_i.  \]
		Then
		\begin{align*}
			q^2 &:= \sum_{i \in I_{j,k}} \E_m |Z_i|^2 \\
			&\geq \sum_{i \in I_{j,k}} \left[ |w_i|^2 \var_m(\xi_i) + |w'_i|^2 \var_m(\psi_i) - 2 \rho |w_i| |w'_i| \sqrt{ \var_m(\xi_i) \var_m(\psi_i) }  \right]
		\end{align*}
		by assumption \eqref{eq:assumprhocorr}.  Thus, we deduce that
		\begin{align}
			q^2 &\geq (1 - \rho) \sum_{i \in I_{j,k}} \left[ |w_i|^2 \var_m(\xi_i) + |w_i'|^2 \var_m(\psi_i) \right] \nonumber\\
			&\geq (1 - \rho) \sigma^2 \sum_{i \in I_{j,k}} \left[ |w_i|^2 + |w_i'|^2 \right]. \label{eq:qlowbnd}
		\end{align}
		In addition, 
		\[ \sum_{i \in I_{j,k}} \E_m |Z_i|^3 \leq 2a \tau \left( \frac{2^j + 2^k}{\sqrt{n}} \right) q^2. \]
		Hence, Lemma \ref{lemma:lsv:BE} implies the existence of an absolute constant $C > 0$ such that
		\begin{align} 
			\sup_{z \in \C} \Prob_m &\left( \left| \sum_{i \in I_{j,k}} \left[ (\xi_i - \E_m[\xi_i]) w_i  + (\psi_i - \E_m[\psi_i]) w'_i \right] - z \right| \leq t \tau \right) \nonumber \\
			&\qquad\qquad\qquad\leq \frac{C \tau}{q} \left(  t + \frac{a (2^j + 2^k) }{ \sqrt{n} } \right). \label{eq:probmconcabs}
		\end{align}
		
		We complete the proof by considering two separate cases.  First, if $k = 0$, then using \eqref{eq:Ijklowbnd} and \eqref{eq:qlowbnd}, we obtain
		\[ q^2 \geq \sigma^2 (1 - \rho) \sum_{i \in I_{j,k}} |w_i|^2 \geq \sigma^2 (1 - \rho) |I_{j,k}| \frac{2^{2j - 2} \tau^2}{n} \geq \sigma^2 (1 - \rho) \delta b \frac{ 2^{2j} \tau^2}{64 LK }. \]
		Hence returning to \eqref{eq:probmconcabs}, we find that
		\begin{align}
			\sup_{z \in \C} \Prob_m &\left( \left| \sum_{i \in I_{j,k}} \left[ (\xi_i - \E_m[\xi_i]) w_i  + (\psi_i - \E_m[\psi_i]) w'_i \right] - z \right| \leq t \tau \right) \nonumber \\
			&\qquad\qquad\qquad\leq \frac{ C' \sqrt{LK }}{\sigma \sqrt{ \delta b (1 - \rho) } } \left( t + \frac{a} {\sqrt{n}} \right), \label{eq:absconstconc1}
		\end{align}
		where $C' > 0$ is an absolute constant.  
		Similarly, if $1 \leq k \leq K$, then 
		\[ q^2 \geq \sigma^2 (1 - \rho) |I_{j,k}| \frac {\tau^2}{n} \left( 2^{2j - 2} + 2^{2k - 2} \right) \geq \sigma^2 (1 - \rho) \delta b \tau^2 \frac{ 2^{2j} + 2 ^{2k}}{64 LK }. \]
		In this case, we again apply \eqref{eq:probmconcabs} to deduce the existence of an absolute constant $C'' > 0$ such that
		\begin{align}
			\sup_{z \in \C} \Prob_m &\left( \left| \sum_{i \in I_{j,k}} \left[ (\xi_i - \E_m[\xi_i]) w_i  + (\psi_i - \E_m[\psi_i]) w'_i \right] - z \right| \leq t \tau \right) \nonumber \\
			&\qquad\qquad\qquad\leq \frac{ C'' \sqrt{LK }}{\sigma \sqrt{ \delta b (1 - \rho) } } \left( t + \frac{a} {\sqrt{n}} \right). \label{eq:absconstconc2}
		\end{align}
		Combining \eqref{eq:absconstconc1} and \eqref{eq:absconstconc2}, we obtain the bound \eqref{eq:sufficesupzm} (with the absolute constant $C := \max\{C', C''\}$), and the proof of the lemma is complete.  
	\end{proof}
	
	Lastly, we will need the following technical anti-concentration bound which is similar to Lemma \ref{lemma:lsv:anticor}.  
	
	\begin{lemma} \label{lemma:lsv:anticor2}
		Let $\{(\xi_i, \psi_i) : 1 \leq i \leq n\}\cup\{ (\xi_i', \psi_i') : 1 \leq i \leq n\}$ be a collection of independent complex-valued random tuples with the property that $(\xi_i, \psi_i)$ has the same distribution as $(\xi_i', \psi_i')$ for $1 \leq i \leq n$, and assume there exist $a, b, \sigma > 0$ such that the events
		\[ \mathcal{E}_i := \{ |\xi_i| \leq a, |\psi_i| \leq a \} \]
		for $1 \leq i \leq n$ satisfy 
		\[ b \leq \min_{1 \leq i \leq n} \Prob(\mathcal{E}_i), \qquad \sigma^2 \leq \min_{1 \leq i \leq n} \var( \xi_i \mid \mathcal{E}_i), \qquad \sigma^2 \leq \min_{1 \leq i \leq n} \var( \psi_i \mid \mathcal{E}_i). \]
		In addition, let $\delta_1, \ldots, \delta_n$ be i.i.d.\ $\{0,1\}$-valued random variables with $\E \delta_i = c_{oo} \in (0,1]$, and assume $\delta_1, \ldots, \delta_n$ are independent of the random tuples $\{(\xi_i, \psi_i), (\xi_i', \psi_i') : 1 \leq i \leq n\}$. Then for any $\delta, \tau \in (0,1)$, any $w=(w_i)_{i=1}^n \in \Incomp(\delta, \tau)$, any $w' = (w'_i)_{i=1}^n \in \C^n$ with $\| w' \| \leq 1$, and any $t \geq 0$, we have
		\begin{align*}
			\sup_{z \in \C} \Prob &\left( \left| \sum_{i=1}^n (\delta_i \xi_i w_i + \delta_i \xi_i' w_i') - z \right| \leq t \tau \right) \\
			&\leq \frac{C}{\sigma \sqrt{c_{oo} \delta b^2 }} \sqrt{ \left \lceil \frac{1}{2} \log_2 \left( \frac{2}{\tau^2 \delta} \right) \right \rceil \left \lceil \log_2 \left( \frac{\sqrt{n}}{\tau} \right) \right \rceil } \left( t + \frac{a}{\sqrt{n}} \right) + \exp(-\delta n c_{oo} b^2 / 16), 
		\end{align*}
		where $C > 0$ is an absolute constant. 
	\end{lemma}
	
	Lemma \ref{lemma:lsv:anticor2} is very similar to Lemma \ref{lemma:lsv:anticor}. The statement of the lemma is somewhat unusual since the hypotheses involve the variables $\xi_i, \psi_i, \xi_i', \psi_i'$, but the conclusion only involves the random variables $\xi_i, \xi_i'$.  This is to match the assumptions of Theorem \ref{thm:lsv}.  The proof of Lemma \ref{lemma:lsv:anticor2} presented below follows the same framework as the proof of Lemma \ref{lemma:lsv:anticor}.  
	
	\begin{proof}[Proof of Lemma \ref{lemma:lsv:anticor2}]
		Since $w \in \Incomp(\delta, \tau)$, Lemma \ref{lemma:lsv:spread} implies the existence of $\pi \subset [n]$ such that $|\pi| \geq \delta n /2$ and
		\[ \frac{ \tau}{\sqrt{n}} \leq |w_i| \leq \sqrt{\frac{2}{\delta n}} \]
		for all $i \in \pi$.  By conditioning on the random variables $\xi_i, \psi_i, \xi_i', \psi_i', \delta_i$ for $i \not\in \pi$ and absorbing their contribution into the constant $z$, it suffices to bound
		\[ \sup_{z \in \C} \Prob \left( \left| \sum_{i \in \pi} (\delta_i \xi_i w_i + \delta_i \xi_i' w_i') - z \right| \leq t \tau \right). \]
		
		We now proceed to truncate the random variables $\xi_i, \psi_i, \xi_i', \psi_i'$ for $i \in \pi$.  Define the events
		\[ \mathcal{E}_i' := \{ |\xi_i'| \leq a, |\psi_i'| \leq a \}, \]
		and observe that $\mathcal{E}_i'$ is independent of $\mathcal{E}_i$ for $1 \leq i \leq n$.  By the Chernoff bound, it follows that
		\[ \sum_{i \in \pi} \oindicator{\mathcal{E}_i\cap\mathcal{E}_i' \cap \{\delta_i = 1\}} \geq \frac{ c_{oo}|\pi| b^2}{2} \geq \frac{c_{oo} \delta b^2 n}{4} \]
		with probability at least $1 - \exp(-c_{oo}\delta b^2 n / 16)$.  Therefore, taking $m := \lceil c_{oo} \delta b^2 n / 4 \rceil$, it suffices to bound
		\begin{align} \label{eq:sufftobndprobm}
			\sup_{z \in \C} \Prob_m &\left( \left| \sum_{i =1}^m (\delta_i \xi_i w_i + \delta_i \xi_i' w_i') - z \right| \leq t \tau \right) = \sup_{z \in \C} \Prob_m &\left( \left| \sum_{i =1}^m (\xi_i w_i + \xi_i' w_i') - z \right| \leq t \tau \right) 
		\end{align}
		where $\Prob_m(\cdot) := \Prob(\cdot \vert E_m, \mathcal{F}_m)$ is the conditional probability given $\mathcal{F}_m$, the $\sigma$-algebra generated by all random variables except $\xi_1, \ldots, \xi_m$, $\psi_1, \ldots, \psi_m$, $\xi_1', \ldots, \xi_m'$, $\psi_1', \ldots, \psi_m'$, $\delta_1, \ldots, \delta_m$, and the event
		\[ E_m := \bigcap_{i=1}^m \left\{ |\xi_i| \leq a, |\psi_i| \leq a, |\xi_i'| \leq a, |\psi_i'| \leq a, \delta_i = 1 \right\} \bigcap_{i=1}^m \left\{ \frac{ \tau}{\sqrt{n}} \leq |w_i| \leq \sqrt{\frac{2}{\delta n}} \right\}. \]
		Here, we have exploited the fact that on the event $E_m$, $\delta_i = 1$ for $i \in [m]$, and so all factors of $\delta_i$ have been replaced by $1$ on the right-hand side of \eqref{eq:sufftobndprobm}.  
		By centering the random variables  (and absorbing the expectations into the constant $z$), it suffices to bound
		\[ \sup_{z \in \C} \Prob_m \left( \left| \sum_{i=1}^m \left[ (\xi_i - \E_m[\xi_i]) w_i  + (\xi_i' - \E_m[\xi_i']) w'_i \right] - z \right| \leq t \tau \right). \]
		
		We again reduce to the case where we only need to consider a subset of the coordinates of $w$ and $w'$ which are roughly comparable.  Indeed, as the random vectors  $(\xi_1, \psi_1, \xi_1', \psi_1'), \ldots, (\xi_m, \psi_m, \xi_m', \psi_m')$ are jointly independent under the probability measure $\Prob_m$, we can condition on any subset of them (and again absorb their contribution into the constant $z$); hence, for any subset $I \subset [m]$, we have
		\begin{align} \label{eq:partitionsum3}
			\sup_{z \in \C} \Prob_m &\left( \left| \sum_{i=1}^m \left[ (\xi_i - \E_m[\xi_i]) w_i  + (\xi_i' - \E_m[\xi_i']) w'_i \right] - z \right| \leq t \tau \right) \\
			&\leq \sup_{z \in \C} \Prob_m \left( \left| \sum_{i \in I} \left[ (\xi_i - \E_m[\xi_i]) w_i  + (\xi_i' - \E_m[\xi_i']) w'_i \right] - z \right| \leq t \tau \right). \nonumber
		\end{align}
		
		We now choose the subset $I$ in a sequence of two steps as was done in the proof of Lemma \ref{lemma:lsv:anticor}.  First, define
		\[ L := 2 \left \lceil \frac{1}{2} \log_2 \left(\frac{2}{\delta \tau^2}\right) \right \rceil, \]
		and for $1 \leq j \leq L$ set
		\[ I_j := \left\{1 \leq i \leq m : \frac{2^{j-1} \tau}{\sqrt{n}} \leq |w_i| < \frac{2^j \tau}{\sqrt{n}} \right\}. \]
		By construction, $I_1, \ldots, I_L$ partition the index set $[m]$.  Hence, by the pigeonhole principle, there exists $j$ such that $|I_j| \geq m / L$.  Second, we partition the set $I_j$ as follows.  Define 
		\[ K := 2 \left \lceil \log_2 \left( \frac{\sqrt{n}}{\tau} \right) \right \rceil, \]
		and for $1 \leq k \leq K$ set
		\[ I_{j,k} := \left\{ i \in I_j : \frac{2^{k-1} \tau}{\sqrt{n}} \leq |w_i'| < \frac{2^k \tau}{\sqrt{n}} \right\} \]
		and define
		\[ I_{j,0} := \left\{ i \in I_j : |w'_i| < \frac{\tau}{\sqrt{n}} \right\}.  \]
		As $\|w ' \| \leq 1$ by assumption, the sets $I_{j,0}, I_{j,1}, \ldots, I_{j,K}$ form a partition of $I_j$.  By the pigeonhole principle, there exists $k$ such that 
		\begin{equation} \label{eq:Ijklowbnd2}
			|I_{j,k}| \geq \frac{m}{L (K+1)} \geq \frac{m}{2 LK }. 
		\end{equation}
		Applying \eqref{eq:partitionsum3} to the set $I_{j,k}$, it now suffices to show
		\begin{align} \label{eq:sufficesupzm2}
			\sup_{z \in \C} \Prob_m &\left( \left| \sum_{i \in I_{j,k}} \left[ (\xi_i - \E_m[\xi_i]) w_i  + (\xi_i' - \E_m[\xi_i']) w'_i \right] - z \right| \leq t \tau \right) \\
			&\qquad\qquad\qquad \leq C \frac{ \sqrt{LK }}{\sigma \sqrt{ c_{oo} \delta b^2} } \left( t + \frac{a} {\sqrt{n}} \right) \nonumber
		\end{align}
		for some absolute constant $C > 0$.  
		
		We will apply Lemma \ref{lemma:lsv:BE} to obtain \eqref{eq:sufficesupzm2}.  For $i \in I_{j,k}$, define
		\[ Z_i := (\xi_i - \E_m[\xi_i]) w_i  + (\xi_i' - \E_m[\xi_i']) w'_i.  \]
		Then
		\begin{align}
			q^2 &:= \sum_{i \in I_{j,k}} \E_m |Z_i|^2 \nonumber \\
			&= \sum_{i \in I_{j,k}} \left[ |w_i|^2 \var_m(\xi_i) + |w'_i|^2 \var_m(\xi_i') \right] \label{eq:qlowbnd2} \\
			&\geq \sigma^2 \sum_{i \in I_{j,k}} \left[ |w_i|^2 + |w_i'|^2 \right] \nonumber
		\end{align}
		since $\xi_i$ and $\xi_i'$ are independent under the measure $\Prob_m$ for $1 \leq i \leq n$.  
		In addition, 
		\[ \sum_{i \in I_{j,k}} \E_m |Z_i|^3 \leq 2a \tau \left( \frac{2^j + 2^k}{\sqrt{n}} \right) q^2. \]
		Hence, Lemma \ref{lemma:lsv:BE} implies the existence of an absolute constant $C > 0$ such that
		\begin{align} 
			\sup_{z \in \C} \Prob_m &\left( \left| \sum_{i \in I_{j,k}} \left[ (\xi_i - \E_m[\xi_i]) w_i  + (\xi_i' - \E_m[\xi_i']) w'_i \right] - z \right| \leq t \tau \right) \nonumber \\
			&\qquad\qquad\qquad\leq \frac{C \tau}{q} \left(  t + \frac{a (2^j + 2^k) }{ \sqrt{n} } \right). \label{eq:probmconcabs2}
		\end{align}
		
		We complete the proof by considering two separate cases.  First, if $k = 0$, then using \eqref{eq:Ijklowbnd2} and \eqref{eq:qlowbnd2}, we obtain
		\[ q^2 \geq \sigma^2 \sum_{i \in I_{j,k}} |w_i|^2 \geq \sigma^2 |I_{j,k}| \frac{2^{2j - 2} \tau^2}{n} \geq \sigma^2 c_{oo} \delta b^2 \frac{ 2^{2j} \tau^2}{32 LK }. \]
		Hence returning to \eqref{eq:probmconcabs2}, we find that
		\begin{align}
			\sup_{z \in \C} \Prob_m &\left( \left| \sum_{i \in I_{j,k}} \left[ (\xi_i - \E_m[\xi_i]) w_i  + (\xi_i' - \E_m[\xi_i']) w'_i \right] - z \right| \leq t \tau \right) \nonumber \\
			&\qquad\qquad\qquad\leq \frac{ C' \sqrt{LK }}{\sigma \sqrt{ c_{oo} \delta b^2 } } \left( t + \frac{a} {\sqrt{n}} \right), \label{eq:absconstconc12}
		\end{align}
		where $C' > 0$ is an absolute constant.  
		Similarly, if $1 \leq k \leq K$, then 
		\[ q^2 \geq \sigma^2 |I_{j,k}| \frac {\tau^2}{n} \left( 2^{2j - 2} + 2^{2k - 2} \right) \geq \sigma^2 c_{oo} \delta b^2 \tau^2 \frac{ 2^{2j} + 2 ^{2k}}{32 LK }. \]
		In this case, we again apply \eqref{eq:probmconcabs2} to deduce the existence of an absolute constant $C'' > 0$ such that
		\begin{align}
			\sup_{z \in \C} \Prob_m &\left( \left| \sum_{i \in I_{j,k}} \left[ (\xi_i - \E_m[\xi_i]) w_i  + (\xi_i' - \E_m[\xi_i]') w'_i \right] - z \right| \leq t \tau \right) \nonumber \\
			&\qquad\qquad\qquad\leq \frac{ C'' \sqrt{LK }}{\sigma \sqrt{ c_{oo} \delta b^2 } } \left( t + \frac{a} {\sqrt{n}} \right). \label{eq:absconstconc22}
		\end{align}
		Combining \eqref{eq:absconstconc12} and \eqref{eq:absconstconc22}, we obtain the bound \eqref{eq:sufficesupzm2} (with the absolute constant $C := \max\{C', C''\}$), and the proof of the lemma is complete.  
	\end{proof}

	\subsection{Incompressible vectors}
	
	In order to control the set of incompressible vectors, we will require the following averaging estimate.  
	
	\begin{lemma}[Invertibility via average distance; Lemma A.4 from \cite{BC}] \label{lemma:lsv:avg}
		Let $A$ be a random matrix taking values in $\Mat_n(\mathbb{C})$ with columns $C_1, \ldots, C_n$.  For any $1 \leq k \leq n$, let $H_k := \Span\{ C_i : i \neq k\}$.  Then, for any $t \geq 0$, 
		\[ \Prob \left( \min_{x \in \Incomp(\delta, \tau)} \|A x \| \leq \frac{ t \tau}{ \sqrt{n} } \right) \leq \frac{2}{\delta n} \sum_{k=1}^n \Prob( \dist(C_k, H_k) \leq t ). \]
	\end{lemma}
	
	Let $A = X +M$ be the matrix from Theorem \ref{thm:lsv}.  Let $C_1, \ldots, C_n$ be the columns of $A$ and $H_k := \Span \{C_i : i \neq k\}$ be as in Lemma \ref{lemma:lsv:avg}.  Our main result for controlling the set of incompressible vectors is the following.  
	
	\begin{lemma} \label{lemma:lsv:dist}
		There exists $C, c > 0$ such that for every $1 \leq k \leq n$, any $t > 0$ and any $s \geq 1$, 
		\[ \Prob( \dist(C_k, H_k) \leq t, s_1(A) \leq s ) \leq  C \left( \frac{ \log(Csn) }{\sqrt{1-\rho}} \left( s^2 \sqrt{t} + \frac{1}{\sqrt{n}} \right) \right)^{1/4}.  \]
	\end{lemma}
	
	The rest of the subsection is devoted to the proof of Lemma \ref{lemma:lsv:dist}.  We complete the proof of Theorem \ref{thm:lsv} in Subsection \ref{sec:proof:lsv}.  We will also need the following result based on \cite[Proposition 5.1]{Rsym} and \cite[Statement 2.8]{GNT}.  
	
	\begin{lemma}[Distance problem via bilinear forms] \label{lemma:lsv:bilinear}
		Let $A = (A_{ij}) \in \Mat_n(\mathbb{C})$, let $C_1, \ldots, C_n$ denote the columns of $A$, and fix $1 \leq k \leq n$.  Let $H_k := \Span\{C_i : i \neq k\}$, $u$ be the $k$-th row of $A$ with the $k$-th entry removed, $v$ be $C_k$ with the $k$-th entry removed, and let $B$ be the $(n-1) \times (n-1)$ submatrix of $A$ formed from removing the $k$-th row and $k$-th column.  If $B$ is invertible, then
		\begin{equation} \label{eq:distancequotient}
			\dist(C_k, H_k) \geq \frac{ |A_{kk} - u B^{-1} v|}{ \sqrt{ 1 + \| u B^{-1} \|^2 } }. 
		\end{equation}
	\end{lemma}
	\begin{proof}
		The proof presented below is based on the arguments given in \cite{GNT, Rsym}.  By permuting the rows and columns, it suffices to assume that $k=1$.  
		Let $h \in \mathbb{S}^{n-1}$ denote any normal to the hyperplane $H_1$.  Then
		\[ \dist(C_1, H_1) \geq | h^\ast C_1|. \]
		We decompose
		\[ C_1 = \begin{pmatrix} A_{11} \\ v \end{pmatrix}, \quad h = \begin{pmatrix} h_1 \\ g \end{pmatrix}, \]
		where $h_1 \in \mathbb{C}$ and $g \in \mathbb{C}^{n-1}$.  Then
		\begin{equation} \label{eq:lsv:distdot}
			\dist(C_1, H_1) \geq | h^\ast C_1| = |\bar{h}_1 A_{11} + g^\ast v |. 
		\end{equation}
		Since $h$ is orthogonal to the columns of the matrix $\begin{pmatrix} u \\ B \end{pmatrix}$, we find
		\[ 0 = h^\ast \begin{pmatrix} u \\ B \end{pmatrix} = \bar{h}_1 u + g^\ast B, \]
		and hence
		\[ g^\ast = - \bar{h}_1 u B^{-1}. \]
		Returning to \eqref{eq:lsv:distdot}, we have
		\begin{equation} \label{eq:lsv:distbnd}
			\dist(C_1, H_1) \geq |h_1| |A_{11} - uB^{-1} v |. 
		\end{equation}
		In addition, 
		\[ 1 = \| h \|^2 = |h_1|^2 + \|g \|^2 = |h_1|^2 (1 + \| u B^{-1} \|^2), \]
		and so
		\begin{equation} \label{eq:lsv:h1}
			|h_1|^2 = \frac{1}{1 + \| u B^{-1} \|^2}. 
		\end{equation}
		The conclusion now follows from \eqref{eq:lsv:distbnd} and \eqref{eq:lsv:h1}.  
	\end{proof}
	
	Our study of the bilinear form $u B^{-1} v$ is based on the following general result, which will allow us to introduce some additional independence into the problem to deal with the fact that $u$ and $v$ are dependent.  Similar decoupling techniques have also appeared in \cite{Costello, CTV, Gotze, GNT, Rsym}.  The lemma below is based on \cite[Lemma 8.4]{Rsym} for quadratic forms.  
	
	\begin{lemma}[Decoupling lemma] \label{lemma:lsv:decouple}
		Let $M \in \Mat_n(\mathbb{C})$, and let $x = (x_i)_{i=1}^n, y = (y_i)_{i=1}^n$ be random vectors in $\C^n$ such that $\{ (x_i, y_i) :1 \leq i \leq n\}$ is a collection of independent random tuples.  Let $(x', y')$ denote an independent copy of $(x,y)$.  Then for every subset $\pi \subset [n]$ and $t \geq 0$, we have
		\begin{align*} 
			\sup_{z \in \C} \Prob&( |x^{\mathrm{T}} M y - z | \leq t)^2 \\
			&\leq \Prob_{x, y, x', y'} \left( |x_\pi^{\mathrm{T}} M_{\pi \times \pi^c} (y_{\pi^c} - y'_{\pi^c}) + (x_{\pi^c} - x'_{\pi^c})^{\mathrm{T}} M_{\pi^c \times \pi} y_\pi + z_0 | \leq 2 t \right), 
		\end{align*}
		where $z_0$ is a random variable whose value is determined by $M_{\pi^c \times \pi^c}, x_{\pi^c}, y_{\pi^c}, x'_{\pi^c}, y'_{\pi^c}$.  
	\end{lemma}
	
	The proof of Lemma \ref{lemma:lsv:decouple} is based on the following decoupling bound from \cite{Scorr, Rsym}. 
	
	\begin{lemma}[Lemma 8.5 from \cite{Rsym}] \label{lemma:lsv:Rdecouple}
		Let $\xi$ and $\psi$ be independent random vectors, and let $\psi'$ be an independent copy of $\psi$.  Let $\mathcal{E}(\xi, \psi)$ be an event which is determined by the values of $\xi$ and $\psi$.  Then
		\[ \Prob( \mathcal{E}(\xi, \psi))^2 \leq \Prob( \mathcal{E}(\xi, \psi) \cap \mathcal{E}(\xi, \psi') ). \]
	\end{lemma}
	
	\begin{proof}[Proof of Lemma \ref{lemma:lsv:decouple}]
		Let $\xi$ be the random vector formed by the tuples $\{(x_i, y_i) : i \in \pi\}$, and let $\psi$ be the random vector formed from the tuples $\{ (x_i, y_i) : i \not\in \pi \}$.  Then $\xi$ and $\psi$ are independent by supposition, and we can apply Lemma \ref{lemma:lsv:Rdecouple}.  To this end, let $\tilde{x}, \tilde{y}$ be random vectors in $\mathbb{C}^n$ defined by 
		\[ \tilde{x}_{\pi} := x_{\pi}, \quad \tilde{x}_{\pi^c} := x'_{\pi^c}, \quad \tilde{y}_{\pi} := y_{\pi}, \quad \tilde{y}_{\pi^c} := y'_{\pi^c}. \]
		An application of Lemma \ref{lemma:lsv:Rdecouple} yields
		\begin{align*}
			\Prob( |x^{\mathrm{T}} M y - z | \leq t )^2 &\leq \Prob_{x, y, \tilde{x}, \tilde{y}} ( |{x}^{\mathrm{T}} M y - z| \leq t, |\tilde{x}^{\mathrm{T}} M \tilde{y} - z| \leq t ) \\
			&\leq \Prob_{x, y, \tilde{x}, \tilde{y}} (|x^{\mathrm{T}} M y - \tilde{x}^{\mathrm{T}} M \tilde{y} | \leq 2t),
		\end{align*}
		where the last inequality follows from the triangle inequality.  We now note that
		\begin{align*}
			x^{\mathrm{T}} M y - \tilde{x}^{\mathrm{T}} M \tilde{y} = x_\pi^{\mathrm{T}} M_{\pi \times \pi^c} (y_{\pi^c} - \tilde{y}_{\pi^c}) + (x_{\pi^c} - \tilde{x}_{\pi^c})^{\mathrm{T}} M_{\pi^c \times \pi} y_\pi + z_0,
		\end{align*}
		where $z_0$ depends only on $M_{\pi^c \times \pi^c}, x_{\pi^c}, y_{\pi^c}, \tilde{x}_{\pi^c}, \tilde{y}_{\pi^c}$.  Since $\tilde{x}_{\pi^c} = x'_{\pi^c}$ and $\tilde{y}_{\pi^c} = y'_{\pi^c}$, the claim follows.  
	\end{proof}

	We now turn to the proof of Lemma \ref{lemma:lsv:dist}.  The arguments presented here follow the general framework of \cite[Section 8.3]{Rsym}.    
	Fix $1 \leq k \leq n$; the arguments and bounds below will all be uniform in $k$.  Let $u$ be the $k$-th row of $A$ with the $k$-th entry removed.  Let $v$ be the $k$-th column of $A$ with the $k$-th entry removed, and let $B$ be the $(n-1) \times (n-1)$ matrix formed by removing the $k$-th row and $k$-th column from $A$.  In view of Lemma \ref{lemma:lsv:bilinear}, it suffices to prove that
	\begin{equation} \label{eq:suffdistpf}
		\Prob \left( \frac{ |A_{kk} - u B^{-1} v|}{ \sqrt{ 1 + \| u B^{-1} \|^2 } } \leq t, s_1(A) \leq s \right) \leq C \left( \frac{ \log(Csn) }{\sqrt{1-\rho}} \left( s^2 \sqrt{t} + \frac{1}{\sqrt{n}} \right) \right)^{1/4} 
	\end{equation}
	for some constants $C, c > 0$.  Our argument is based on applying Lemma \ref{lemma:lsv:decouple} to decouple the bilinear form $u B^{-1} v$ and then applying our anti-concentration bounds from Subsection \ref{sec:anticonc} to bound the resulting expressions.  We divide the proof of \eqref{eq:suffdistpf} into a number of sub-steps.  
	
	\subsubsection{Step 1: Constructing a random subset $\pi$}
	
	Following \cite{Rsym}, we decompose $[n-1]$ into two random subsets $\pi$ and $\pi^c$.  Let $\delta_1, \ldots, \delta_{n-1}$ be i.i.d.\ $\{0,1\}$-valued random variables, independent of $X$, with $\E \delta_i = c_{oo}/2$, where $c_{oo}$ is a constant defined by
	\[ c_{oo} := \delta/8 \]
	and $\delta \in (0,1)$ was previously fixed.  
	We then define $\pi := \{ i \in [n-1] : \delta_i = 0 \}$.  By the Chernoff bound, it follows that
	\begin{equation} \label{eq:condpic}
		|\pi^c| \leq c_{oo} n 
	\end{equation} 
	with probability at least $1 - C_{oo}' \exp(-c_{oo}' n)$ for some constants $C_{oo}', c_{oo}' > 0$.

	\subsubsection{Step 2: Estimating $\|B^{-1}\|_2$}
	
	Lemma \ref{lemma:lsv:sizeuB} below will allow us to estimate the denominator appearing on the right-hand side of \eqref{eq:distancequotient}.  To this end, let $(u', v')$ be an independent copy of $(u,v)$, also independent of $X$.  
	
	\begin{lemma} \label{lemma:lsv:sizeuB}
		There exist constants $C, c > 0$ such that, for any $s \geq 1$, the random matrix $B$ has the following properties with probability at least $1 - Cs \exp(-cn)$.  If $s_1(B) \leq s$, one has:
		\begin{enumerate}[(i)]
			\item \label{item:lsv:geqHSnorm} for any $t_0 \geq 0$, with probability at least $1 - C {\log(Cns)} \left( s t_0 + n^{-1/2} \right)$ in $u, u', \pi$, 
			\[ \|(u - u')_{\pi^c} B^{-1}_{\pi^c \times [n-1]} \| \geq t_0 \|B^{-1} \|_2. \]
			\item \label{item:lsv:geqHSnorm2} for any $t_0 \geq 0$, with probability at least $1 - C {\log(Cns)} \left( s t_0 + n^{-1/2} \right)$ in $v, v', \pi$, 
			\[ \| B^{-1}_{[n-1] \times \pi^c} (v - v')_{\pi^c} \| \geq t_0 \|B^{-1} \|_2. \]
		\end{enumerate}
	\end{lemma}
	
	In order to prove the lemma, we will need the following elementary result.
	
	\begin{lemma}[Sums of dependent random variables; Lemma 8.3 from \cite{Rsym}] \label{lemma:lsv:sumdependent}
		Let $Z_1, \ldots, Z_n$ be arbitrary non-negative random variables (not necessarily independent), and $p_1, \ldots, p_n$ be non-negative numbers such that 
		\[ \sum_{j=1}^n p_j = 1. \]
		Then, for every $t \geq 0$, we have
		\[ \Prob \left( \sum_{j=1}^n p_j Z_j \leq t \right) \leq 2 \sum_{j=1}^n p_j \Prob( Z_j \leq 2 t). \]
	\end{lemma}

	\begin{proof}[Proof of Lemma \ref{lemma:lsv:sizeuB}]
		By Corollary \ref{cor:lsv:inversestruct} and the union bound, under the assumption $s_1(B) \leq s$, we have
		\[ x_j := \frac{ B^{-1} e_j }{\|B^{-1} e_j \|} \in \Incomp(\delta, \tau), \qquad y_j := \frac{e_j^\ast B^{-1} }{ \| e_j^\ast B^{-1} \| } \in \Incomp(\delta, \tau) \]
		for $1 \leq j \leq n-1$ with probability at least $1 - C s \exp(-cn)$.  Here, $e_1, \ldots, e_{n-1}$ are the standard basis elements of $\mathbb{C}^{n-1}$.  Fix a realization of $B$ for which this property holds.  We will prove that both properties hold with the desired probability for this fixed realization of $B$.  
		
		For \eqref{item:lsv:geqHSnorm}, we note that
		\begin{align*}
			\|(u-u')_{\pi^c} B^{-1}_{\pi^c \times [n-1]} \|^2 &= \sum_{j=1}^{n-1} | (u-u')_{\pi^c} B^{-1}_{\pi^c \times [n-1]} e_j |^2 \\
			&= \sum_{j=1}^{n-1} |(u-u')_{\pi^c} (x_j)_{\pi^c} |^2 \| B^{-1} e_j \|^2. 
		\end{align*}
		Taking $p_j := \| B^{-1} e_j \|^2 / \| B^{-1} \|_2^2$, we see that $\sum_{j=1}^{n-1} p_j = 1$, and hence
		\begin{align*}
			\Prob_{u, u', \pi} &\left( \| (u-u')_{\pi^c} B^{-1}_{\pi^c \times [n-1]} \| \leq t_0 \| B^{-1} \|_{2} \right) \\
			&\qquad\qquad\leq \Prob_{u, u', \pi} \left( \sum_{j=1}^{n-1} |(u-u')_{\pi^c} (x_j)_{\pi^c} |^2 p_j \leq t_0^2 \right) \\
			&\qquad\qquad\leq 2 \sum_{j=1}^{n-1} p_j \Prob_{u, u', \pi} (|(u-u')_{\pi^c} (x_j)_{\pi^c} |^2 \leq 2 t_0^2 ) \\
			&\qquad\qquad\leq 2 \sup_{w \in \Incomp(\delta, \tau)} \Prob_{u,u',\pi} ( |(u-u')_{\pi^c} w_{\pi^c} | \leq \sqrt{2} t_0)
		\end{align*}
		by Lemma \ref{lemma:lsv:sumdependent}.  Recalling our choice of $\delta$, $\tau$ (\ref{eq:taudef}), and $c_{oo}$, the claim now follows from the anti-concentration bound given in Lemma \ref{lemma:lsv:anticor2}.  
		
		The proof of \eqref{item:lsv:geqHSnorm2} is similar.  Indeed, we have
		\[ \| B^{-1}_{[n-1] \times \pi^c} (v - v')_{\pi^c} \|^2 = \sum_{j=1}^{n-1} | (y_j)_{\pi^c} (v - v')_{\pi^c} |^2 \| e_j^\ast B^{-1} \|^2. \]
		Applying Lemma \ref{lemma:lsv:sumdependent} with $p_j := \| e_j^\ast B^{-1} \|^2 / \| B^{-1} \|_2^2$, we conclude that
		\begin{align*} 
			\Prob_{v, v', \pi} & \left( \| B^{-1}_{[n-1] \times \pi^c} (v - v')_{\pi^c} \| \leq t_0 \|B^{-1} \|_2 \right) \\
			&\qquad\qquad\leq 2 \sum_{j=1}^{n-1} p_j \Prob_{v,v',\pi} \left( | (y_j)_{\pi^c} (v - v')_{\pi^c} | \leq \sqrt{2} t_0 \right) \\
			&\qquad\qquad\leq 2 \sup_{w \in \Incomp(\delta,\tau)} \Prob_{v, v', \pi} \left( |w_{\pi^c}^{\mathrm{T}} (v-v')_{\pi^c}| \leq \sqrt{2} t_0 \right). 
		\end{align*}
		As before, the conclusion now follows from the anti-concentration bound given in Lemma \ref{lemma:lsv:anticor2}.  
	\end{proof}
	
	\subsubsection{Step 3: Working on the appropriate events}
	We have one last preparatory step before we can apply the decoupling lemma, Lemma \ref{lemma:lsv:decouple}.  In this step, we define the events we will need to work on for the remainder of the proof.  To this end, define the events
	\[ \mathcal{B}_A := \{ s_1(A) \leq s \}, \qquad \mathcal{B}_B := \{ s_1(B) \leq s \}, \qquad \mathcal{B}_u := \{ \| u \| \leq s \}. \]
	We note that $\mathcal{B}_A \subset \mathcal{B}_B$ and $\mathcal{B}_A \subset \mathcal{B}_u$ since $u$ and $B$ are sub-matrices of $A$.  Consider the random vectors
	\begin{equation} \label{eq:defww'}
		w := (u-u')_{\pi^c} B^{-1}_{\pi^c \times [n-1]}, \qquad w' := B^{-1}_{[n-1] \times \pi^c} (v - v')_{\pi^c}. 
	\end{equation}
	It is possible that $w = 0$ or $w' = 0$, although we will show that these events happen with small probability momentarily.  
	
	Let $t_0 > 0$.  Consider the event
	\begin{equation} \label{eq:condB2}
		\|B^{-1} \|_2 \leq \frac{1}{t_0} \min\{ \|w \|, \|w'\| \}. 
	\end{equation}
	By Lemma \ref{lemma:lsv:sizeuB}, we find
	\[ \Prob_{B, u, u', v, v', \pi} ( \eqref{eq:condB2} \text{ holds } \lor \mathcal{B}_B^c) \geq 1 - C \log (csn ) (s t_0 + n^{-1/2}) - Cs \exp(-cn) \]
	for some constants $C, c > 0$.  In particular, since $\|B^{-1}\|_2 > 0$, it follows that when the event in \eqref{eq:condB2} occurs, it must be the case that $w$ and $w'$ are both nonzero.  In order to avoid several different cases later in the proof, let us define $\omega$ and $\omega'$ as follows.  If $w$ is nonzero, we take $\omega := w$, and if $w$ is zero, we define $\omega$ to be a fixed vector in $\Incomp(\delta, \tau)$.  We define $\omega'$ analogously in terms of $w'$.  It follows that on the event \eqref{eq:condB2}, we have
	\begin{equation} \label{eq:condomega}
		\omega = w, \qquad \omega' = w'.
	\end{equation}
	
	Next, consider the event
	\begin{align} \label{eq:condIncomp}
		\frac{\omega}{\| \omega\|} \in \Incomp(\delta, \tau), \qquad \frac{\omega'}{\|\omega'\|} \in \Incomp(\delta, \tau). 
	\end{align}
	Let us fix an arbitrary realization of $u, v, u', v'$ and a realization of $\pi$ which satisfies \eqref{eq:condpic}.  We will apply Corollary \ref{cor:lsv:inversestruct} to control the event in \eqref{eq:condIncomp}.  Indeed, we only need to consider the cases when $w \neq 0$ or $w' \neq 0$.  In these cases, it follows that $\omega = w$ or $\omega' = w'$.  Let us suppose this is the case.  Then $\omega = (u - u') P_{\pi^c} B^{-1}$ or $\omega' = B^{-1} P_{\pi^c} (v - v')$, where $P_{\pi^c}$ is an orthogonal projection onto those coordinates specified by $\pi^c$.  Thus, from Corollary \ref{cor:lsv:inversestruct}, we deduce that
	\[ \Prob_B ( \eqref{eq:condIncomp} \text{ holds } \lor \mathcal{B}_B^c \mid u, v, u', v', \pi \text{ satisfies } \eqref{eq:condpic}) \geq 1 - C' s \exp(-c'n) \]
	for some constants $C', c' > 0$.  Combining the probabilities above, we conclude that
	\begin{align*} 
		\Prob_{B, u, v, u', v', \pi} &\left( ( \eqref{eq:condpic}, \eqref{eq:condB2}, \eqref{eq:condIncomp} \text{ hold}) \lor \mathcal{B}_B^c \right)  \\
		&\qquad\qquad\geq 1 - C'' \log(C''sn) (s t_0 + n^{-1/2}) - C'' s \exp(-c'' n) \\
		& \qquad\qquad= : 1 - p_0
	\end{align*}
	for some constants $C'', c'' > 0$. 
	
	It follows that there exists a realization of $\pi$ that satisfies \eqref{eq:condpic} and such that
	\[ \Prob_{B, u, v, u', v'} \left( ( \eqref{eq:condB2}, \eqref{eq:condIncomp} \text{ hold}) \lor \mathcal{B}_B^c \right) \geq 1- p_0. \]
	We fix such a realization of $\pi$ for the remainder of the proof.  Using Fubini's theorem, we deduce that the random matrix $B$ has the following property with probability at least $1 - \sqrt{p_0}$:
	\[ \Prob_{u, v, u', v'} \left( ( \eqref{eq:condB2}, \eqref{eq:condIncomp} \text{ hold}) \lor \mathcal{B}_B^c \mid B \right) \geq 1- \sqrt{p_0}. \]
	Since the event $\mathcal{B}_B$ depends only on $B$ and not on $u, v, u', v'$, it follows that the random matrix $B$ has the following property with probability at least $1 - \sqrt{p_0}$: either $\mathcal{B}_B^c$ holds, or
	\begin{equation} \label{eq:conddagger}
		\mathcal{B}_B \text{ holds and } \Prob_{u, v, u', v'} \left( \eqref{eq:condB2}, \eqref{eq:condIncomp} \text{ hold} \mid B \right) \geq 1 - \sqrt{p_0}. 
	\end{equation}
	
	\subsubsection{Step 4: Decoupling} 
	Recall that we are interested in bounding $\Prob_{B, u, v, A_{kk}}(\mathcal{E} \land \mathcal{B}_A)$, where
	\[ \mathcal{E} := \left\{ \frac{ |A_{kk} - u B^{-1} v|}{ \sqrt{ 1 + \| u B^{-1} \|^2 } } \leq t \right\}. \]
	We first observe that
	\[ \Prob_{B, u, v, A_{kk}}( \mathcal{E} \land \mathcal{B}_A) \leq \Prob_{B, u, v, A_{kk}}(\mathcal{E} \land \mathcal{B}_B \land \mathcal{B}_u). \]
	On the event $\mathcal{B}_u$, we have
	\[ \| u B^{-1} \| \leq \| u \| \| B^{-1} \| \leq s \| B^{-1} \|_2. \]
	In addition, if $s_1(B) \leq s$, then $\|B^{-1} \|_2 \geq 1/s$.  Hence, on the event $\mathcal{B}_B \land \mathcal{B}_u$, 
	\[ 1 + \| u B^{-1} \|^2 \leq 1 + s^2 \| B^{-1} \|_2^2 \leq 2 s^2 \| B^{-1} \|_2^2. \]
	Thus, we obtain
	\[ \Prob_{B, u, v, A_{kk}}( \mathcal{E} \land \mathcal{B}_A) \leq \Prob_{B, u, v, A_{kk}}(\mathcal{E}' \land \mathcal{B}_B ), \]
	where 
	\[ \mathcal{E}' := \left\{ |A_{kk} - u B^{-1} v| \leq \sqrt{2} ts \| B^{-1} \|_2 \right\}. \] 
	Thus, we find
	\begin{align*}
		\Prob_{B, u, v, A_{kk}}( \mathcal{E} \land \mathcal{B}_A) \leq \Prob_{B, u, v, A_{kk}} ( \mathcal{E}' \land \eqref{eq:conddagger} \text{ holds}) + \Prob_{B, u, v, A_{kk}}( \mathcal{B}_B \land \eqref{eq:conddagger} \text{ fails}).  
	\end{align*}
	The last probability is bounded above by $\sqrt{p_0}$ by the previous step.  We conclude that
	\[ \Prob_{B, u, v, A_{kk}}( \mathcal{E} \land \mathcal{B}_A) \leq \sup_{\substack{B \text{ satisfies } \eqref{eq:conddagger} \\ A_{kk} \in \mathbb{C}}} \Prob_{u, v}( \mathcal{E}' \mid B, A_{kk}) + \sqrt{p_0}. \]
	We now begin to work with the random vectors $u', v'$ (recall that $(u', v')$ are independent of $X$).  To do so, we will work on a larger probability space which also includes the random vectors $u', v'$.  Indeed, computing the probability above on the larger space which includes $u', v'$, we conclude that
	\[ \Prob_{B, u, v, A_{kk}}( \mathcal{E} \land \mathcal{B}_A) \leq \sup_{\substack{ B \text{ satisfies } \eqref{eq:conddagger} \\ A_{kk} \in \mathbb{C}}} \Prob_{u,v,u',v'}( \mathcal{E}' \mid B, A_{kk}) + \sqrt{p_0}. \]
	
	For the remainder of the proof, we fix a realization of $B$ which satisfies \eqref{eq:conddagger} and fix an arbitrary realization of $A_{kk}$.  By supposition, both $B$ and $A_{kk}$ are independent of $u, v, u', v'$.  It remains to bound the probability
	\[ p_1 := \sup_{z \in \mathbb{C}} \Prob_{u,v,u',v'}( \mathcal{E}_z'), \]
	where
	\[ \mathcal{E}_z' := \left\{ |z - u B^{-1} v| \leq \sqrt{2} ts \| B^{-1} \|_2 \right\}. \]
	
	To bound $p_1$, we apply the decoupling lemma, Lemma \ref{lemma:lsv:decouple}.  Indeed, by Lemma \ref{lemma:lsv:decouple}, 
	\[ p_1^2 \leq \Prob_{u, v, u', v'} (\mathcal{E}'' ), \]
	where
	\[ \mathcal{E}'' := \left\{ | u_{\pi} B^{-1}_{\pi \times \pi^c} (v - v')_{\pi^c} + (u - u')_{\pi^c} B^{-1}_{\pi^c \times \pi} v_{\pi} + z_0 | \leq 2 \sqrt{2} st \| B^{-1} \|_2 \right\} \]
	and $z_0$ is a complex number depending only on $B_\{\pi^c\times\pi^c \}$, $u_{\pi^c}, v_{\pi^c}, u_{\pi^c}', v_{\pi^c}'$.  Using \eqref{eq:conddagger} (where the conditioning on $B$ is no longer required since $B$ is now fixed), we find
	\[ p_1^2 \leq \Prob_{u, v, u', v'} (\mathcal{E}'' \land \eqref{eq:condB2}, \eqref{eq:condIncomp} \text{ hold}) + \sqrt{p_0}, \]
	and hence
	\[ p_1^2 \leq \Prob_{u,v,u', v'}(\mathcal{E}''' \land \eqref{eq:condomega}, \eqref{eq:condIncomp} \text{ hold}) + \sqrt{p_0}, \]
	where
	\begin{align*}
		\mathcal{E}''' &:= \left\{  \left| u_{\pi} w_{\pi}' + w_{\pi} v_{\pi} + z_0 \right| \leq 2 \sqrt{2} \frac{st}{t_0} \max\{ \|w \|, \|w'\| \} \right\}; 
	\end{align*}
	here, we used the fact that on the event \eqref{eq:condB2}, the event \eqref{eq:condomega} holds.  
	
	\subsubsection{Step 5: Applying the anti-concentration bounds}
	Recall that $w, w'$ depend only on $u_{\pi^c}, v_{\pi^c}, u_{\pi^c}', v_{\pi^c}'$.  In addition, $u_\pi$ and $v_{\pi}$ are independent of these random vectors.  Let us fix a realization of the random vectors $u_{\pi^c}, v_{\pi^c}, u_{\pi^c}', v_{\pi^c}'$ which satisfy \eqref{eq:condomega} and \eqref{eq:condIncomp}.  This completely determines $w$ and $w'$; moreover, $z_0$ is also completely determined.  Therefore, we conclude that
	\[ p_1^2 \leq \sup_{ \substack{w, w' \text{ satisfy } \eqref{eq:condomega}, \eqref{eq:condIncomp}\\z_0 \in \mathbb{C}} } \Prob_{u_{\pi}, v_{\pi}} \left(\left| u_{\pi} w_{\pi}' + w_{\pi} v_{\pi} + z_0 \right| \leq 2 \sqrt{2} \frac{st}{t_0} \max\{ \|w \|, \|w'\| \} \right) + \sqrt{p_0}. \]
	In order to bound this first term on the right-hand side, we will apply the anti-concentration bound given in Lemma \ref{lemma:lsv:anticor}.  Without loss of generality, let us assume that $\max\{ \|w \|, \|w'\| \} = \|w\|$.  Then dividing through by $\|w\|$, we find that
	\[ p_1^2 \leq \sup_{ \substack{w, w' \text{ satisfy } \eqref{eq:condomega}, \eqref{eq:condIncomp}\\z_0 \in \mathbb{C}} } \Prob_{u_{\pi}, v_{\pi}} \left(\left| u_{\pi} \frac{w_{\pi}'}{\|w\|} + \frac{w_{\pi}}{\|w\|} v_{\pi} + z_0 \right| \leq 2 \sqrt{2} \frac{st}{t_0}  \right) + \sqrt{p_0}, \]
	and 
	\[ \frac{ \left\| {w'} \right\| }{\|w\|} \leq 1. \]
	In view of Lemma \ref{lemma:lsv:anticor} (where we recall that $|\pi| \geq n(1-\delta/8)$ due to \eqref{eq:condpic}), we conclude that
	\[ p_1^2 \leq C''' \frac{1}{\sqrt{1 - \rho}} \log(C'''ns) \left( \frac{s^2t}{t_0} + \frac{1}{\sqrt{n}} \right) + \sqrt{p_0} \]
	for some constant $C''' > 0$.  
	
	\subsubsection{Step 6: Completing the proof}
	Combining the bounds from the previous steps, we obtain
	\[ \Prob_{B, u, v, A_{kk}}( \mathcal{E} \land \mathcal{B}_A) \leq p_1 + \sqrt{p_0}. \]
	We now proceed to simplify the expression to obtain \eqref{eq:suffdistpf}.  We still have the freedom to chose $t_0 > 0$; let us take $t_0 := \sqrt{t}$.  In addition, we may assume that the expression
	\begin{equation} \label{eq:errorbndform}
		C''' \log(C''' ns) \left( {s^2\sqrt{t}} + \frac{1}{\sqrt{n}} \right) 
	\end{equation}
	is less than one as the bound is trivial otherwise.  In particular, this implies that $s \leq \exp(\sqrt{n})$.  Among others, this means that the error term $C'' s \exp(-c'' n)$ can be absorbed into terms of the form \eqref{eq:errorbndform} (by increasing the constant $C'''$ if necessary).  After some simplification, the bound for $p_1$ obtained in the previous step (with the substitution $t_0 := \sqrt{t}$) yields
	\[ \Prob_{B, u, v, A_{kk}}( \mathcal{E} \land \mathcal{B}_A) \leq C \left( \frac{ \log(Csn) }{\sqrt{1-\rho}} \left( s^2 \sqrt{t} + \frac{1}{\sqrt{n}} \right) \right)^{1/4} \]
	for some constant $C > 0$.  This completes the proof of \eqref{eq:suffdistpf}, and hence the proof of Lemma \ref{lemma:lsv:dist} is complete.

	\subsection{Proof of Theorem \ref{thm:lsv}} \label{sec:proof:lsv}
	
	In this subsection, we complete the the proof of Theorem \ref{thm:lsv}.  Indeed, for any $s \geq 1$ and any $0 < t \leq 1$, we have
	\begin{align}
		\Prob \left( s_n(A) \leq \frac{t}{\sqrt{n}}, s_1(A) \leq s \right) &\leq \Prob \left( \min_{x \in \Sp^{n-1}} \|Ax \| \leq \frac{t}{\sqrt{n}}, s_1(A) \leq s \right) \nonumber \\
		&\leq \Prob \left( \min_{x \in \Incomp(\delta, \tau)} \|A x \| \leq \frac{ t}{ \sqrt{n} }, s_1(A) \leq s \right) \label{eq:finalsingbnd} \\
		&\qquad + \Prob \left( \min_{x \in \Comp(\delta, \tau)} \|Ax \| \leq \frac{1}{\sqrt{n}}, s_1(A) \leq s \right) \nonumber
	\end{align}
	due to our decomposition of the unit sphere into compressible and incompressible vectors.  It remains to bound each of the terms on the right-hand side.  
	
	For the incompressible vectors, we combine Lemmas \ref{lemma:lsv:avg} and \ref{lemma:lsv:dist} to find that, for any $t > 0$, 
	\begin{align*}
		\Prob \left( \min_{x \in \Incomp(\delta, \tau)} \|A x \| \leq \frac{ t \tau}{ \sqrt{n} }, s_1(A) \leq s \right) &\leq \frac{2}{\delta n} \sum_{k=1}^n \Prob( \dist(C_k, H_k) \leq t, s_1(A) \leq s ) \\
		&\leq \frac{2C}{\delta} \left( \frac{ \log(Csn) }{\sqrt{1-\rho}} \left( s^2 \sqrt{t} + \frac{1}{\sqrt{n}} \right) \right)^{1/4} 
	\end{align*}
	for some constant $C > 0$.  Recalling the definitions of $\delta$ and $\tau$ (\ref{eq:taudef}), we conclude that
	\begin{equation} \label{eq:Incompfinal}
		\Prob \left( \min_{x \in \Incomp(\delta, \tau)} \|A x \| \leq \frac{ t}{ \sqrt{n} }, s_1(A) \leq s \right) \leq C' \left( \frac{ \log(C'sn) }{\sqrt{1-\rho}} \left( \sqrt{s^5 t} + \frac{1}{\sqrt{n}} \right) \right)^{1/4} 
	\end{equation}
	for some constant $C' > 0$.  
	For compressible vectors, Lemma \ref{lemma:lsv:comp} implies the existence of constants $C'', c'' > 0$ such that
	\begin{equation} \label{eq:Compfinal}
		\Prob \left( \min_{x \in \Comp(\delta, \tau)} \|Ax \| \leq \frac{1}{\sqrt{n}}, s_1(A) \leq s \right) \leq C'' \exp(-c'' n). 
	\end{equation}
	
	Combining \eqref{eq:Incompfinal} and \eqref{eq:Compfinal} with \eqref{eq:finalsingbnd}, we conclude that, for any $s \geq 1$ and any $0 < t \leq 1$, 
	\begin{align*} 
		\Prob \left( s_n(A) \leq \frac{t}{\sqrt{n}}, s_1(A) \leq s \right) &\leq C' \left( \frac{ \log(C'sn) }{\sqrt{1-\rho}} \left( \sqrt{s^5 t} + \frac{1}{\sqrt{n}} \right) \right)^{1/4} + C'' \exp(-c'' n) \\
		&\leq C''' \left( \frac{ \log(C'''sn) }{\sqrt{1-\rho}} \left( \sqrt{s^5 t} + \frac{1}{\sqrt{n}} \right) \right)^{1/4} 
	\end{align*}
	for some constant $C''' > 0$, where the second inequality follows from the fact that the first error term dominates the second for all $n$ sufficiently large.  The proof of Theorem \ref{thm:lsv} is complete.

	\section{Singular values of $A_n$ and uniform integrability}\label{sect:SingVal}
	
	\subsection{Tightness}
	
	We begin with a bound on the largest singular values of $A_n - z$.  
	
	\begin{lemma}\label{Tightness}
		If Condition \textbf{C1} holds, there exists $r>0, C>0$ such that the following hold. \begin{itemize}
			\item For all $z\in\CC$, there exists $C_z>0$ such that almost surely 
			$$ \limsup_{n\rightarrow\infty}\int_{0}^{\infty} t^r d\nu_{A_n-z}(t)<C_z \text{ and thus } (\nu_{A_n-z})_{n\geq 1} \text{ is tight.}$$
			
			\item Almost surely
			$$\limsup_{n\rightarrow\infty}\int_\CC |w|^r d\mu_{A_n}(w)<C \text{ and thus } (\mu_{A_n})_{n\geq 1} \text{ is tight.}.$$
		\end{itemize}
	\end{lemma}
	
	\begin{proof}
		We follow the approach of Lemma 3.1 in \cite{HeavyIId}. The tightness follows from the moment bound and Markov's inequality. The moment bound on $\mu_{A_n}$ follows from the bound on $\nu_{A_n}$ and Weyl's inequality, Lemma \ref{Weyl}. One has from Lemma \ref{Basics}, $s_k(A_n-z)\leq s_k(A_n)+|z|$ for every $1\leq k\leq n$ and thus using the fact for any  $x,y\geq0, (x+y)^r\leq 2^r(x^r+y^r)$ we can assume $z=0$. We aim to work with matrices with independent entries, and thus decompose $A_n=U_n+L_n$ where $L_n$ is strictly lower triangular, and $U_n$ is upper triangular. Note for all $0\leq k\leq n-1$, we have by Lemma \ref{Basics}
		$$s_{1+k}(A_n)\leq s_{1+\lfloor k/2\rfloor}(U_n)+s_{1+\lceil k/2\rceil}(L_n).$$
		We now restrict $r$ such that $0<r\leq 2$. Thus 
		$$\int_{0}^{\infty} t^r d\nu_{A_n}(t)\leq 8\left[\int_{0}^{\infty} t^r d\nu_{U_n}(t)+\int_{0}^{\infty} t^r d\nu_{L_n}(t) \right].$$
		We show only  
		$$ \limsup_{n\rightarrow\infty}\int_{0}^{\infty} t^r d\nu_{U_n}(t)<\infty\qquad\qquad\text{a.s.}$$
		as the proof that 
		$$ \limsup_{n\rightarrow\infty}\int_{0}^{\infty} t^r d\nu_{L_n}(t)<\infty\qquad\qquad\text{a.s.}$$ follows in the exact same way.
		
		By the Schatten bound, Lemma \ref{Schattent},
		$$\int_{0}^{\infty} t^r d\nu_{U_n}(t)\leq Z_n:=\frac{1}{n}\sum_{i=1}^{n}Y_{n,i}$$
		where 
		$$Y_{n,i}=\left(\sum_{j=i}^{n}a_n^{-2}|X_{ij}|^2 \right)^{r/2}.$$
		For every $1\leq i< j\leq n$ we let $X_{ji}'$ be a copy of $\xi_1$, independent of all $X_{ij}$ and all other $X_{ji}'$. In addition let $$Y_{n,i}':=\left(\sum_{j=i}^{n}a_n^{-2}|X_{ij}|^2 + \sum_{j=1}^{i-1}a_n^{-2}|X_{ij}'|^2 \right)^{r/2}$$ then 
		$$Z_n\leq\frac{1}{n}\sum_{i=1}^n Y_{n,i}',$$
		since $Y_{n,i}\leq Y_{n,i}'$. The proof then follows exactly as in Lemma 3.1 of \cite{HeavyIId}.
	\end{proof}
	
	\subsection{Distance from a row and a vector space.} Throughout the rest of this section we assume the atom variables of $X_n$ satisfy Condition \textbf{C2}. The proof of Proposition 3.3 from \cite{HeavyIId} can be adapted in a straight forward way to get Proposition \ref{RowSubspace1} below. We give a brief explanation of the changes to the proof needed for entries that are independent but not necessarily identically distributed.   
	
	\begin{proposition}\label{RowSubspace1}
		Let $0<\gamma<1/2$, and $R$ be the $i$-{th} row of $a_n(A_n-z)$ with the $i$-th entry set to zero. There exists $\delta>0$ depending on $\alpha,\gamma$ such that for all $d$-dimensional subspaces $W$ of $\CC^n$ with $n-d\geq n^{1-\gamma}$, one has 
		$$\PP\left(\dist(R,W)\leq n^{(1-2\gamma)/\alpha} \right)\leq e^{-n^\delta}.$$
	\end{proposition}
	
	\begin{proof}
		Assume $R$ is the $i$-th row of $a_n(A_n-z)$ with the $i$-th entry set to zero. If $X^{(i)}$ is the $i$-th row of $X_n$ with the $i$-th entry set to zero, we have
		$$\dist(R,W)\geq\dist(X^{(i)},W_1)$$
		where $W_1=\Span(W,e_i)$. Note the entries of $X^{(i)}$ are independent, but have two potentially different distributions, in contrast with \cite{HeavyIId} where the entries are independent and identically distributed. However, Lemma \ref{TruncatedMoments} can be applied to either distribution. Under Condition \textbf{C2} the slowly varying function, $L(t)$, in (\ref{eq:indivassump}) is bounded and $L(t)\rightarrow c>0$ as $t\rightarrow\infty$ for both entries. To adapt the proof of Proposition 3.3 in \cite{HeavyIId} apply Lemma \ref{TruncatedMoments} to the entries of $X^{(i)}$ to get uniform bounds on the truncated moments without assuming the entries are identically distributed. 
	\end{proof}
	
	We now give some results for stable random variables which will be helpful. For $0<\beta<1$, let $Z=Z^{(\beta)}$ denote the one-sided positive $\beta$-stable distribution such that for all $s\geq 0$,\begin{equation}\label{eq:Zdef}
		\EE\exp(-sZ)=\exp(-s^\beta).
	\end{equation}
	
	Recall for $y,m>0$,
	$$y^{-m}=\Gamma(m)^{-1}\int_{0}^\infty x^{m-1}e^{-xy}dx.$$
	Thus for all $m>0$,\begin{equation}\label{eq:stablemoments}
		\EE[Z^{-m}]=\Gamma(m)^{-1}\int_{0}^{\infty}x^{m-1}e^{-x^\beta}dx,
	\end{equation} 
	and if $Z_1,\dots,Z_n$ are i.i.d.\ copies of $Z$ and $w_1,w_2,\dots,w_n$ are non-negative real numbers then\begin{equation}\label{eq:betastable}
		\sum_{i=1}^{n}w_iZ_i\overset{d}{=}\left(\sum_{i=1}^{n}w_i^\beta \right)^{1/\beta}Z_1.
	\end{equation}

	\begin{lemma}[Lemma 3.5 in \cite{HeavyIId}] \label{StocDomination}
		Let $0<\alpha<2$ and $Y$ be a random variable such that $t^{\alpha}\PP(|Y|\geq t)\rightarrow c$ as $t\rightarrow\infty$ for some $c>0$. Then there exists $\ee>0$ and $p\in(0,1)$ such that the random variable $|Y|^2$ dominates stochastically the random variable $\ee DZ$, where $\PP(D=1)=1-\PP(D=0)=p$ is a Bernoulli random variable, $Z=Z^{(\alpha/2)}$ and $D$ and $Z$ are independent.
	\end{lemma}

	\begin{lemma}\label{CouplingLemma}
		Let $X^{(i)}$ be the $i$-th row of the matrix $X_n$, with the $i$-th entry set to zero. Let $w_j\in[0,1]$ be numbers such that $w(n):=\sum_{j=1}^{n}w_j\geq n^{1/2+\ee}$ for some $\ee>0$. Let $Z=Z^{(\beta)}$ with $\beta=\alpha/2$. Then there exists $\delta>0$ and a coupling of $X^{(i)}$ and $Z$ such that 
		$$\PP\left(\sum_{j=1}^{n}w_j|X_{ij}|^2\leq\delta w(n)^{1/\beta}Z  \right)\leq Ce^{-cn^\delta}$$
		for constants $C,c>0$.
	\end{lemma}
	
	\begin{proof}
		Let $D=(D_j)_{j=1}^{i}$ and $D'=(D_j)_{j=i+1}^n$ be two independent vectors of i.i.d.\ Bernoulli random variables given by Lemma \ref{StocDomination} for $Y=X_{21}$ with parameter $p$ and $Y=X_{12}$ with parameter $p'$ respectively.  We know from Lemma \ref{StocDomination} there exists $\ee'>0$, such that for independent random variables $Z_j$ satisfying (\ref{eq:Zdef}) such that for every $j$, $w_j|X_{ij}|^2$ stochastically dominates $\ee'w_j D_j Z_j$. Then there exists a coupling (see Lemma 2.12 in \cite{hofstad_2016}) such that  
		$$\PP\left(\sum_{j=1}^{n}w_j|X_{ij}|^2\geq\ee'\sum_{j=1}^{n}w_j D_j Z_j  \right)=1.$$
		Let $\mathbf{a}=(a_1,\dots,a_n)\in\{0,1\}^n$, and $A_\mathbf{a}$ be the event $D_j=a_j$ for all $j$. Then define the random variable $Z$ pointwise on $A_\mathbf{a}$, $\mathbf{a}\neq \mathbf{0}$
		$$Z(\omega):=\frac{\sum_{j=i}^nw_ja_jZ_j(\omega)}{\left(\sum_{j=1}^nw_j^\beta a_j\right)^{1/\beta}},$$
		for $\omega\in A_\mathbf{a}$ and $Z(\omega)=Z_1(\omega)$ on $A_\mathbf{0}$. From (\ref{eq:betastable}), we see $Z$ satisfies (\ref{eq:Zdef}) and the distribution of $Z$ does not depend on $D_1,\dots, D_n$. Thus it is sufficient to show there exists $\ee''>0$ such that 
		$$\PP\left(\sum_{j=1}^{n}w_j^\beta D_j\leq \ee''w(n) \right)\leq Ce^{-cn^{\ee''}}.$$
		Note $w_j^{\beta}\geq w_j$, and thus $\EE\sum_{j=1}^{n}w_j^\beta D_j\geq \min(p,p')w(n)$. Therefore for $0<\ee''<\min(p,p')$,\begin{align*}
			&\PP\left(\sum_{j=1}^{n}w_j^\beta D_j\leq\ee''w(n)\right)\\
			&\qquad\qquad\leq\PP\left(\big|\sum_{j=1}^{n}(w_j^\beta D_j-\EE w_j^\beta D_j) \big|\geq (\min(p,p')-\ee'')w(n) \right)\\
			&\qquad\qquad\leq 2e^{-\frac{1}{2}(\min(p,p')-\ee'')^2w(n)^2/n},
		\end{align*} where the last bound follows from Hoeffding's inequality.	
	\end{proof}

	We now give another bound on the distance between a row of $A_n$ and a deterministic subspace. 
	
	\begin{proposition}\label{Rowsubspace2}
		Take $0<\gamma<\alpha/4$. Let $R$ be the $i$-th row of $a_n(A_n-z)$ with the $i$-th entry set to zero. There exists an event $E$ such that for any $d$-dimensional subspace $W$ of $\CC^n$ with $n-d\geq n^{1-\gamma}$, we have for sufficiently large $n$
		$$\EE[\dist^{-2}(R,W);E]\leq c(n-d)^{-2/\alpha}\quad \text{and}\quad \PP(E^c)\leq cn^{-\frac{1}{2}+\gamma(\frac{2}{\alpha}-\frac{1}{2})},$$
		where $c>0$ is an absolute constant which does not depend on the choice of row.
	\end{proposition}

	\begin{proof}
		We follow the approach of the proof of Proposition 3.7 in \cite{HeavyIId}. The only difference with Proposition 3.7 in \cite{HeavyIId} is the entries here are independent but not necessarily identically distributed. Assume that $R$ is the $i$-th row of $a_n(A_n-z)$ with the $i$-th entry set to zero. Note\begin{equation}\label{eq:Distreduction}
			\dist(R,W)\geq\dist(X^{(i)},W_1), 
		\end{equation}
		where $W_1=\Span(W,e_i)$, $e_i$ is the $i$-th basis vector, and $X^{(i)}=(X_{ij})_{1\leq j\leq n}$. Though the $i$-th entry of $X^{(i)}$ is not necessarily zero, inequality \eqref{eq:Distreduction} still holds since the subspace spanned by $e_i$ is contained in $W_1$. Let $\mathcal{J}$ denote the set of indices $j$ such that $|X_{ij}|\leq a_n$. It is a straight forward application of the Chernoff bound to show there exists $\delta>0$ such that
		$$\PP(|\mathcal{J}|<n-\sqrt{n})\leq e^{-n^\delta}.$$ 
		We begin by showing for any set $J\subset\{1,\dots,n\}$ such that $|J|\geq n-\sqrt{n}$,
		$$\EE[\dist^{-2}(R,W);E_J|\mathcal{J}=J]\leq c(n-d)^{-2/\alpha},$$
		for some event $E_J$ satisfying $\PP(E_J^c|\mathcal{J}=J)\leq cn^{-\frac{1}{2}+\gamma(\frac{2}{\alpha}-\frac{1}{2})}$. Without loss of generality assume $J:=\{1,\dots, n' \}$ with $n'\geq n-\sqrt{n}$. Let $\pi_J$ be the orthogonal projection onto the first $n'$ canonical basis vectors. Let $W_2=\pi_J(W_1)$, set 
		$$W'=\Span\left(W_2,\EE(\pi_J(X^{(i)})|\mathcal{J}=J) \right).$$
		Note $d-\sqrt{n}\leq \dim(W')\leq\dim(W_1)+1\leq d+2$. Define 
		$$Y=\pi_J(X^{(i)})-\EE(\pi_J(X^{(i)})|\mathcal{J}=J).$$
		One has 
		$$\dist(R,W)\geq\dist(Y,W').$$
		$Y$ is a mean zero vector under $\PP(\cdot|\mathcal{J}=J)$. Note $W'\subseteq \pi_J(\CC^n)$ and $Y\in\pi_J(\CC^n)$, so we will work with both as objects in only $\pi_J(\CC^n)\simeq \CC^{n'}$ and not the larger vector space $\CC^n$. Let $P$ be the orthogonal projection matrix onto $(W')^{\perp}$ in $\CC^{n'}$. Note $\tr P=n'-\dim(W')$ satisfies for sufficiently large $n$ \begin{equation}\label{eq:Tracebound}
			2(n-d)\geq\tr P\geq \frac{1}{2}(n-d).
		\end{equation}

		By construction \begin{equation}\label{eq:Traceequation}
			\EE(\dist^2(Y,W')|\mathcal{J}=J )=\EE\left[\sum_{j,k=1}^{n'}Y_jP_{jk}\bar{Y}_k |\mathcal{J}=J\right]=\sum_{j=1}^{n'}P_{jj}\EE[|Y_j|^2|\mathcal{J}=J].
		\end{equation}

		Let $S:=\sum_{j=1}^{n'}P_{jj}|Y_j|^2$. Before beginning note by Lemma \ref{TruncatedMoments} there exists $C>0$ such that $\EE[|Y_j|^2|\mathcal{J}=J]\leq Ca_n^2/n$ for $1\leq j\leq n'$. Thus\footnote{We believe there is a small typo in the bound of $\EE(|\dist^2(Y,W')-S |^2\big|\mathcal{J}=J)$ in the proof of Proposition 3.7 in \cite{HeavyIId}} \begin{align*}
			\EE(|\dist^2(Y,W')-S |^2\big|\mathcal{J}=J)&=\EE\left(\left|\sum_{k\neq j} Y_jP_{jk}\bar{Y}_k\right|^2\big|\mathcal{J}=J \right)\\
			&\leq \frac{2Ca_n^4}{n^2}\sum_{j,k}|P_{jk}|^2\\
			&=\frac{2Ca_n^4}{n^2}\|P\|_2^2\\
			&=\frac{2Ca_n^4}{n^2}\tr(P^*P)\\
			&=\frac{2Ca_n^4}{n^2}\tr(P).
		\end{align*}
		Thus \begin{equation}\label{eq:disttodiag}
			\EE(|\dist^2(Y,W')-S |^2\big|\mathcal{J}=J)=O\left(a_n^4\frac{n-d}{n^2} \right).
		\end{equation}
		
		Let $Z$ be as in Lemma \ref{CouplingLemma}. Set $w_j=P_{jj}$ and for $\ee>0$, consider the event 
		$$\Gamma_J:=\left\{\sum_{j=1}^{n'}w_j|X_{ij}|^2\geq\ee(n-d)^{1/\beta}Z \right\},$$
		where $\beta=\alpha/2$.	From Lemma \ref{CouplingLemma} there exists a coupling of $X_{i1},\dots,X_{i,n'}$ and $Z$ such that\begin{equation}\label{eq:GammaC}
			\PP(\Gamma_J^c)\leq Ce^{-cn^\delta},
		\end{equation} 
		for some $\delta>0$ and some choice of $\ee>0$. Since $(a-b)^2\geq a^2/2-b^2$ for $a,b\in\RR$ we have $S\geq \frac{1}{2}S_a-S_b$ where 
		$$S_a:=\sum_{j=1}^{n'}w_j|X_{ij}|^2,$$
		and 
		$$S_b:=\sum_{j=1}^{n'}w_j\EE[|X_{ij}|\big|X_{ij}\leq a_n]^2.$$
		From Lemma \ref{TruncatedMoments} and (\ref{eq:Tracebound}) one has 
		$$S_b\leq h^{(\alpha)}(n,d)$$
		where $h^{(\alpha)}(n,d)=\Theta( (n-d)a_n^2/n^2)$ if $\alpha\in(0,1]$ and $h^{(\alpha)}(n,d)=\Theta( (n-d))$ if $\alpha\in(1,2)$. Let $G_J^1$ be the event that $S_a\geq 3S_b$. There exists some $c_0$ such that 
		$$\PP((G_J^1)^c\cap\Gamma_J|\mathcal{J}=J)\leq\PP(Z\leq c_0(n-d)^{-1/\beta}h^{(\alpha)}(n,d)|\mathcal{J}=J).$$ 
		From the assumption $n-d\geq n^{1-\gamma}$ with $0<\gamma<\alpha/4$ we have $(n-d)^{-1/\beta}h^{(\alpha)}(n,d)\leq Cn^{-\ee_0}$ for some $C,\ee_0>0$. From here, using the bound (\ref{eq:stablemoments}) on the negative second moment of $Z$ , it is straightforward to show that for every $p>0$ there exists a constant $\kappa_p$ such that \begin{equation}\label{eq:G1C}
			\PP((G_J^1)^c\cap\Gamma_J|\mathcal{J}=J)\leq\kappa_pn^{-p}.
		\end{equation}
		Set $\tilde{\Gamma}_J=G_J^1\cap\Gamma_J$. On $\tilde{\Gamma}_J$, $S\geq \frac{1}{6}S_a\geq\frac{\ee}{6}(n-d)^{2/\alpha}Z$, and therefore 
		$$\EE[S^{-2};\tilde{\Gamma}_J|\mathcal{J}=J]\leq c_1(n-d)^{-4/\alpha}\EE[Z^{-2}].$$
		and thus using again the negative second moment bound on $Z$,\begin{equation}\label{eq:negsecondmomentdiag}
			\EE[S^{-2};\tilde{\Gamma}_J|\mathcal{J}=J]=O\left((n-d)^{-4/\alpha}\right).
		\end{equation} 
		Let $G_J^2$ be the event $\{\dist^2(Y,W')\geq S/2 \}$. Note using Markov's inequality and the Cauchy-Schwarz inequality leads to \begin{align}\label{eq:distdiag}
			\PP\left(\dist^{2}(Y,W')\leq S/2;\tilde{\Gamma}_J|\mathcal{J}=J \right)&\leq\PP\left(\frac{|\dist^2(Y,W')-S|}{S}\geq 1/2;\tilde{\Gamma}_J|\mathcal{J}=J \right)\nonumber \\
			&\leq2\EE\left[\frac{|\dist^2(Y,W')-S|}{S};\tilde{\Gamma}_J|\mathcal{J}=J \right]\\
			&\leq2\sqrt{\EE\left[|\dist^2(Y,W')-S|^2|\mathcal{J}=J\right]\EE[S^{-2};\tilde{\Gamma}_J|\mathcal{J}=J]}\nonumber.
		\end{align}
		Then by (\ref{eq:disttodiag}), (\ref{eq:negsecondmomentdiag}), and (\ref{eq:distdiag})\begin{equation}\label{eq:G2C}
			\PP\left((G_J^2)^c\cap\tilde{\Gamma}_J|\mathcal{J}=J \right)=O\left(a_n^2n^{-1}(n-d)^{\frac{1}{2}-\frac{2}{\alpha}} \right).
		\end{equation} 
		By the Cauchy-Schwarz inequality 
		$$\EE\left[\dist^{-2}(X,W);G_J^2\cap\tilde{\Gamma}_J|\mathcal{J}=J \right]\leq2\EE\left[S^{-1};\tilde{\Gamma}_J|\mathcal{J}=J\right]=O\left((n-d)^{-2/\alpha}\right).$$
		To conclude take $E_J=G_J^2\cap\tilde{\Gamma}_J$. Then 
		$$\PP(E_J^c|\mathcal{J}=J)\leq\PP(\Gamma_J^c|\mathcal{J}=J)+\PP((G_J^1)^c\cap\Gamma_J|\mathcal{J}=J)+\PP\left((G_J^2)^c\cap\tilde{\Gamma}_J|\mathcal{J}=J \right).$$
		It is then straightforward to show using (\ref{eq:GammaC}), (\ref{eq:G1C}), and (\ref{eq:G2C}) 
		$$\PP(E_J^c|\mathcal{J}=J)=O\left(n^{-\frac{1}{2}+\gamma(\frac{2}{\alpha}-\frac{1}{2})}\right).$$
		Take $E=\bigcup_{J\in\mathbf{J}}E_J\cap\{\mathcal{J}=J\}$ where $\mathbf{J}=\{J\subseteq[n]:|J|\geq n-\sqrt{n} \}$ to complete the proof.	
	\end{proof}
	
	%-----------------------------------------------------------------

	\subsection{Application of Theorem \ref{thm:lsv}}\label{Leastsingapp} We now show that Theorem \ref{thm:lsv} can be applied to $A_n$ to bound $s_n(A_n-z)$ from below with high probability. The difficulty is connecting the hypothesis on the spectral measure in Condition \textbf{C2} to the correlation of truncated random variables in the statement of Theorem \ref{thm:lsv}.

	\begin{theorem}\label{AnSatSSVC}
		For all $z\in\CC$, there exists $C,r>0$ such that
		$$\PP\left(s_n(A_n-z)\leq Cn^{-r}\right)= o(1).$$
	\end{theorem}
	
	It is worth noting $o(1)$ can be improved to $n^{-\ee}$ for some sufficiently small $\ee>0$. The proof of Theorem \ref{AnSatSSVC} will be an application of Theorem \ref{thm:lsv}. First we need a bound on the operator norm, and hence the largest singular value, of $A_n-z$.  
	
	\begin{lemma}\label{Largest Singular Value}
		For every $z\in\CC$, there exists $C>0$ depending on the distribution of the entries and $z$ such that for any $k>1/\alpha$ and $n$ sufficiently large,
		$$\Prob(\|X_n-a_nz\|\geq n^k)\leq \frac{C}{n^{((k-1)\alpha/2)-2}}.$$
	\end{lemma}
	
	\begin{proof}
		Since
		$$\|X_n-a_nz\|\leq \|X_n\|+a_n|z|\leq \|X_n\|+C_zn^{1/\alpha},$$
		it is sufficient to bound $\PP\left(\|X_n\|\geq cn^k\right)$, for some $0<c<1$ and $n$ sufficiently large. Note 
		
		$$\|X_n\|^2\leq n^2\max_{1\leq i,j\leq n}\{|X_{ij}|^2\}.$$
		Thus, by the union bound, \begin{align*}
			\PP\left(\|X_n\|\geq cn^k \right)&=\PP\left(\|X_n\|^2\geq c^2n^{2k} \right)\\
			&\leq\PP\left(\max_{1\leq i,j\leq n}\{|X_{ij}|^2\}\geq c^2n^{2(k-1)} \right)\\
			&\leq\PP\left(\bigcup_{1\leq i,j\leq n}\{|X_{ij}|^2\geq c^2n^{2(k-1)}\} \right)\\
			&\leq n^2\max_{i,j=1,2}\left\{\PP\left(|X_{ij}|^2\geq c^2n^{2(k-1)} \right)\right\}\\
			&\leq n^2\frac{C'\EE|X_{12}|^{\alpha/2}}{n^{2(k-1)\alpha/4}},
		\end{align*} where the last inequality follows from Markov's inequality and (\ref{eq:momentbound}).
	\end{proof}

	\begin{proof}[Proof of Theorem \ref{AnSatSSVC}]
		Clearly $X_n$ satisfies conditions (i) and (ii) of Theorem \ref{thm:lsv}, so it only remains to check condition (iii) before we can apply Theorem \ref{thm:lsv}. Let $\mathcal{E}_{ij}^n:=\{|X_{ij}|< \log(n)a_n,|X_{ji}|< \log(n)a_n  \}$. Since $\{ (X_{ij}, X_{ji}) : 1 \leq i < j \leq n\}$ is a collection i.i.d.\ random tuples we focus on showing for some $n$, $\mathcal{E}_{12}^n$ satisfies the desired conditions. From the tail bounds on $X_{12},$ and $X_{21}$, 
		$$\PP(\mathcal{E}^n_{12})\geq 1-\frac{C}{\log(n)^{\alpha}n}.$$
		
		$\var(X_{12}|\mathcal{E}^n_{12})=0$ if and only if $X_{12}$ is constant on $\mathcal{E}^n_{12}$, and hence on any subset of $\mathcal{E}^n_{12}$. Thus if $\var(X_{12}|\mathcal{E}^{n_0}_{12})>0$ for some $n_0$, then $\var(X_{12}|\mathcal{E}^n_{12})>0$ for all $n\geq n_0$. Since $X_{12}$ is non-constant, there must be some $n$ sufficiently large such that $X_{12}$ is non-constant on $\mathcal{E}_{12}^n$. Thus $\var(X_{12}|\mathcal{E}^n_{12})>0$ for all $n$ sufficiently large. The same argument follows for $\var(X_{21}|\mathcal{E}^n_{12})$. 
		
		Now assume for all $n$ sufficiently large,
		$$|\Corr\left(X_{12}|\mathcal{E}_{12}^n, X_{21}|\mathcal{E}_{12}^n \right)|=1.$$
		
		Then there exists $\theta_n\in[0,2\pi)$ such that on $\mathcal{E}_{12}^n$, $$X_{12}-\EE[X_{12}|\mathcal{E}^n_{12}]=e^{i\theta_n}\sqrt{\frac{\var(X_{12}|\mathcal{E}^n_{12})}{\var(X_{21}|\mathcal{E}^n_{12})}} \left[X_{21}-\EE[X_{21}|\mathcal{E}^n_{12}]\right].
		$$ For $\alpha<1$, by Lemma \ref{TruncatedMoments} $$\frac{\EE[|X_{12}|\indicator{{|X_{12}|\leq t}}]}{\left(\frac{\alpha}{1-\alpha}L(t)t^{1-\alpha}\right)}\rightarrow1,$$
		as $t\rightarrow\infty$, for some slowly varying function $L$. We assumed the atom variables satisfy Condition \textbf{C2} (ii), specifically $a_n\sim cn^{1/\alpha}$ and thus $L(t)\rightarrow c'$ for some constant $c'>0$ as $t\rightarrow\infty$. For $\alpha\geq 1$, $\EE[|X_{12}|\indicator{{|X_{12}|\leq t}}]$ is dominated by $t^r$ for any $r>0$. Thus on $\mathcal{E}_{12}^n$ 
		$$X_{12}=R_nX_{21}+C_n$$
		where $R_n$ is a sequence of complex numbers and \begin{equation}\label{eq:boundedconstant}
			\lim\limits_{n\rightarrow\infty}\frac{C_n}{a_n}=0,
		\end{equation} since $\sqrt{\frac{\var(X_{12}|\mathcal{E}^n_{12})}{\var(X_{21}|\mathcal{E}^n_{12})}}$ is slowly varying by Lemma \ref{TruncatedMoments}.
		If $R_n$ has a bounded subsequence $R_{n_k}$, we shall take the corresponding sequences $a_{n_k}$, and $n_k$ and for simplicity denote all three by $R_n,a_n,$ and $n$. If not, then we note on $\mathcal{E}_{12}^n$ 
		$$X_{21}=(R_n)^{-1}X_{12}+C_n'$$
		where $(R_n)^{-1}$ is bounded, and (\ref{eq:boundedconstant}) still holds for $C_n'$. Thus we will assume $R_n$ is bounded and we take a subsequence which converges to $R$.
		Let $r>0$ and $B$ be a Borel subset of the unit sphere in $\CC^2$ such that $\theta_d(\partial B)=0$. Note \begin{align}\label{eq:EandEc}
			\theta_d(B)m_\alpha([r,\infty)=&\lim\limits_{n\rightarrow\infty}n\PP\left(\frac{(X_{12},X_{21})}{\|(X_{12},X_{21}) \|}\in B ,\|(X_{12},X_{21}) \|\geq ra_n \right)=&\nonumber\\
			&\lim\limits_{n\rightarrow\infty}n\PP\left(\frac{(X_{12},X_{21})}{\|(X_{12},X_{21}) \|}\in B,\|(X_{12},X_{21}) \|\geq ra_n, \mathcal{E}^n_{12} \right)\\
			&\quad\quad+\lim\limits_{n\rightarrow\infty}n\PP\left(\frac{(X_{12},X_{21})}{\|(X_{12},X_{21}) \|}\in B,\|(X_{12},X_{21}) \|\geq ra_n, (\mathcal{E}^n_{12})^c \right),\nonumber	
		\end{align}
		and \begin{align}\label{eq:Ectozero}
			&\lim\limits_{n\rightarrow\infty}n\PP\left(\frac{(X_{12},X_{21})}{\|(X_{12},X_{21}) \|}\in B,\|(X_{12},X_{21}) \|\geq ra_n, (\mathcal{E}^n_{12})^c \right)\nonumber\\
			&\quad\leq\lim\limits_{n\rightarrow\infty}n\PP((\mathcal{E}^n_{12})^c)\nonumber\\
			&\quad=\lim\limits_{n\rightarrow\infty}n\PP(|X_{12}|\geq \log(n)a_n,\text{or}|X_{21}|\geq \log(n)a_n)\\
			&\quad\leq\lim\limits_{n\rightarrow\infty}Cn(\log(n)a_n)^{-\alpha}\nonumber\\
			&\quad=\lim\limits_{n\rightarrow\infty} C'n(\log(n)n^{1/\alpha})^{-\alpha}\nonumber\\
			&\quad=0\nonumber.
		\end{align}
		Thus we will consider only\begin{equation} 
			\lim\limits_{n\rightarrow\infty}n\PP\left(\frac{(X_{12},X_{21})}{\|(X_{12},X_{21}) \|}\in B,\|(X_{12},X_{21}) \|\geq ra_n, \mathcal{E}^n_{12} \right).
		\end{equation}

		Define the set $A$ as 
		$$A=\{(z,w)\in\CC^2:z=Rw, |z|^2+|w|^2=1 \},$$
		where $R$ is the limit of $R_n$. Let $(z_1,z_2)\in\CC^2$ be such that $(z_1,z_2)\notin A$ and $|z_1|^2+|z_2|^2=1$. Let $O$ be a small open neighborhood of $(z_1,z_2)$ in $\CC^2$ such that $A\cap \bar{O}=\emptyset$. We will consider the limit\begin{equation}\label{eq:OLimit}
			\lim\limits_{n\rightarrow\infty}n\PP\left(\frac{(X_{12},X_{21})}{\|(X_{12},X_{21}) \|}\in O,\|(X_{12},X_{21}) \|\geq ra_n, \mathcal{E}^n_{12} \right).
		\end{equation} Before we deal with this limit note that on $\mathcal{E}^n_{12}$ \begin{equation}
			\frac{(X_{12},X_{21})}{\|(X_{12},X_{21}) \|}=\frac{(RX_{21},X_{21})}{\|(X_{12},X_{21}) \|}+\frac{((R_n-R)X_{21},0)}{\|(X_{12},X_{21}) \|}+\frac{(C_n,0)}{\|(X_{12},X_{21}) \|},
		\end{equation}\begin{equation}\label{eq:Revtriangleonnorm}
			\left|\|(RX_{21},X_{21})\|-\|((R_n-R)X_{21},0)\|-\|(C_n,0)\| \right|\leq\|(X_{12},X_{21})\|,
		\end{equation} and\begin{equation}\label{eq:Triangleonnorm}
			\|(X_{12},X_{21})\|\leq\|(RX_{21},X_{21})\|+\|((R_n-R)X_{21},0)\|+\|(C_n,0)\|.
		\end{equation} We aim to show the unit vector in (\ref{eq:OLimit}) is approaching the bad set $A$, which will lead to a contradiction of Condition \textbf{C2}. To this end note by factoring out $\|(RX_{21},X_{21})\|$ from (\ref{eq:Revtriangleonnorm}) and (\ref{eq:Triangleonnorm}), we get
		$$\|(X_{12},X_{21})\|\geq\|(RX_{21},X_{21})\|\left|1-\frac{\|((R_n-R)X_{21},0)\|+|C_n|}{\|(RX_{21},X_{21})\|} \right|,$$
		and 
		$$\|(X_{12},X_{21})\|\leq\|(RX_{21},X_{21})\|\left[1+\frac{\|((R_n-R)X_{21},0)\|+|C_n|}{\|(RX_{21},X_{21})\|} \right].$$
		Since $|C_n|=o(a_n)$, it follows from (\ref{eq:Triangleonnorm}) that if $\|(X_{12},X_{21})\|\geq ra_n$, then $\|(RX_{21},X_{21})\|\geq ca_n$ for some $c>0$. It then follows that on $\{\|(X_{12},X_{21})\|\geq ra_n,\mathcal{E}^n_{12} \}$
		$$\|(X_{12},X_{21})\|\geq\|(RX_{21},X_{21})\|\left[1-o(1)\right],$$
		and 
		$$\|(X_{12},X_{21})\|\leq\|(RX_{21},X_{21})\|\left[1+o(1)\right].$$
		Thus \begin{align}\label{eq:C2contra}
			&\lim\limits_{n\rightarrow\infty}n\PP\left(\frac{(X_{12},X_{21})}{\|(X_{12},X_{21}) \|}\in O,\|(X_{12},X_{21}) \|\geq ra_n, \mathcal{E}^n_{12} \right)\nonumber\\
			&\quad=\lim\limits_{n\rightarrow\infty}n\PP\left(\frac{(RX_{21},X_{21})}{\|(RX_{21},X_{21}) \|}+o(1) \in O,\|(X_{12},X_{21}) \|\geq ra_n, \mathcal{E}^n_{12} \right)\\
			&\quad=0,\nonumber
		\end{align} since $\frac{(RX_{21},X_{21})}{\|(RX_{21},X_{21}) \|}+o(1)$ is a small perturbation of a vector in $A$ which is disjoint from $O$. From (\ref{eq:EandEc}), (\ref{eq:Ectozero}), and (\ref{eq:C2contra}) we see
		$$\lim\limits_{n\rightarrow\infty}n\PP\left(\frac{(X_{12},X_{21})}{\|(X_{12},X_{21}) \|}\in O,\|(X_{12},X_{21}) \|\geq ra_n\right)=0.$$
		Thus for arbitrary $(z_1,z_2)\in\CC^2$, $(z_1,z_2)\notin\supp(\theta_d)$ and $\supp(\theta_d)\subseteq A$, a contradiction of Condition \textbf{C2} (i). Therefore there exists arbitrarily large $n$ such that
		$$|\Corr\left(X_{12}|\mathcal{E}_{12}^n, X_{21}|\mathcal{E}_{12}^n \right)|<1.$$  
		
		From the above we see Theorem \ref{thm:lsv} may be applied to $A_n-z$, which combining with Lemma \ref{Largest Singular Value} completes the proof of Theorem \ref{AnSatSSVC}. 	
	\end{proof}

	\subsection{Uniform integrability} For $0<\delta<1$ we define  $K_\delta=[\delta,\delta^{-1}]$. We aim to show that $\log(\cdot)$ is uniformly integrable in probability with respect to $\{\nu_{A-z}\}_{n\geq 1}$, i.e. for all $\ee>0$ \begin{equation}\label{eq:UniformInt}
		\lim\limits_{\delta\rightarrow 0}\lim\limits_{n\rightarrow\infty}\PP\left(\int_{K_\delta^c}|\log(x)|d\nu_{A_n-z}(x)>\ee \right)= 0.
	\end{equation}

	From Lemma \ref{Tightness} there exists a constant $c_0>0$ such that with probability 1 
	$$\limsup_{n\rightarrow\infty}\int_{1}^{\infty}|\log(x)|d\nu_{A_n-z}(x)<c_0.$$
	From this, the part of the integral in (\ref{eq:UniformInt}) over $[\delta^{-1},\infty)$ is not a concern. Thus it suffices to prove that for every sequence $\delta_n$ converging to $0$,
	$$\frac{1}{n}\sum_{i=0}^{n-1}\mathbf{1}_{\{ s_{n-i}\leq \delta_n\}}\log s_{n-i}^{-2}$$ 
	converges to $0$ in probability. By Theorem \ref{AnSatSSVC} we may, with probability $1-o(1)$, lower bound $s_{n-i}$ by $cn^{-r}$ for all $i$. Take $0<\gamma<\alpha/4$ to be fixed later. Using the polynomial lower bound for $0\leq i\leq n^{1-\gamma}$, it follows that it is sufficient to prove that
	$$\frac{1}{n}\sum_{i=\lfloor n^{1-\gamma}\rfloor }^{n-1}\mathbf{1}_{\{ s_{n-i}\leq \delta_n\}}\log s_{n-i}^{-2}$$
	converges in probability to $0$. 
	
	We next aim to show there exists an event $F_n$ such that for some $\delta>0$ and $c>0$
	$$\PP(F_n^c)\leq c\exp(-n^\delta),$$
	and 
	$$\EE[s_{n-i}^{-2}|F_n]\leq c(\frac{n}{i})^{\frac{2}{\alpha}+1},$$
	for $\lfloor n^{1-\gamma}\rfloor\leq i\leq n-1$.
	First to see why this implies convergence in probability to zero, note 
	$$\PP(s_{n-i}\leq\delta_n)\leq\PP(F_n^c)+c\delta_n^2(\frac{n}{i})^{1+2/\alpha}.$$
	It follows there exists a sequence $\ee_n=\delta_n^{1/(\frac{2}{\alpha}+1)}$ tending to zero such that $\PP(s_{n-\lfloor n\ee_n\rfloor}\leq\delta_n)$ converges to $0$. Hence it is sufficient to prove, given $F_n$, 
	\begin{equation} \label{eq:fnto0}
		\frac{1}{n}\sum_{i=\lfloor n^{1-\gamma}\rfloor }^{\lfloor\ee_n n\rfloor}\log s_{n-i}^{-2}
	\end{equation} 
	converges to zero in probability. Using the negative second moment bound on $F_n$ and the concavity of $\log$ we have\begin{align*}
		\EE\left[\frac{1}{n}\sum_{i=\lfloor n^{1-\gamma}\rfloor }^{\lfloor\ee_n n\rfloor}\mathbf{1}_{\{ s_{n-i}\leq \delta_n\}}\log s_{n-i}^{-2}\big| F_n \right]&\leq\frac{1}{n}\sum_{i=\lfloor n^{1-\gamma}\rfloor }^{\lfloor\ee_n n\rfloor}\log\EE[s_{n-i}^{-2}|F_n]\\
		&\leq\frac{c_1}{n}\sum_{i=1}^{\lfloor\ee_n n\rfloor}\log(\frac{n}{i})\\
		%&=c_1(-\ee_n\log\ee_n+\ee_n+o(1))\\
		&\rightarrow 0
	\end{align*}
	where the last sum converges to zero by a Riemann Sum approximation.  Thus, by Markov's inequality we obtain convergence (given $F_n$) of \eqref{eq:fnto0} in probability to zero.
	
	Now we finish with the construction of such an event $F_n$. Let $B_n$ be the matrix formed by the first $n-\lfloor i/2\rfloor$ rows of $a_n(A_n-z)$. If $s_1'\geq s_2'\geq\cdots\geq s_{n-\lfloor i/2\rfloor}'$ are the singular values of $B_n$, then by Cauchy interlacing \begin{equation}\label{eq:Cauchyint}
		s_{n-i}\geq\frac{s_{n-i}'}{a_n}.
	\end{equation}
	By the Tao-Vu negative second moment lemma, Lemma \ref{Negative2ndMom}, we have \begin{equation}\label{eq:neq2ndmoment}
		s_1'^{-2}+\cdots s_{n-\lfloor i/2\rfloor}'^{-2}=\dist_1^{-2}+\cdots\dist_{n-\lfloor i/2\rfloor}^{-2}
	\end{equation}
	where $\dist_j$ is the distance from the $j$-th row of $B_n$ to the subspace spanned by the other rows. Using (\ref{eq:Cauchyint}) and (\ref{eq:neq2ndmoment})  we get
	$$\frac{i}{2}s_{n-i}^{-2}\leq a_n^2\sum_{j=i}^{n-\lfloor i/2\rfloor}\dist_j^{-2}.$$
	
	Now let $\dist_j'$ be the distance from the $j$-th row of $B_n$ with its $j$-th entry removed and the subspace spanned by the other rows minus their $j$-th entry. Since $\dist_j\geq\dist_j'$ we have  \begin{equation}\label{eq:ithsingbound}
		\frac{i}{2}s_{n-i}^{-2}\leq a_n^2\sum_{j=i}^{n-\lfloor i/2\rfloor}\dist_j'^{-2}
	\end{equation}

	Let $F_{n,i}$ be the event that for all $1\leq j\leq n-\lfloor i/2\rfloor$, $\dist_j'\geq (n-1)^{(1-2\gamma)/\alpha}$. Since the span of all but 1 row of $B_n$ is at most $n-i/2$, we can use Proposition \ref{RowSubspace1} to obtain, for sufficiently large $n$, 
	$$\PP(F_{n,i}^c)\leq\exp(-(n-1)^{\delta'}),$$
	for some $\delta'>0$ (after a union bound). Let $F_n=\bigcap_{i=0}^{n^{1-\gamma}} F_{n,i}$, then 
	$$\PP(F_n^c)\leq\exp(-(n-1)^{\delta}),$$
	for some $\delta>0$.
	
	We now aim to show the desired negative second moment bound on $F_n$. Recall from Proposition \ref{Rowsubspace2} there exists an event $E_j$ independent from the rows $i\neq j$ of $B_n$ without their $i$-th entry such that $\PP(E_j^c)\leq n^{-\frac{1}{2}+\gamma(\frac{2}{\alpha}-\frac{1}{2})}$ and for any subspace $W$ with dimension $d<n-n^{1-\gamma}$ one has
	$$\EE[\dist(R,W)^{-2};E_j]\leq c(n-d)^{-2/\alpha}$$
	where $R$ is the $j$-th row of $B_n$ with the $j$-th entry removed. Thus 
	$$\EE[\dist_j'^{-2};E_j]=O(i^{-2/\alpha}),$$
	for $i\leq j\leq n-\lfloor i/2\rfloor$.
	
	By taking the lower bound of $\dist_j'\geq (n-1)^{(1-2\gamma)/\alpha}$ on $E^c_j\cap F_n$, we get 
	$$\EE[\dist_j'^{-2};F_n]\leq c_2\left(i^{-2/\alpha}+n^{-\frac{1}{2}+\gamma(\frac{2}{\alpha}-\frac{1}{2})-2(1-2\gamma)/\alpha}\right).$$
	So for $\gamma$, not dependent on $i$, sufficiently small one has 
	$$\EE[\dist_j'^{-2};F_n]\leq c_3i^{-2/\alpha}.$$
	Thus by (\ref{eq:ithsingbound})
	$$\EE[s_{n-i}^{-2};F_n]\leq c_3a_n^2ni^{-(1+2/\alpha)}.$$
	The result then follows from the assumption $a_n\sim cn^{1/\alpha}$. 
	
	\subsection{Proof of Theorem \ref{EigenConv}}
	
	We have shown that almost surely $\nu_{A_n-zI_n}$ converges weakly to $\nu_{z,\alpha,\theta_d}$ and that $\log(\cdot)$ is uniformly integrable in probability with respect to $(\nu_{A_n-zI_n})_{n\geq 1}$. Thus by Lemma \ref{Girko}, $\mu_{A_n}$ converges weakly to some deterministic probability measure $\mu_{\alpha,\theta_d}$ in probability.

	\appendix

	\section{Additional lemmas}
	\subsection{Concentration}
	
	\begin{lemma}[McDiarmid's inequality, Theorem 3.1 in \cite{McDiarmid}]\label{McDiarmid}
		Let $X_1,X_2,\dots,X_n$ be independent random variables taking values in $R_1,R_2,\dots,R_n$ respectively. Let $F:R_1\times\cdots\times R_n\rightarrow\CC$ be a function with the property that for every $1\leq i\leq n$ there exists a $c_i>0$ such that 
		$$|F(x_1,x_2,\dots,x_i,\dots,x_n)-F(x_1,x_2,\dots,x_i',\dots,x_n)|\leq c_i$$
		for all $x_j\in R_j$, $x_i'\in R_i$ for $1\leq j\leq n$. Then for any $t>0$,
		$$\PP(|F(X)-\EE F(X)|\geq \sigma t)\leq C\exp(-ct^2)$$
		for absolute constants $C,c>0$ and $\sigma^2:=\sum_{i=1}^{n}c_i^2$.
	\end{lemma}

	For the following lemma we need a way of breaking up an elliptic random matrix $X$ into independent pieces.  For a matrix $M=(m_{ij})_{i,j=1}^n$ we let the $k$-wedge of $M$ be the $n\times n$ matrix $C_k$ with entries $c_{ij}=m_{ij}$ if either $i=k$ and $j\geq i$ or $j=k$ and $i\geq j$, with $c_{ij}=0$ otherwise. Note
	$$M_n=\sum_{k=1}^n C_k,$$
	and $\rank(C_k)\leq 2$ for all $1\leq k\leq n$. Recall the total variation norm of a function $f:\RR\rightarrow\RR$ is given by
	$$\|f\|_{TV}=\sup\sum_{k\in\ZZ}|f(x_{k+1})-f(x_k)|,$$
	where the supremum runs over all sequences $(x_k)_{k\in\ZZ}$ with $x_{k+1}\geq x_k$.    
	
	\begin{lemma}\label{TVwIndWed}
		Let $M$ be a $n\times n$ random matrix with independent wedges. Then for any $f:\RR\rightarrow\RR$ going to $0$ at $\pm\infty$ with $\|f\|_{TV}\leq 1$ and every $t\geq 0$,
		$$\PP\left(\left| \int fd\nu_M-\EE\int fd\nu_M\right|\geq t\right)\leq C\exp(-cnt^2),$$
		for absolute constants $C,c>0$.
	\end{lemma}
	
	\begin{proof}
		If $A,B\in \Mat_n(\CC)$ let $F_A$ and $F_B$ be the cumulative distribution functions of $\nu_A$ and $\nu_B$. By the Lidskii inequality (see Theorem 3.3.16 in \cite{HornTopics}) for singular values 
		$$\|F_A-F_B\|_\infty\leq\frac{\rank(A-B)}{n}.$$
		
		Assume $f$ is smooth. Integrating by parts we get\begin{equation}\label{eq:TVSing}
			\left|\int fd\nu_A-\int fd\nu_B \right|=\left|\int_\RR f'(t)(F_A(t)-F_B(t))dt \right|\leq\frac{\rank(A-B)}{n}\|f\|_{TV}.
		\end{equation} 
		Since $\left|\int fd\nu_A-\int fd\nu_B \right|$ depends on only finitely many points for any $f$, we can extend the previous inequality to any function of finite total variation. 
		
		Now fix $f:\RR\rightarrow\RR$ going to $0$ at $\pm\infty$ with $\|f\|_{TV}\leq 1$. Let $\mathcal{C}_k$ be the space of all $k$-wedges and $F_f:\mathcal{C}_1\times\mathcal{C}_2\times\dots\times\mathcal{C}_n$ be the function defined by 
		$$F_f(C_1,C_2,\dots,C_n)=\int fd\nu_A$$ where $A$ is the matrix with $k$-wedge $C_k$. By (\ref{eq:TVSing}) 
		$$|F_f(C_1,C_2,\dots,C_i,\dots,C_n)-F_f(C_1,C_2,\dots,C_i',\dots,C_n)|\leq \frac{2}{n},$$
		and thus by McDiarmid's inequality, Lemma \ref{McDiarmid}, and the definition of $F_f$,
		$$\PP\left(\left| \int fd\nu_M-\EE\int fd\nu_M\right|\geq t \right) \leq C\exp(-cnt^2).$$
	\end{proof}

	\subsection{Singular value estimates}
	
	\begin{lemma}[See \cite{HornTopics}, Chapter 3]\label{Basics}
		If $A$ and $B$ are $n\times n$ complex matrices then 
		$$s_1(AB)\leq s_1(A)s_1(B)\text{ and }s_1(A+B)\leq s_1(A)+s_1(B),$$
		$$\max_{1 \leq k \leq n}|s_k(A+B)-s_k(A)|\leq\|B\|,$$
		$$s_{i+j-1}(A+B)\leq s_i(A)+s_j(B)$$
		for $1\leq i,j\leq n$ and $i+j\leq n+1$. In addition, 
		$$\max_{1 \leq i \leq n}|s_i(A)-s_i(B)|\leq s_1(A-B).$$
	\end{lemma}
	
	%\begin{lemma}[Rudelson-Vershynin row bound \cite{RVLO}, see also Lemma B.2 in \cite{HeavyIId}]\label{RowBound}
	%	Let $A$ be a $n\times n$ complex matrix with rows $R_1,\dots, R_n$. Define the vector space $R_{-i}:=\Span\{R_j:j\neq i\}$. We then have 
	%	$$n^{-1/2}\min_{1 \leq i \leq n}\dist(R_i,R_{-i})\leq s_n(A)\leq\min_{1 \leq i \leq n}\dist(R_i,R_{-i}).$$
	%\end{lemma}
	
	\begin{lemma}[Tao-Vu Negative Second Moment, \cite{UnivCL} Lemma A.4]\label{Negative2ndMom}
		If $A$ is a full rank $n'\times n$ complex matrix with rows $R_1,\dots,R_{n'}$ and $R_{-i}:=\Span\{R_j:j\neq i\}$, then 
		$$\sum_{i=1}^{n'}s_i(A)^{-2}=\sum_{i =1}^{n'}\dist(R_i,R_{-i})^{-2}.$$
	\end{lemma}
	
	\begin{lemma}[Cauchy interlacing, see \cite{HornTopics}]\label{CauchyInterlacing}
		Let $A$ be an $n\times n$ complex matrix. If $B$ is an $n'\times n$ matrix obtained by deleting $n-n'$ rows from $A$, then for every $1\leq i\leq n'$,
		$$s_i(A)\geq s_i(B)\geq s_{i+n-n'}(A).$$
	\end{lemma}

	\begin{lemma}[See Remark 1 in \cite{Ky-FanHoffman}]\label{lem:KyFanHoffmanBound}
		Let $A$ be an $n\times n$ matrix and let $\Re(A)=(A+A^*)/2$ denote the real part of $A$. If $s_1\geq s_2\geq \dots\geq s_n$ and $\lambda_1\geq\lambda_2\geq\dots\geq\lambda_n$ denote the singular values of $A$ and eigenvalues of $\Re(A)$ respectively, then \begin{equation}
			\lambda_j\leq s_j
		\end{equation} for every $1\leq j\leq n$.
	\end{lemma}
	
	\begin{lemma}[Weyl's inequality, \cite{Weyl408}, see also Lemma B.5 in \cite{HeavyIId}]\label{Weyl}
		For every $n\times n$ complex matrix $A$ with eigenvalues ordered as $|\lambda_1(A)|\geq|\lambda_2(A)|\geq\dots\geq|\lambda_n(A)|$ one has
		$$\prod_{i=1}^{k}|\lambda_i(A)|\leq\prod_{i=1}^ks_i(A)\text{ and }\prod_{i=k}^{n}s_i(A)\leq\prod_{i=k}^n|\lambda_i(A)|$$
		for all $1\leq k\leq n$. Moreover for $r>0$
		$$\sum_{k=1}^{n}|\lambda_k(A)|^r\leq\sum_{k=1}^ns_k(A)^r.$$
		
	\end{lemma}
	
	\begin{lemma}[Schatten Bound, see proof of Theorem 3.32 in \cite{Matrixineq}]\label{Schattent}
		Let $A$ be an $n\times n$ complex matrix with rows $R_1,\dots,R_n$. Then for every $0<r\leq 2$,
		$$\sum_{i=1}^{n}s_k(A)^r\leq\sum_{i =1}^r\|R_k\|^r.$$
	\end{lemma}

	\subsection{Moments of stable distributions}
	\begin{lemma}[See \cite{HeavyIId} and \cite{feller-vol-2} Theorem VIII.9.2]\label{TruncatedMoments}
		Let $Z$ be a positive random variable such that for every $t>0$, 
		$$\PP(Z\geq t)=L(t)t^{-\alpha}$$
		for some slowly varying function $L$ and some $\alpha\in(0,2)$.
		Then for every $p>\alpha$,
		$$\lim\limits_{t\rightarrow\infty}\frac{\EE[Z^p\indicator{Z\leq t}]}{ c(p)L(t)t^{p-\alpha}}\rightarrow1,$$	
		where $c(p):=\alpha/(p-\alpha)$.
	\end{lemma}
	
	\subsection{Proof of Proposition \ref{prop:RDE}}\label{app:PropRDEProof} From Proposition \ref{prop:RecursiveEquation} one has \begin{equation*}
		\begin{pmatrix}
			a(z,\eta)&b(z,\eta)\\
			b'(z,\eta)&c(z,\eta)
		\end{pmatrix}=-\left(\begin{pmatrix}
			\eta&z\\
			\bar{z}&\eta
		\end{pmatrix}+\sum_{k=1}^\infty\begin{pmatrix}
		0&y_k^{(1)}\\
		\bar{y}_{k}^{(2)}&0\\
	\end{pmatrix}\tilde{R}(U)_{kk}\begin{pmatrix}
	0&y_k^{(2)}\\
	\bar{y}_{k}^{(1)}&0\\
\end{pmatrix} \right)^{-1}.
	\end{equation*} We will consider the point process $\{(y_k^{(1)},y_k^{(2)})\}_{k\geq1}$ in the polar form $\{(r_k,w_k)\}_{k\geq 1}$ where $\{r_k\}_{k\geq 1}$ is a Poisson point process with intensity measure $m_\alpha$ and $w_1,w_2,\dots$ is a collection of independent identically $\theta_d$ distributed random variables independent of $\{r_k\}_{k\geq 1}$. We will denote the coordinates of $w_k$ by $(w_k^{(1)},w_k^{(2)})$. In polar coordinates the recursive equation becomes \begin{equation}\label{eq:PPPRDE}
	\begin{pmatrix}
		a(z,\eta)&b(z,\eta)\\
		b'(z,\eta)&c(z,\eta)
	\end{pmatrix}=-\left(\begin{pmatrix}
		\eta&z\\
		\bar{z}&\eta
	\end{pmatrix}+\sum_{k=1}^\infty r_k^2\begin{pmatrix}
		c_k(z,\eta)|w^{(1)}_k|^2&b'_k(z,\eta)w^{(1)}_kw^{(2)}_k\\
		b_k(z,\eta)\bar{w}^{(1)}_k\bar{w}^{(2)}_k&a_k(z,\eta)|w^{(2)}_k|^2
	\end{pmatrix} \right)^{-1},
\end{equation} where \begin{equation*}
\begin{pmatrix}
	a_k(z,\eta)&b_k(z,\eta)\\
	b'_k(z,\eta)&c_k(z,\eta)
\end{pmatrix}:=\tilde{R}(U)_{kk}.
\end{equation*} To complete the proof that the matrix in \eqref{eq:limitresolvent} satisfies the distributional equation we need the following lemma, which is essentially Theorem 10.3 in \cite{Arous_2007}.
\begin{lemma}\label{lem:SeriesRepofStableVectors}
	Let $\{r_k\}_{k\geq 1}$ be a Poisson point process with intensity measure $m_\alpha$ for $\alpha\in(0,2)$ and $v_1,v_2,\dots$ be a collection of bounded independent and identically $\nu$ distributed random vectors for some probability measure $\nu$. Then \begin{equation}
		\sum_{k=1}^{\infty}r_k^2v_k\overset{d}{=}S,
	\end{equation} where $S$ is an $\frac{\alpha}{2}$-stable random vector with spectral measure $\Gamma_\nu$, where $\Gamma_\nu$ is the measure on the unit sphere obtained by the image of the measure $\frac{\Gamma(2-\alpha/2)\cos(\pi\alpha/4)}{1-\alpha/2}\|v\|^{\alpha/2}d\nu(v)$ under the map $v\mapsto \frac{v}{\|v\|}$.
\end{lemma}

\begin{proof}
	Let $(X_k)_{k\geq1}$ be a sequence of i.i.d.\ non-negative random variables such that \begin{equation}
		\PP(X_1\geq u)=\frac{L(u)}{u^{\alpha/2}},
	\end{equation} for some slowly varying function $L$, and define the normalizing constants\begin{equation}
	b_n:=\inf\left\{b:\PP(X\geq b)\leq\frac{1}{n} \right\}.
\end{equation} From Theorem 10.3 in \cite{Arous_2007}, $\frac{1}{b_n}\sum_{k=1}^nX_kv_k$ converges in distribution to $S$. It also follows from Proposition \ref{OrderStatConv} that \begin{equation}
\sum_{k=1}^{n}\delta_{X_kv_k/b_n}\Rightarrow\sum_{k=1}^{\infty}\delta_{r_k^{2}v_k},
\end{equation} as $n\rightarrow\infty$. It then follows from the almost sure uniform summability of the sequence $(X_k)_{k\geq 1}$ that \begin{equation}
\frac{1}{b_n}\sum_{k=1}^nX_kv_k\Rightarrow\sum_{k=1}^{\infty}r_k^2v_k,
\end{equation} as $n\rightarrow\infty$.
\end{proof} An application of Lemma \ref{lem:SeriesRepofStableVectors} to the series in \eqref{eq:PPPRDE} gives \eqref{eq:RDE}. 

\bibliography{heavy}
\bibliographystyle{abbrv}

\end{document}